\documentclass[a4paper, 11pt, reqno]{amsart}

\usepackage[inner=3cm, outer=3cm, top=3cm, bottom=3cm]{geometry}
\usepackage{color}
\usepackage{amsmath,amssymb,amsfonts,amsthm}
\usepackage{graphicx}
\usepackage{float}
\usepackage{enumerate}
  \usepackage[shortlabels]{enumitem}

 \newtheorem{theorem}{Theorem}[section]
 
 \newtheorem{lemma}[theorem]{Lemma}
 \newtheorem{proposition}[theorem]{Proposition}
 
 \theoremstyle{definition}
   \newtheorem{assumption}[theorem]{Assumption}

\theoremstyle{definition}
\newtheorem{definition}[theorem]{Definition}

\theoremstyle{remark}
\newtheorem{remark}[theorem]{Remark}

 \newcommand{\eps}{\varepsilon}

\newcommand{\norm}[1]{\Vert#1\Vert}
\newcommand{\normm}[1]{{\left\vert\kern-0.25ex\left\vert\kern-0.25ex\left\vert #1 
    \right\vert\kern-0.25ex\right\vert\kern-0.25ex\right\vert}}
\newcommand{\abs}[1]{\left\vert#1\right\vert}

\newcommand{\inner}[1]{\left(#1\right)}

\newcommand{\comi}[1]{\left<#1\right>}

\begin{document}

\title[Well-posedness    for the three-dimensional Prandtl equations]{Well-posedness in Gevrey function space  for the three-dimensional Prandtl  equations}

\author[W.-X. Li and T. Yang]{ Wei-Xi Li \and Tong Yang}

\date{}

\address[W.-X. Li]{School of Mathematics and Statistics, and Computational Science Hubei Key Laboratory,   Wuhan University,  430072 Wuhan, China
  }

\email{
wei-xi.li@whu.edu.cn}

\address[T.Yang]{
Department of Mathematics, City University of Hong Kong, Hong Kong
 \\ 
 $\&$ School of Mathematics and Statistics, Chongqing University, Chongqing, China }

\email{
matyang@cityu.edu.hk}

\begin{abstract}
In the paper, we study  the   three-dimensional Prandtl equations without any monotonicity condition on
the velocity field. We prove that when  one  tangential  component of the velocity field
  has a single curve of non-degenerate critical points with respect to the normal variable,  the system is  locally  well-posed in the  Gevrey function space with  Gevrey 
 index in  $]1, 2].$   The proof is based  on some new observation of cancellation 
 mechanism in the three space dimensional system in addition to those  in the two-dimensional setting
 obtained in \cite{awxy,GM,LY,MW}.  
\end{abstract}

\subjclass[2010]{35Q30, 35Q31}
\keywords{3D Prandtl boundary layer,  non-degenerate critical points, Gevrey class}

 \maketitle


\section{Introduction and main results}

The inviscid limit for the incompressible Navier-Stokes equations is 
one of the most fundamental problems
in fluid mechanics.  The justification remains a challenging problem   from the mathematical point of view in particular with physical boundary conditions because of the appearance of boundary layer. Under the  no-slip boundary condition, the behavior
 of the boundary layer can be  described by the Prandtl system  introduced by Prandtl  \cite{prandtl}  to study    
the behavior of the incompressible flow near a rigid 
wall at high Reynolds number.
Formally,  the  asymptotic limit of the Navier-Stokes equations is represented by the Euler equations  away from boundary and  by the Prandtl equations within the boundary layer. A mathematically rigorous justification  of the  vanishing viscosity limit basically remains unsolved up to now even though it has been achieved in some special settings, see for example \cite{CKV,GM-2, GMM, GN,GI,Ma, Samm}  and the references therein for the recent progress.

   The mathematical study on   the boundary layer has a very long history, however, so far  the   theory is only well developed in various function spaces  for the 2D Prandtl system.  In fact, the 2D Prandtl system can be reduced to a scalar nonlinear and nonlocal degenerate parabolic equation that has
   loss of  derivative in tangential variable. The degeneracy in the
viscosity dissipation coupled with  the loss of derivative
in the  nonlocal term is the main difficulty in
 the well-posedness theories.     To overcome the degeneracy,  it is natural, in the spirit of abstract  Cauchy-Kovalevskaya theorem,  to perform estimates  within the category of the analytic function space,  and in this context the well-posedness was obtained  in \cite{Samm} (see the earlier work \cite{asa}, and  also \cite{KV,cannone} for further generalization), even with the justification of
the vanishing viscosity limit and in 3D.  If the initial data have only
finite order of regularity, the well-posedness in Sobolev space was
first obtained by Oleinik  (cf. \cite{oleinik-3}) under the monotonicity  assumption in the normal direction, where the Crocco transformation was used to overcome the loss of the derivative.  Recently,  an approach based on
energy method was developed for the well-posedness in Sobolev space  under Oleinik's monotonicity assumption, see \cite{awxy,MW} where the key 
observation is some kind of cancellation property in the
convection term due to the monotonicity.  Furthermore,  the well-posedness results  in Gevrey space  were  achieved in the recent works \cite{DG,GM, LY} for the initial data  without analyticity or monotonicity,  where  some
further cancellation properties were observed near the non-degenerate critical points.   Different from  the analytic context,    the Gevrey space with index $\sigma>1$ contains compactly supported functions.  
  We also refer to \cite{lwx} for the smoothing effect in Gevrey space under the monotonicity assumption.   In the aforementioned works, only local-in-time well-posedness results are obtained. On the other hand,
the global weak solution was established by \cite{xin-zhang} with additional
assumption on pressure while
 the existence of global strong solutions still remains unsolved, although there are several works (see \cite{IV,xz,zhangzhang} for example)  about
 the lifespan of  solutions in different settings.  

On the other hand, in general boundary separation    happens  that implies
 the Prandtl equations can no longer be a suitable model for describing
the behavior of the flow near the boundary. In mathematics,
this is related to the fact that the Prandtl equations without the analyticity or monotonicity  are in general   ill-posed, cf. \cite{e-2, GV-D, GN1} and the references therein.  

Compared to the 2D case,  
 much less is known about the three-dimensional Prandtl equations.  As for the well-posedness theories, only partial results have been obtained
 in some specific settings,  cf.  \cite{Samm} in the analytic context and \cite{LWY1}  under some constraint on flow structure. 
In addition  to the difficulties for 2D,  another
major difficulty in 3D arises from
  the secondary flow.    As it will be seen in the
later analysis,  the cancellation properties observed in 2D case are not
enough to overcome the difficulties in the analysis. In addition, we need
to use some new cancellations in the 3D setting.
We will  explain this further  in Subsection \ref{subsec22}
about the new ideas and approach to be used.

 In this paper, we will study the well-posedness of 3D Prandtl system without  the analyticity or  monotonicity.  For this, let us
first mention the paper on the ill-posedness \cite{LWY2} which shows that
even for a perturbation of shear flow, without the structural
condition, the linearized Prandtl equations are ill-posed. In fact, the
ill-posedness estimate on the solution operator implies that the optimal 
Gevrey index for the well-posedness without any structural condition
is $2$. Hence, the result of this paper about the well-posedness in Gevrey
function space with index in $]1,2]$ complements the ill-posedness estimate
in \cite{LWY2}.  
   Without loss of generality,  we will consider the system in  a periodic domain in tangential direction, that is, in $\Omega=\mathbb T^2\times\mathbb R_+.$

  Denote by      $(u, v)$   the tangential component and  by $w$     the vertical component of the 
velocity field,  then the 3D Prandtl system in  $\Omega$  
 reads
\begin{equation}\label{prandtl+}
\left\{
\begin{aligned}
&\partial_t u + \inner{u \partial_x   + v\partial_y +w\partial_z}u -\partial_{z}^2u + \partial_x p
=0,\quad t>0,~ (x,y,z)\in\Omega, \\
&\partial_t v + \inner{u\partial_x   + v\partial_y +w\partial_z}v -\partial_{z}^2v+ \partial_y p
=0,\quad t>0,~ (x,y,z)\in\Omega, \\
&\partial_xu +\partial_yv +\partial_zw=0, \quad t>0,~ (x,y,z)\in\Omega, \\
&u|_{z=0} = v|_{z=0} = w|_{z=0} =0 , \quad  \lim_{z\to+\infty} (u,v) =\big(U(t,x,y), V(t,x,y)\big), \\
&u|_{t=0} =u_0, \quad v|_{t=0}=v_0, \quad   (x,y,z)\in\Omega,
\end{aligned}
\right.
\end{equation}
where   $(U(t,x,y), V(t,x,y))$ and $p(t,x,y)$ are the boundary traces  
of the tangential velocity field and pressure of the outer flow, 
satisfying Bernoulli's law
\begin{eqnarray*}
\left\{
\begin{aligned}
&\partial_t U + U\partial_x U+V\partial_y U +\partial_x p=0,\\
&\partial_t V + U\partial_x V+V\partial_y V +\partial_y p=0.
\end{aligned}
\right.
\end{eqnarray*}
Note   $p, U,V$ are given functions determined by the Euler flow,  and  \eqref{prandtl+} is a degenerate parabolic system losing one order
derivative in the tangential variable. 
We refer to \cite{mamoudi,oleinik-3,prandtl} for the background
and  mathematical presentation of  this fundamental system.

We will consider the Prandtl system \eqref{prandtl+} under the assumption  that  one component of the initial tangential velocity component, for example $u_0$ admits  a single curve of non-degenerate critical points.    The precise assumption on the initial data is
given as follows. 

\begin{assumption}\label{maas}  Let $\delta>2$ be a given number.  Suppose there exists a single curve $z=\gamma_0(x,y)$ in $\Omega$ with $0< \sup_{\mathbb T^2}\gamma_0<+\infty$ and  several constants  $C_0>0,$ $0<c_0<1 $ and $0<\epsilon_0<{1\over 4} \sup_{\mathbb T^2}\gamma_0,$  such that for any   $(x,y)\in\mathbb T^2$ the initial datum
   $  u_0 \in C^6(\Omega)$  satisfies  the following properties:
   
   \begin{eqnarray*}
   	\left\{
   	\begin{aligned}
   	&	\partial_z u_0(x,y,z)=0\ \textrm{iff}\ z=\gamma_0(x,y),\ \textrm{and} \ \partial_z^2 u_{0}(x,y,\gamma_0(x,y))\neq 0;\\
   	&		  c_0 \comi z^{-\delta}\leq \partial_z u_0(x,y,z) \leq  c_0^{-1} \comi z^{-\delta} \ \textrm{if}\ \abs{z-\gamma_0(x,y)}\geq \epsilon_0;\\
   	&\sum_{2\leq j\leq 6} \abs{\partial_z^j u_0(x,y,z)}\leq  C_0 \comi z^{-\delta-1}  \ \textrm{for any}\ z\geq 0,
   	\end{aligned}
\right.
   \end{eqnarray*}
where  $\comi{z}=\big(1+\abs z^2\big)^{1/2}.$
\end{assumption}

Next we introduce the Gevrey function space in the tangential variables $x,y.$  In
this paper, we use   $\ell,\kappa$ to denote  two fixed constants
satisfying 
\begin{equation}\label{ellalpha} 
	\kappa\geq 1, \quad \ell>3/2~~{\rm and}~~    \ell +1/2<\delta\leq \ell+1, 
\end{equation} 
where $\delta >2$ is  given  in Assumption \ref{maas}.

\begin{definition}[Gevrey space in tangential variables $x$ and $y$]
\label{defgev}  Let $U,V$ be the data given in \eqref{prandtl+}.  With each pair $(\rho,\sigma)$, $\rho>0$ and $\sigma\geq 1, $  a Banach space $X_{\rho,\sigma}$   consists of all  smooth  vector functions
 $(u,v)$  such that  $\norm{(u,v)}_{\rho,\sigma}<+\infty,$ where the Gevrey norm   $\norm{\cdot}_{\rho,\sigma}$ is defined below.     For each multi-index  $\alpha=(\alpha_1,\alpha_2)\in\mathbb Z_+^2$,  denote $\partial^\alpha=\partial_x^{\alpha_1}\partial_y^{\alpha_2}$ and 
 \begin{eqnarray*}
 	\psi=\partial_z (u-U)=\partial_z u,\quad \eta=\partial_z (v-V)=\partial_z v.
 \end{eqnarray*} Then the Gevrey norm is defined by 
\begin{eqnarray}\label{trinorm}
	\norm{(u,v)}_{\rho,\sigma} &=& \sup_{ \abs\alpha \geq 7} \frac{\rho^{\abs\alpha-6}}{[\inner{\abs\alpha-7}!]^{\sigma}} \Big( \norm{\comi z^{\ell-1}\partial^\alpha (u-U)}_{L^2}+\norm{\comi z^{\kappa}\partial^\alpha (v-V)}_{L^2} \Big)\nonumber  \\
	&&+\sup_{ \abs\alpha \leq 6}  \big( \norm{\comi z^{\ell-1}\partial^\alpha (u-U)}_{L^2} +\norm{\comi z^{\kappa}\partial^\alpha (v-V)}_{L^2} \big) \nonumber \\
	&& +\sup_{ \abs\alpha \geq 7} \frac{\rho^{\abs\alpha-6}}{[\inner{\abs\alpha-7}!]^{\sigma}}   \norm{\comi z^{\ell}\partial^\alpha \psi}_{L^2}+\sup_{ \abs\alpha \leq 6} \norm{\comi z^{\ell}\partial^\alpha \psi}_{L^2} \nonumber \\
	&&+ \sup_{  \abs\alpha \geq 7}   \frac{\rho^{\abs\alpha-5}}{[\inner{\abs\alpha-6}!]^{\sigma}}  \abs\alpha\norm{\comi z^{\kappa+2}\partial^\alpha \eta}_{L^2}+ \sup_{  \abs\alpha \leq 6}     \norm{\comi z^{\kappa+2}\partial^\alpha\eta}_{L^2}    \\
	&&+\sup_{\stackrel{1\leq j\leq 4}{\abs\alpha+j\geq 7}} \frac{\rho^{\abs\alpha+j-6}}{[\inner{\abs\alpha+j-7}!]^{\sigma}}\norm{\comi z^{\ell+1} \partial^\alpha \partial_z^j\psi}_{L^2}+\sup_{\stackrel{1\leq j\leq 4}{\abs\alpha+j\leq 6}}  \norm{\comi z^{\ell+1} \partial^\alpha \partial_z^j\psi}_{L^2}  \nonumber \\
	&+&\sup_{\stackrel{1\leq j\leq 4}{\abs\alpha+j\geq 7}} \frac{\rho^{\abs\alpha+j-5}}{[\inner{\abs\alpha+j-6}!]^{\sigma}} \abs\alpha \norm{\comi z^{\kappa+2} \partial^\alpha \partial_z^j\eta}_{L^2} +\sup_{\stackrel{ 1\leq j\leq 4}{\abs\alpha+j\leq 6}}   \norm{\comi z^{\kappa+2}\partial^\alpha  \partial_z^j \eta}_{L^2},\nonumber
\end{eqnarray}
where and throughout the paper we use $L^2$ instead of $L^2(\Omega)$ 
without confusion. Moreover, we define another Gevrey space  $Y_{\rho,\sigma}$  consist of smooth functions $F(x,y)$ such that $\normm{F}_{\rho,\sigma}<+\infty,$  where  
\begin{eqnarray*}
	\normm{F}_{\rho,\sigma}=\sup_{\abs\alpha\geq 0} \frac{\rho^{\abs\alpha}}{  (\abs\alpha !)^{\sigma}} \norm{\partial^\alpha F}_{L^2(\mathbb T^2)}.
\end{eqnarray*}
  \end{definition}
  
 \begin{remark} \label{remaniso}
 Note the factors in front of the $L^2$-norms of $\partial^\alpha \psi$ and  $\partial^\alpha \eta$    are anisotropic, likewise, for the mixed derivatives in the last two  lines of \eqref{trinorm}. 
 \end{remark}

We will look for the solutions to \eqref{prandtl+} in the Gevrey function space $X_{\rho,\sigma}$.   For this, the  initial data $(u_0,v_0)$  satisfy the following compatibility conditions
 \begin{equation}\label{comcon}
 \left\{
 \begin{aligned}
 &(u_0, v_0)|_{z=0}=(0,0),~~\lim_{z\rightarrow +\infty}(u_0,v_0)|=(U,V)~~{\rm and}~~(\partial_{z}\psi_0, \partial_{z}\eta_0)|_{z=0}=(\partial_x p,\partial_yp),\\
 &\partial_z^3\psi_0\big |_{z=0}=  \psi_0\inner{\partial_x \psi_0-\partial_y\eta_0} |_{z=0}+2 \eta_0 \partial_y\psi_0|_{z=0}+\partial_t\partial_xp,\\
 &\partial_z^3\eta_0\big |_{z=0}=  \eta_0\inner{\partial_y\eta_0-\partial_x\psi_0} |_{z=0}+2\psi_0\partial_x\eta_0 |_{z=0}+\partial_t\partial_yp,	
	\end{aligned}
	\right.
	\end{equation}
where $\psi_0=\partial_zu_0$ and  $\eta_0=\partial_zv_0.$ 

The main result of this paper can be stated as follows.

\begin{theorem}
	\label{maintheorem}
	Let $1<\sigma \leq 2,$ under the  compatibility conditions \eqref{comcon}, suppose $U,V, p\in Y_{2\rho_0,\sigma}$ and $(u_0,v_0)\in X_{2\rho_0,\sigma}$ for some $\rho_0>0.$ Moreover, suppose   $u_0$ satisfies Assumption \ref{maas}.   Then the Prandtl system \eqref{prandtl+} admits a unique solution $(u,v) \in L^\infty\big([0,T];~X_{\rho,\sigma}\big)$ for some $T>0$ and some $0<\rho<2\rho_0.$  
\end{theorem}


 \begin{remark}
Up to a coordinate transformation,  the assertion in the  above theorem still holds if we impose the existence of non-degenerate points on the tangential component
of the initial velocity field in an arbitrary  but fixed direction instead of $u_0.$
 \end{remark}

 \begin{remark}
 	The assumption that  $U,V, p\in Y_{2\rho_0,\sigma}$ is closely related to the up-to-boundary Gevrey regularity for the Euler equations,  see for example \cite{KV1} and the references therein.    
 \end{remark}
 
 \begin{remark}
 	     It remains unsolved about
 the well-posedness in the category of Sobolev space with finite
reguarity that has been well developed  in 2D. And a more important open problem is  the mathematical   justification on  the approximation of the solution to Navier-Stokes equation with no-slip boundary
condition  by the those of Euler and Prandtl equations.  We expect the present work may
shed some light on  the vanishing viscosity limit for Navier-Stokes equations
in 3D setting.  
  \end{remark}

To simplify the notations,   we will mainly focus  only on  the constant outer flow
  when $(U,V)  \equiv (0,0),$  and  the argument  can be extended, without
essential difficulty to  general  functions $U$ and $V$ as given
in the proof of Theorem 1.4 at the end of the paper.

 For the constant outer flow, 
 the Prandtl system \eqref{prandtl+}  can be written as   
\begin{equation}\label{prandtl}
\left\{
\begin{aligned}
&\partial_t u + \inner{u \partial_x   + v\partial_y +w\partial_z}u -\partial_{z}^2u 
=0,\quad t>0,~ (x,y,z)\in\Omega, \\
&\partial_t v + \inner{u\partial_x   + v\partial_y +w\partial_z}v -\partial_{z}^2v 
=0,\quad t>0,~ (x,y,z)\in\Omega,  \\
&(u,v)|_{z=0} =(0,0), \quad  \lim_{z\to+\infty} (u,v) =\big(0, 0\big), \\
&(u,v)|_{t=0} =(u_0, v_0), \quad   (x,y,z)\in\Omega,
\end{aligned}
\right.
\end{equation}
with 
\begin{eqnarray*}
	w(t,x,y,z)=- \int_0^z\partial_x u(t,x, y,\tilde z)\,d\tilde z- \int_0^z\partial_y v(t,x, y,\tilde z)\,d\tilde z.
\end{eqnarray*}

The existence and uniqueness of solutions to  system \eqref{prandtl} can be stated as follows. 
  
 \begin{theorem}
 \label{maithm1}	
 	Let $1<\sigma\leq  2.$ Suppose  the initial datum      $(u_0,v_0)$  belongs to $ X_{2\rho_0,\sigma}$ for some $\rho_0>0,$  and satisfies  the    compatibility condition   \eqref{comcon} with constant pressure.  
	 Then the system \eqref{prandtl} admits a unique solution $(u,v) \in L^\infty\big([0,T];~X_{\rho,\sigma}\big)$ for some $T>0$ and some $0<\rho<2\rho_0.$ 
 \end{theorem}

We will focus on proving  Theorem \ref{maithm1}  and show  at the end of the last section  about  how to extend  the argument to the case with
 general outer flow when $U,V\in Y_{2\rho_0,\sigma}.$    
  Furthermore, we will present in detail the proof of Theorem \ref{maithm1} for $\sigma\in [3/2,2]$. Note that the constraint  $\sigma\geq 3/2$ is not essential and indeed it is just a technical assumption  for clear presentation. 
 And we will explain  at the end of Section \ref{sec7} about
 how to modify the proof for the case when $1<\sigma<3/2$.

The rest of the paper is organized as follows.   In Section  \ref{sec2}, we first introduce the notations used in the paper, and then explain the main difficulties and the new ideas.
 We will prove  in Sections  \ref{sec4}-\ref{sec7} the a priori estimate given in Section \ref{sec2}.    The proof of   the well-posedness for the Prandtl system is given in the last section.

 \section{Notations and methodology}
 \label{sec2}
  This section  and Sections \ref{sec4}-\ref{sec7} are to derive a priori estimate for the Prandtl system \eqref{prandtl},  which is crucial for proving Theorem \ref{maithm1}.  In Subsection \ref{subsec21}, we introduce
 some  auxiliary functions to be estimated. In  Subsection \ref{subsec22},
 we  explain the difficulties   for the well-posedness of 3D Prandtl system and  then present the ideas  to overcome them.  The
a priori estimate is given in Subsection \ref{subsec23}    with its proof given
in Sections \ref{sec4}-\ref{sec7}.

Let $\delta, \ell$ and $\kappa$ be some given numbers satisfying \eqref{ellalpha},  and let $(u,v)$ be a solution to  the Prandtl system \eqref{prandtl} in $[0,T]\times\Omega,$  satisfying the properties stated as follows.  There exists a single curve $z=\gamma(t,x,y)$ of non-degenerate critical points, that is,   for $(t,x,y)\in[0,T]\times\mathbb T^2,$
\begin{eqnarray*}
	\partial_z u(t,x,y,z)=0\  \textrm{iff} \ z=\gamma(t,x,y,z),  \ \textrm{and}\  \abs{\partial_z^2 u(t,x,y,\gamma(t,x,y))}> 0. 
\end{eqnarray*}
To simply the argument we may assume without loss of generality that 
\begin{eqnarray*}
	\gamma\equiv 1,
\end{eqnarray*}
that is $1$ is the only non-degenerate critical point of $u,$
and the general case can be derived quite similarly with slight modification. 
Moreover, there exist three constants $0<c, \epsilon<1/4$  and $C>0$ such that  for any      $t\in[0, T]$  and  for any $(x,y)\in\mathbb T^2$   we have 
 \begin{eqnarray}\label{condi}
\left\{
\begin{aligned}
&\partial_z u(t,x,y,z)\equiv 0~ \textrm{iff} \ z=1,  \ \textrm{and}\  \abs{\partial_z^2 u(t,x,y,1)}> 0;\\
  &\abs{\partial_z^2 u(t,x,y,z)}\geq  c, ~\, {\rm if}~~\,\abs{z-1}\leq 2\epsilon; \\
  & c \comi z^{-\delta}\leq\abs{   \partial_z u(t,x,y,z) }\leq    c^{-1}\comi z^{-\delta},~~\,{\rm if}~~\,\abs{z-1}\geq \epsilon;\\
  &\sum_{j=1}^6\abs{ \partial_z^ju(t,x,y,z)}\leq   c^{-1}  \comi z^{-\delta-1} ~~\, \textrm {for  } ~z\geq 0, 
 \end{aligned}
 \right.
\end{eqnarray}
and  
 \begin{equation}\label{+condi1}
 \begin{aligned}
 &\sum_{\abs\alpha\leq 3}  \Big(\norm{\comi z^{\ell-1}\partial^\alpha u}_{L^\infty} +  \norm{\comi z^{\kappa}\partial^{\alpha} v}_{L^\infty} +  \norm{ \partial^{\alpha} w}_{L^\infty} \Big)\\
 &\qquad+\sum_{\abs\alpha\leq 4}  \norm{\comi z^{\ell}\partial^{\alpha}  \partial_zu}_{L_{x,y}^\infty(L_z^2)} +\sum_{\stackrel{ \abs\alpha+j\leq 4}{j\geq 1}}   \norm{\comi z^{\ell+1}\partial^{\alpha}\partial_z^j \partial_zu}_{L_{x,y}^\infty(L_z^2)}\\
 &\qquad\qquad\quad\qquad\qquad\quad\qquad+\sum_{  \abs\alpha+j\leq 4} \norm{\comi z^{\kappa+2}\partial^{\alpha}\partial_z^j \partial_z v}_{L_{x,y}^\infty(L_z^2)}\leq  C,
\end{aligned}
\end{equation}
where $L_{x,y}^\infty(L_z^2)=L^\infty\inner{\mathbb T^2; L^2(\mathbb R_+)}$ stands for the classical Sobolev space, so does  the Sobolev space $L_{x,y}^2(L_z^\infty).$

Let $0\leq \tau_1, \tau_2\leq 1$ be two  $C^\infty(\mathbb R)$ smooth functions  such that
 \begin{equation}\label{chi1}
 \tau_1\equiv 1 ~\,\textrm{on}\, ~\big\{\abs{z-1}> 3\epsilon/2 \big\} , \quad \tau_1\equiv 0 ~\,\textrm{on}\, ~ \big\{\abs{z-1}\leq \epsilon \big\},
\end{equation}
and  
 \begin{equation}\label{cutofffu}
    	\tau_2\equiv 1 ~\,\textrm{on}\, ~\big\{\abs{z-1}\leq 3\epsilon/2 \big\},\quad {\rm supp}\, \tau_2\subset\big\{\abs{z-1}\leq 2\epsilon \big\}.
\end{equation}
Observe 
\begin{equation}\label{suppch2}
\tau_1+\tau_2\geq 1,\quad	\tau_1'=\tau_1'\tau_2,\quad \tau_2'=\tau_2'\tau_1, ~~{\rm and } ~~\inner{1-\tau_2}=\inner{1-\tau_2}\tau_1,
\end{equation}
because $\tau_2\equiv 1 $ on  supp\,$\tau_1',$   $\tau_1\equiv 1 $ on  supp\,$\tau_2',$  and $\tau_1\equiv 1$  on  supp\,$(1-\tau_2).$ Here and  throughout the paper, $f'$ and $f''$ stand for the first and the
second order derivatives of  $f$.

 \subsection{Notations} \label{subsec21} From now on, we  write $\partial^\alpha$ instead of $\partial_{x}^{\alpha_1}\partial_{y}^{\alpha_2}$ for $\alpha\in\mathbb Z_+^2.$    Let  $(u,v)$ solve the system \eqref{prandtl}.   Define $\psi $ and $\eta$  by
 \begin{eqnarray}\label{ps}
   \psi=\partial_z u ~~\,\textrm{and} ~~\,\eta=\partial_z v. \end{eqnarray}
Applying $\partial_z$ to the equations for $u$ and $v$ in \eqref{prandtl},  we obtain the equations solved by  $\psi$ and $\eta,$ that is,   
\begin{equation}\label{equforpsi}
\left\{
\begin{aligned}
	&\partial_t \psi +\big(u\partial_x    +v\partial_y	 +w\partial_z \big)\psi-\partial _{z}^2\psi=g ,\\
	&\partial_t \eta +\big(u\partial_x    +v\partial_y	 +w\partial_z \big) \eta-\partial _{z}^2\eta=h  ,
	\end{aligned}
	\right.
\end{equation}
where $g ,h  $ are given by 
\begin{equation}\label{defR1}
\left\{
\begin{aligned}
	& g =(\partial_y v)\psi- ( \partial_yu)\eta,\\
	&h  = (\partial_x u)\eta-(\partial_xv) \psi.
	\end{aligned}
	\right.
\end{equation}
Moreover, denote 
\begin{equation}
\label{s}	
\xi=\partial_z\psi=\partial_z^2u,\quad \zeta=\partial_z\eta=\partial_z^2v,
\end{equation}
and  it follows from \eqref{equforpsi} that
 \begin{equation}\label{seco}
\left\{
\begin{aligned}
	&\partial_t \xi +\big(u\partial_x    +v\partial_y	 +w\partial_z \big)\xi-\partial _{z}^2\xi=\sum_{j=1}^3\theta_j,\\
	&\partial_t \zeta +\big(u\partial_x    +v\partial_y	 +w\partial_z \big) \zeta-\partial _{z}^2\zeta = \sum_{j=1}^3\mu_j,
	\end{aligned}
	\right.
\end{equation}
where 
\begin{equation}\label{sor}
\left\{
\begin{aligned}
	&\theta_1=2\xi \partial_yv-2\eta\partial_y\psi,\quad \theta_2=\psi \partial_y\eta-\zeta\partial_yu,\quad \theta_3=\xi\partial_xu-\psi\partial_x\psi,\\
	&\mu_1=2\zeta\partial_xu-2\psi\partial_x\eta,\quad \mu_2=\eta\partial_x\psi-\xi\partial_xv,\quad\mu_3=\zeta\partial_y v-\eta\partial_y\eta.
	\end{aligned}
	\right.
\end{equation}
For   each multi-index $\alpha=(\alpha_1,\alpha_2)\in\mathbb Z_+^2$, we   define   $g_{\alpha}$ and $h_\alpha$ by 
\begin{equation}\label{galpha}
	g_\alpha= \partial^\alpha g = \partial_x^{\alpha_1}\partial_y^{\alpha_2}g   , \quad \textrm{and}\quad h_\alpha= \partial^\alpha h  =\partial_x^{\alpha_1}\partial_y^{\alpha_2}h  .
\end{equation} 
Similarly, define 
\begin{equation}\label{themu}
	\vec \theta_\alpha=\inner{\partial^\alpha \theta_1,\cdots, \partial^\alpha \theta_3},\quad 	\vec \theta_\mu=\inner{\partial^\alpha \mu_1,\cdots, \partial^\alpha \mu_3}
\end{equation}
And for each $m\geq1$ define 
 \begin{equation}\label{Gaal}
 \left\{
 \begin{aligned}
 	&\Gamma_m= \psi\partial_x^m v-\eta\partial_x^m u,\quad \tilde\Gamma_m=\inner{\partial_y^m v}\psi-\inner{\partial_y^m u}\eta,\\
 	&H_m=\xi \partial_x^mv-\eta\partial_x^m\psi,\quad \tilde H_m=\xi \partial_y^mv-\eta\partial_y^m\psi,\\
 	&G_m=\xi\partial_x^m u-\psi\partial_x^m\psi,\quad \tilde G_m=\xi\partial_y^m u-\psi\partial_y^m\psi.
 	\end{aligned}
 	\right.
 \end{equation} 
 Let $\tau_1,\tau_2$ be given by \eqref{chi1} and \eqref{cutofffu}.  Motivated by \cite{awxy,MW}, 
when  $ \psi\neq 0$  in supp\,$\tau_1$  then we can define $f_{m}$ with $ m\geq 1,$      by 
 \begin{equation}\label{falpha}
 	f_{m}=\tau_1\partial_x^{m} \psi-\tau_1\frac{\xi }{ \psi}\partial_x^{m}  u=\tau_1\psi\partial_z\Big(\frac{\partial_x^m u}{\psi}\Big),
 \end{equation}
 recalling $\psi,\xi$ are defined in \eqref{ps} and \eqref{s}.  
Likewise, define $\tilde f_m$   by 
 \begin{equation*}
 	\tilde f_m=\tau_1\partial_y^{m} \psi-\tau_1\frac{\xi }{ \psi}\partial_y^{m} u=\tau_1 \psi \partial_z\Big(\frac{\partial_y^m u}{ \psi}\Big).
 \end{equation*}
On the other hand, when $\xi\neq 0$ in supp\ $\tau_2$ then we define
 \begin{equation}
 \label{t2}
 q_m=	\tau_2\partial_x^m\xi-\tau_2\frac{\partial_z\xi}{\xi}\partial_x^m \psi=\tau_2\xi\partial_z\Big(\frac{\partial_x^m \psi}{\xi}\Big),
 \end{equation}
 and similarly for $\tilde q_m.$
  It follows from \eqref{prandtl}  that 
\begin{equation*}
	\label{bonval}
	(\partial_z\psi,\partial_z\eta)|_{z=0}=(\partial_zf_m, \Gamma_m)|_{z=0}=(\partial_z\tilde f_m, \tilde \Gamma_m)|_{z=0}=(g_\alpha, h_\alpha)|_{z=0}=(0, 0). 
\end{equation*}
 
  Note the definitions of $f_m$
  and $\tilde f_m$ are 
  motivated by \cite{awxy,MW}, and meanwhile the type of  auxiliary functions $q_m, \tilde q_m$ are used by \cite{LY}.   As in \cite{LY} if we can control  the term 
  $\tau_2\partial_x^m\psi$ then
    the estimates on $\partial_x^m u$ and $\partial_x^m\psi$ can be derived from the weighted $L^2$-norm of $f_m$ and $q_m$ by   Hardy inequality and Poincar\'e inequality, likewise, for $\partial_y^m u$ and $\partial_y^m\psi.$  As a result, we have the upper bounds on $\partial^\alpha u$  and $\partial^\alpha \psi,$   by the inequality  
 \begin{equation}
 	\label{realp}
 	\forall~\alpha\in\mathbb Z_+^2,~\forall~ F\in H^\infty,\quad \norm{\partial^\alpha F}_{L^2(\mathbb T^2)}^2\leq \norm{\partial_x^{\abs\alpha}F}_{L^2(\mathbb T^2)}^2+ \norm{\partial_y^{\abs\alpha} F}_{L^2(\mathbb T^2)}^2.
 \end{equation}
In order to obtain the upper bound  of $\partial^\alpha \eta,$ we will apply the energy method to the equation \eqref{equforpsi} for $\eta.$   It remains to  estimate $\partial^\alpha v$ with $\abs\alpha=m,$ and this can be deduced from the estimation on   the auxilliary functions  $\Gamma_m$ and $\tilde\Gamma_m$ on the monotonicity part where
 $\psi\neq 0,$ and from $H_m, \tilde H_m$ on the concave part where $\xi\neq 0.$   Finally, we need to estimate
 $g_\alpha$ and $h_\alpha$ since they appear in the equations for $f_m$ and $\partial^\alpha \eta$   
with the degeneracy in tangential variables.   We will 
explain further in the next subsection about
the motivation for introducing the above auxilliary functions.

\subsection{Difficulties and  methodologies}\label{subsec22}
In this subsection,
we will  explain the main difficulties and the new ideas   introduced in this
paper.

 When   applying  $\partial^\alpha$ to the equations in \eqref{prandtl},  we  lose derivatives in the tangential variables $ x $ and $y$ in   the terms 
\begin{equation}\label{twoterms}
	\inner{\partial^\alpha w}\psi, \quad \inner{\partial^\alpha w} \eta.
\end{equation}
Similar to 2D case, part of these terms
can be   handled under the  the existence of single curve of non-degenerate critical points,     by  some kind of  cancellation properties,   cf. \cite{GM, LY}. 
In the 3D setting considered in this paper,
 for the   monotonicity part that 
$ \abs{\psi}>0,$ we can apply the same cancellation as in 2D case,  to the equations of $\partial^\alpha u$ and $\partial^\alpha  \psi.$  This indeed eliminates the first term  in \eqref{twoterms},  but meanwhile  a new term
\begin{eqnarray*}
	g_\alpha=\partial^\alpha g 
\end{eqnarray*} 
appears when applying $\partial^\alpha$ to the equation \eqref{equforpsi} for $\psi.$   Note that $g$
also has the loss of one order derivative in $y$ variable.
 This  prevents us to investigate the well-posedness in  Sobolev space with
finite order of regularity. 
 When performing estimates in the concave (or convex) part where $\abs\xi>0,$  the situation is quite different from the one in \cite{GM,LY}. Precisely,  in addition to the term 
\begin{eqnarray*}
	\inner{\tau_2' \int_0^z\partial_x^{m+1}u\ d\tilde z,~ \tau_2 \partial_x^{m+1}u}_{L^2\inner{\Omega}}
\end{eqnarray*}
that can be handled by some kind of crucial representation of $\partial_x^m u$ introduced by \cite{GM},  
 we are faced  two new terms caused by the appearance of the secondary component $v$ and the cutoff function $\tau_2:$
\begin{eqnarray}\label{conter}
	\inner{\tau_2' \int_0^z\partial_x^{m}\partial_y v\ d\tilde z,~ \tau_2 \partial_x^{m+1}u}_{L^2\inner{\Omega}}\ {\rm and}	\ \inner{\tau_2  \partial_x^{m}\partial_y v, ~ \tau_2 \partial_x^{m+1}u}_{L^2\inner{\Omega}},
\end{eqnarray}  
where the degeneracy occurs in the nonlocal term caused by $v$.

\medskip
 
\noindent\underline{\it Estimates on  $\partial^\alpha u$ and $\partial^\alpha \psi$}.  If we have the desired upper bounds for the terms in \eqref{conter} then we can make use of the same cancellation as in  2D case \cite{LY} to obtain a new equation for the auxilliary function $f_m, q_m$ defined by \eqref{falpha} and \eqref{t2}.  Even though  this can not
 avoid the degeneracy
 in the tangential variables because of $g_\alpha, \vec\theta_\alpha$.     
Our observation is that this kind of cancellation transfers  the degeneracy coming from the first non-local term in \eqref{twoterms} to a local term $g_\alpha.$ To work with    $g_\alpha$,  we have the advantage  to use another kind of cancellation. Precisely, multiplying the first equation in \eqref{equforpsi}     by $\partial_y v$, and the second equation   by  $\partial_y u$,  and then their subtraction yields  the equation for $g $. Based on this,
  we can  apply our approach  used in 2D \cite{LY} to perform energy estimate for $g_\alpha$ and similarly for $\vec\theta_\alpha$, and then for $f_m,$ in the context of Gevrey function space rather than  the analytic function space, if we ignore  at this  moment the degeneracy caused by the second term in \eqref{twoterms}.  Similar argument applies
also to $h_\alpha$ and $\tilde f_m, \tilde q_m.$  As a result, we can obtain the estimates as desired for $\partial^\alpha u$ and $  \partial^\alpha \psi$ by Hardy inequality and Poincar\'e inequality. 

\medskip
\noindent\underline{\it Estimates on   $\partial^\alpha v$ and $\partial^\alpha\eta$}. Now we turn to  the second tangential component of velocity field.   
Firstly, for $\eta,$  note that if the order of its derivatives in the energy  are   one order less than the ones of $u, v$ and $\psi,$  then we do not have
derivative loss for this term.  This is why in the definition of
 the   Gevrey norm, there is an anisotropic term for $\eta$ (see Remark \ref{remaniso}).   

To handle $v$ and its derivatives,  we will not apply the energy 
estimation directly because
    it  involves the second term in \eqref{twoterms}.     Instead,
 we observe that the estimate on  $\partial_x^m v$ can be derived through  $h_{m-1,0}=\partial_x^{m-1} h$ defined in \eqref{galpha}  in the   monotonicity part where  $ \abs\psi>0,$   and  through $\partial_x^{m-1}\mu_2$ defined in \eqref{sor} in the concave part of $\abs\xi>0$.    On the other hand, it is not easy to estimate the 
bounds of the lower order derivatives
 $\partial_x^j v, j<m,$  in the expression of $h_{m-1,0}.$    This is why we introduce the auxillary functions $\Gamma_m$ and $H_m,$ which contain only
 the leading term in the representations of $h_{m-1,0}$ and $\partial_x^{m-1}\mu_2.$  It is
then clear that we can  get the upper bound of $\partial_x^m v$ from $\Gamma_m$ in the  monotonicity part  and  from $H_m$ in concave part.    Furthermore, $\Gamma_m$  can be handled  by  a new cancellation as shown in  
the equation for $\Gamma_m,$  where  the terms in \eqref{twoterms}  
do not appear due to cancellation.  Similar argument applies
to $H_m, \tilde \Gamma_m$ and $\tilde H_m.$  And this leads to the desired estimate on $\partial^\alpha v.$    

 \medskip
\noindent\underline{\it Estimates on   the terms in \eqref{conter}}.   We will use the similar idea as above for treatment of $\partial_x^m v,$ together with the crucial representation of $\partial_x^m u$ introduced by \cite{GM}; see Subsections \ref{subsec61} and \ref{subsec62} for  detailed computation.

\subsection{A priori estimate}\label{subsec23}

We first introduce the Gevery function space and the equipped norm. 

 \begin{definition}  
\label{gevspace}	
Denote $\vec a=(u,v)$ with $(u,v)$ satisfying the Prandtl system  \eqref{prandtl} and the conditions \eqref{condi}-\eqref{+condi1}.  Let  $X_{\rho,\sigma}$ be the Gevrey function space given in Definition \ref{defgev},  equipped with the norm $\norm{\cdot}_{\rho,\sigma}$	 defined by \eqref{trinorm}.  Set
\begin{eqnarray*}
\begin{aligned}
	& \abs{\vec a}_{\rho,\sigma}  =  \norm{\vec a}_{\rho,\sigma} + \sup_{m\geq 7}   \frac{\rho^{m-6}}{[\inner{m-7}!]^{\sigma}}   \Big( \norm{\comi z^{\ell} f_{m}}_{L^2}+ \norm{q_{m}}_{L^2}+\norm{\tau_2\partial_x^m \xi}_{L^2}\Big)\\
	&\quad+ \sup_{m\geq 7}   \frac{\rho^{m-6}}{[\inner{m-7}!]^{\sigma}}   \Big( \norm{\comi z^{\ell} \tilde f_{m}}_{L^2}+\norm{  \tilde q_{m}}_{L^2}+\norm{ \tau_2\partial_y^m \xi}_{L^2}\Big)\\
	 &\quad+ \sup_{m\geq 7}   \frac{\rho^{m-6}}{[(m-7)!]^{\sigma}}   \Big(\norm{\comi z^{\kappa+\delta} \Gamma_{m}}_{L^2}+ \norm{G_{m}}_{L^2}+\norm{H_{m}}_{L^2} \Big)\\
	  &\quad+ \sup_{m\geq 7}   \frac{\rho^{m-6}}{[(m-7)!]^{\sigma}}   \Big(  \norm{\comi z^{\kappa+\delta} \tilde \Gamma_{m}}_{L^2}+\norm{\tilde G_{m}}_{L^2}+\norm{\tilde H_{m}}_{L^2}\Big)\\
	&\quad + \sup_{\abs\alpha\geq 7}   \frac{\rho^{\abs\alpha-5}}{[\inner{\abs\alpha-6}!]^{\sigma}} \abs\alpha\Big(\norm{\comi z^{\kappa+\delta}  g_{\alpha}}_{L^2}+ \norm{\comi z^{\kappa+\delta} h_{\alpha}}_{L^2} + \norm{\vec\theta_{\alpha}}_{L^2}+ \norm{  \vec\mu_{\alpha}}_{L^2} \Big) \\
	& \quad +\sup_{1\leq m \leq 6}   \Big( \norm{\comi z^{\ell} f_{m}}_{L^2}+ \norm{q_{m}}_{L^2}+\norm{ \tau_2\partial_x^m \xi}_{L^2}+\norm{\comi z^{\ell} \tilde f_{m}}_{L^2}+\norm{  \tilde q_{m}}_{L^2}+\norm{ \tau_2\partial_y^m \xi}_{L^2}\Big)\\
	&\quad + \sup_{1\leq m\leq 6}      \big(\norm{\comi z^{\kappa+\delta} \Gamma_{m}}_{L^2}+ \norm{G_{m}}_{L^2}+\norm{H_{m}}_{L^2}+ \norm{\comi z^{\kappa+\delta} \tilde \Gamma_{m}}_{L^2}+\norm{\tilde G_{m}}_{L^2}+\norm{\tilde H_{m}}_{L^2}\big)\\
	&\quad +\sup_{0\leq\abs\alpha \leq 6}    \Big(  \norm{ \comi z^{\kappa+\delta} g_{\alpha}}_{L^2}+ \norm{\comi z^{\kappa+\delta}h_{\alpha}}_{L^2}  + \norm{\vec\theta_{\alpha}}_{L^2}+ \norm{  \vec\mu_{\alpha}}_{L^2} \Big),
	 \end{aligned}
\end{eqnarray*}
where the functions $\tau_j,$ $g_\alpha, h_\alpha, f_m,$ etc., are defined in Subsection \ref{subsec21}. 
 Similarly, we define $\abs{\vec a_0}_{\rho,\sigma}$ for  $\vec a_0=(u_0,v_0),$  the initial datum in \eqref{prandtl}.  The $L^2$-norm of a vector $\vec p=(p_1,p_2,p_3)$ is defined by 
 \begin{eqnarray*}
 	\norm{\vec p}_{L^2}^2=\sum_{1\leq j\leq 3} \norm{p_j}_{L^2}^2.
 \end{eqnarray*}
\end{definition}

\begin{remark}\label{remeqi} It is clear that $\abs{\vec{a}}_{\rho,\sigma}\leq \abs{\vec a}_{\tilde\rho,\sigma}$ for any $\rho\leq\tilde\rho.$ Moreover, 
direct calculation gives 
\begin{equation*} 
\norm{\vec a}_{\rho,\sigma}\leq \abs{\vec a}_{\rho,\sigma}\leq C_{\rho,\rho^*} \big(\norm{\vec a}_{\rho*,\sigma}+\norm{\vec a}_{\rho*,\sigma}^2\big)
\end{equation*}
 for any   $\rho<\rho^*,$  with $C_{\rho, \rho^*}$ being
a constant depending only on the difference $\rho^*-\rho$.  
\end{remark}

The main result of this part can be stated as follows.  

\begin{theorem}[A priori estimate in Gevrey space]\label{apriori}
Let $3/2\leq\sigma\leq 2$ and  $0<\rho_0\leq 1.$  Suppose  $(u,v)\in L^\infty\inner{[0, T];~X_{\rho_0,\sigma}}$ is the solution to the Prandtl system \eqref{prandtl}    such that  the properties listed in  \eqref{condi}-\eqref{+condi1} hold. 
 Then  there exists   a constant   $C_*>1,$   such that    the estimate 
 \begin{equation}\label{esapr}
	\abs{\vec a (t)}_{\rho,\sigma}^2\leq C_* \abs{\vec a_0}_{\rho, \sigma}^2+C_* \int_{0}^t \inner{\abs{\vec a(s)}_{\rho,\sigma}^2+\abs{\vec a(s)}_{\rho,\sigma}^4} \,ds+C_* \int_{0}^t\frac{ \abs{\vec a(s)}_{\tilde\rho,\sigma}^2}{\tilde \rho-\rho}\,ds 
\end{equation}	 
 holds for any pair $(\rho,\tilde\rho)$ with   $0<\rho<\tilde \rho<\rho_0$ and any $t\in[0,T].$ Here, the constant $C_*$ depends only on   the constants in \eqref{condi}-\eqref{+condi1} as well as the Sobolev embedding constants.  
 \end{theorem}
 
 \begin{remark}
 	When $1<\sigma<3/2$,  we can obtain a similar a priori estimate as \eqref{esapr}, 
 cf. Theorem  \ref{thmnewapr} in Section \ref{sec7}.  
 \end{remark}

In view of Remark \ref{remeqi},  each  term in \eqref{esapr} is  well-defined. We will proceed to derive the upper bound of $\abs{\vec a}_{\rho,\sigma}$
 in  the next   Sections \ref{sec4}-\ref{sec7}. 

 Before proving  Theorem \ref{apriori},  we first list some
basic  inequalities to be used.  
  
\begin{lemma}\label{lemequa}
With the notations in Subsection \ref{subsec21}, the  following 
inequalities hold. 

\noindent (i) For any integer $k\geq 1$ and for any  pair $(\rho,\tilde \rho)$ with $0<\rho<\tilde\rho\leq 1,$ we have 
\begin{equation}
\label{factor}
k\inner{\frac{\rho}{\tilde\rho}}^k\leq 	\frac{1}{\tilde\rho} k\inner{\frac{\rho}{\tilde\rho}}^k\leq\frac{1}{\tilde\rho-\rho}.
\end{equation}

\noindent (ii)  For any suitable function $F,$
\begin{equation}
	\label{soblev}
	\norm{F}_{L^\infty(\Omega)}\leq  \sqrt 2\Big(\norm{F}_{L_{x,y}^2(L_z^\infty)}+\norm{\partial_x F}_{{L_{x,y}^2(L_z^\infty)}}+\norm{\partial_y F}_{{L_{x,y}^2(L_z^\infty)}}+\norm{\partial_x \partial_y F}_{{L_{x,y}^2(L_z^\infty)}}\Big),
\end{equation}
and
\begin{equation}
	\label{soblev+}
	\begin{aligned}
\norm{F}_{L^\infty(\Omega)} \leq & 2\Big(\norm{ F}_{L^2}+\norm{\partial_{x}  F}_{L^2}+\norm{\partial_{y} F}_{L^2}+\norm{  \partial_{z}F}_{L^2}\Big)\\ &+2\Big(\norm{\partial_{x} \partial_{y}F}_{L^2}+\norm{\partial_{x} \partial_{z}F}_{L^2}+\norm{\partial_{y} \partial_{z}F}_{L^2}+\norm{\partial_{x}\partial_y \partial_{z}F}_{L^2}\Big).
	\end{aligned}
\end{equation}

\noindent (iii)  For any $0<r\leq 1$ and any $\alpha=(\alpha_1, \alpha_2)\in\mathbb Z_+^2$ with $\abs\alpha=m$,      we have
\begin{equation}\label{etan}
\begin{aligned}
 &\norm{\comi z^{\ell-1}\partial^\alpha u}_{L^2}+ \norm{\comi z^{\ell}\partial^\alpha \psi}_{L^2}+ \norm{\tau_2\partial^\alpha \xi}_{L^2}+\norm{\comi z^{\ell}f_m }_{L^2} +\norm{\comi z^{\kappa}\partial^\alpha v}_{L^2}\\
 &\quad +\norm{   \comi z^{\kappa+\delta}  \Gamma_m}_{L^2}+\norm{ H_m}_{L^2}+\norm{   G_m}_{L^2}\\[3pt]
  \leq&~
	\left\{
	\begin{aligned}
	& \frac{   [\inner{m-7}!]^{ \sigma}}{r^{ (m-6)}}\abs{\vec a}_{r,\sigma},\quad {\rm if}~m \geq 7,\\
	&\abs{\vec a}_{r,\sigma},\quad {\rm if}~  m \leq 6,
	\end{aligned}
	\right.
	\end{aligned}
\end{equation} 
and 
\begin{equation}\label{emix}
	\norm{\comi z^{\ell+1}\partial^\alpha\partial_z^j \psi}_{L^2}\leq 
	\left\{
	\begin{aligned}
	& \frac{   [\inner{\abs\alpha+j-7}!]^{ \sigma}}{r^{ (\abs\alpha+j-6)}}\abs{\vec a}_{r,\sigma},\quad {\rm if}~\abs\alpha+j\geq 7 ~{\rm and}~ 1\leq j\leq 4,\\
	& \abs{\vec a}_{r,\sigma},\quad {\rm if}~ \abs\alpha+j\leq 6~{\rm and}~1\leq j\leq 4.
	\end{aligned}
	\right.
\end{equation}

\noindent (iv)  For any $0<r\leq 1$ and any $\alpha=(\alpha_1, \alpha_2)\in\mathbb Z_+^2$,      we have
\begin{equation}
\label{chi2est} 
\begin{aligned}
 & \norm{ \partial^\alpha w}_{L_{x,y}^2(L_z^\infty)}+\abs\alpha  \norm{\comi z^{\kappa+2}  \partial^\alpha \eta}_{L^2}\\
&\quad +\abs\alpha\Big(\norm{ \comi z^{\kappa+\delta} g_{\alpha }}_{L^2}+ \norm{\comi z^{\kappa+\delta} h_{\alpha}}_{L^2}+  \norm{\comi z^{\kappa+2}  \partial^\alpha \eta}_{L^2}+  \norm{\vec\theta_\alpha}_{L^2}+  \norm{\vec\mu_\alpha }_{L^2}\Big) \\
    \leq & \left\{
	\begin{aligned}
	& \frac{    [\inner{\abs\alpha-6}!]^{ \sigma}}{r^{ (\abs\alpha-5)}}\abs{\vec a}_{r,\sigma},~{\rm if}~\abs\alpha\geq 7,\\
	& \abs{\vec a}_{r,\sigma},\quad {\rm if}~\abs\alpha \leq 6,
	\end{aligned}
	\right.
	\end{aligned}
\end{equation}
and 
\begin{equation}\label{mixoneta}
	 \abs\alpha\,\norm{\comi z^{\kappa+2}\partial^\alpha\partial_z^j \eta}_{L^2}\leq 
	\left\{
	\begin{aligned}
	& \frac{   [\inner{\abs\alpha+j-6}!]^{ \sigma}}{r^{ (\abs\alpha+j-5)}}\abs{\vec a}_{r,\sigma},\quad {\rm if}~\abs\alpha+j\geq 7 ~{\rm and}~ 1\leq j\leq 4,\\
	& \abs{\vec a}_{r,\sigma},\quad {\rm if}~\abs\alpha+j\leq 6~{\rm and}~1\leq j\leq 4.
	\end{aligned}
	\right.
\end{equation}
\noindent (v) Let $\sigma\geq 1$ and let $m\geq 7.$  Then  for any $0<r\leq 1$ we have
\begin{equation}
\label{fm1}
\norm{\partial_x^{m-1} \xi}_{L^2} \leq \norm{  \partial_z f_{m-1}}_{L^2} +    C m^{-\sigma}\frac{\big[\inner{m-7}!\big]^{\sigma}}{r^{m-7} } \abs{\vec a}_{r,\sigma}.
\end{equation}

 \end{lemma}

\begin{proof}
We refer to \cite[Lemma 3.2]{LY}	for the proof of  (i) and (v), and the proof of   (ii)  follows from the standard Sobolev inequalities (see for example \cite[Lemma A.1]{LY}).  The other inequalities are  direct consequences of the definition of $\abs{\vec a}_{r,\sigma}.$
\end{proof}

 \section{Estimates on $f_m, \tilde f_m, q_m$ and $\tilde q_m$} 
 \label{sec4}
 
This section is for deriving the upper bounds of the weighted $L^2$-norms of  $f_m$ and $q_m,$  defined in  \eqref{falpha} and \eqref{t2}.  And $\tilde f_m$ and $\tilde q_m$ can be handled similarly.  To estimate  $f_m$, we
   will use the cancellation  introduced in \cite{awxy,MW} for 2D Prandtl equations.  The estimation on $q_m$ relies  on another kind of cancellation used in \cite{LY}.    
 
 To simplify the notation, we  use from now on the capital letter $C$ to denote some generic  constant that may vary from line to line, and it
depends  only on the constants   in   \eqref{condi}-\eqref{+condi1}  as well as  the Sobolev embedding constants, in particular, it is independent of the order of derivatives denoted by $m$.

  The proposition below is the main result in this section.

 \begin{proposition}\label{propfm} Let $3/2\leq\sigma\leq 2$ and  $0<\rho_0\leq 1.$ 
Suppose  $(u,v)\in L^\infty\inner{[0, T];~X_{\rho_0,\sigma}}$  is the solution to the Prandtl system \eqref{prandtl} satisfying the conditions \eqref{condi}-\eqref{+condi1}.  
 Then
 	  for any  $ m\geq 7,$   any $t\in[0,T]$ and    any pair $\inner{\rho,\tilde\rho}$ with $0<\rho<\tilde\rho< \rho_0$,  we have
\begin{multline*}
    \frac{\rho^{2( m-6)}}{ [( m-7)!]^{2\sigma}}\Big(\norm{\comi z^\ell f_m(t)}_{L^2}^2+\norm{q_m(t)}_{L^2}^2+\norm{\comi z^\ell \tilde f_m(t)}_{L^2}^2+\norm{ \tilde q_m(t)}_{L^2}^2\Big) \\
    + \frac{\rho^{2( m-6)}}{ [( m-7)!]^{2\sigma}}\int_0^t\Big(\norm{\comi z^\ell\partial_z f_m(s)}_{L^2}^2+ \norm{\comi z^\ell \partial_z\tilde f_m(s)}_{L^2}^2 \Big) ds\\
    \leq   C\abs{\vec a_0}_{\rho,\sigma}^2+
 	C \inner{  \int_0^t   \inner{ \abs{\vec a(s)}_{ \rho,\sigma}^2+\abs{\vec a(s)}_{ \rho,\sigma}^4 }   \,ds+    \int_0^t  \frac{  \abs{\vec a(s)}_{ \tilde\rho,\sigma}^2}{\tilde\rho-\rho}\,ds}. 
\end{multline*}
 \end{proposition}
 
 We will only estimate  $f_m$ and $q_m $ because  $\tilde f_m$ and $\tilde q_m$ can be estimated similarly.  To prove the above proposition, we first derive the equations solved by $f_m$,
  $ m\geq1$, as 
		\begin{equation}
		\label{equforf}
		\partial_t f_m +\big(u \partial_x  +v\partial_y +w\partial_z\big) f_m  -\partial _{z}^2f_m = \tau_1\partial_x^m g  +\mathcal  J_m,
	\end{equation}
 	where  $g $ is defined by \eqref{defR1}, and 
 		\begin{eqnarray*}
	\mathcal J_m 
   &=&\tau_1\chi \sum_{j=1}^m{m\choose j}\big[\big(\partial_x^j  u \big)\partial_x\partial_x^{m-j}u+ \big(\partial_x^j v\big)\partial_y\partial_x^{m-j}u\big] +\tau_1\chi \sum_{j=1}^{m-1}{m\choose j}\big(\partial_x^j w\big) \partial_x^{m-j} \psi\\
   && -\tau_1\sum_{j=1}^m{m\choose j} \big[\big(\partial_x^j u\big)\partial_x\partial_x^{m-j}\psi+\big(\partial_x^j v\big)\partial_y\partial_x^{m-j}\psi\big]  -\tau_1\sum_{j=1}^{m-1}{m\choose j} \big(\partial_x^j w\big)\partial_z\partial_x^{m-j} \psi\\
 &&-\tau_1\Big[\frac{\partial_z g -\eta\partial_y\psi-\chi g }{\psi} -\partial_x \psi+\chi \partial_x u+\chi\partial_y v+2\chi\partial_z\chi\Big]\partial_x^m u+2\tau_1 (\partial_z\chi)\partial_x^m\psi\\
	&&+\inner{w\tau_1'-\tau_1''} \big(\partial_x^{m} \psi-\frac{\partial_z \psi }{ \psi}\partial_x^{m}  u\big)-2\tau_1'\partial_z \big(\partial_x^{m} \psi-\frac{\partial_z \psi }{ \psi}\partial_x^{m}  u\big) 
	\end{eqnarray*}
with $\chi\stackrel{\rm def}{ =} \partial_z\psi / \psi.$
Here and throughout the paper,  ${m\choose j}$ denotes the binomial coefficient.  
 
We just give a sketch for obtaining   \eqref{equforf}-\eqref{gammequ}. 	 The derivation of  \eqref{equforf} relies on the cancellation property observed in \cite{awxy,MW}. 
	  Applying $\partial_x^m$ to the equations \eqref{prandtl} and \eqref{equforpsi}  for $u$ and $\psi,$    it  follows from Leibniz formula that  
 \begin{equation}\label{uequ}
 \begin{aligned}
&	\partial_t \partial_x^m u+\big(u \partial_x  +v\partial_y +w\partial_z\big) \partial_x^m u-\partial _{z}^2\partial_x^m u+\inner{\partial_x^m  w} \psi  \\
 =&-\sum_{j=1}^{m}{m\choose j}\big[\big(\partial_x^j u\big)\partial_x\partial_x^{m-j}u+ \big(\partial_x^j v\big)\partial_y\partial_x^{m-j}u\big] -\sum_{j=1}^{m-1}{m\choose j}\big(\partial_x^j w\big) \partial_x^{m-j} \psi
 \end{aligned}
 \end{equation}
 and
    \begin{multline}\label{eqpsi}
 \partial_t \partial_x^m  \psi+\big(u \partial_x  +v\partial_y +w\partial_z\big) \partial_x^m \psi-\partial _{z}^2\partial_x^m  \psi+\inner{\partial_x^m  w}\partial_z  \psi \\
 = \partial_x^m g  -\sum_{j=1}^{m}{m\choose j} \big[\big(\partial_x^j u\big)\partial_x\partial_x^{m-j}\psi+\big(\partial_x^j v\big)\partial_y\partial_x^{m-j}\psi\big]  -\sum_{j=1}^{m-1}{m\choose j} \big(\partial_x^j w\big)\partial_z\partial_x^{m-j}\psi.
 \end{multline}
 Multiplying the first  equation above by $ \tau_1\partial_z\psi/\psi$ and the second equation by $\tau_1$  and  then  subtracting one  by  another,  we obtain the equation for $f_m.$ 


The following three lemmas are for the estimates on
 the terms involving  $\partial_x^m g$ and $\mathcal J_m $   appearing on the right sides of \eqref{equforf}.  

\begin{lemma}\label{lemg1}
	Let $\sigma\leq 2.$  Then for any   $ m\geq 7,$  and  any pair $(\rho,\tilde\rho)$ with $0<\rho<\tilde\rho<\rho_0\leq 1,$ we have
	\begin{eqnarray*}
		 \inner{ \comi z^\ell\tau_1 \partial_x^m g ,  ~\comi z^\ell  f_m}_{L^2} \leq \frac{C [\inner{ m-7}!]^{2\sigma}}{ \rho^{2( m-6)}}\frac{\abs{\vec a}_{\tilde \rho,\sigma}^2}{\tilde\rho-\rho}.
	\end{eqnarray*}
\end{lemma}

\begin{proof} Note that $\partial_x^m g =g_\alpha$ with $\alpha=(m,0)\in\mathbb Z_+^2,$ and that $\ell\leq \kappa+\delta$ in view of \eqref{ellalpha}. Then we 
	use \eqref{chi2est} and \eqref{etan}  to get 
	\begin{eqnarray*}
			 \inner{ \comi z^\ell\tau_1 \partial_x^m g ,  ~\comi z^\ell  f_m}_{L^2} &\leq& m^{-1}\frac{ [\inner{ m-6}!]^{\sigma}}{\tilde\rho^{ m-5}}  \abs{\vec a}_{\tilde \rho,\sigma}  \frac{ [\inner{ m-7}!]^{\sigma}}{\tilde \rho^{ m-6}}\abs{\vec a}_{\tilde \rho,\sigma}\\
			 &\leq&  \tilde\rho^{-1} m^{\sigma-1} \frac{ \rho^{2( m-6)}}{ \tilde \rho^{2( m-6)}}  \frac{ [\inner{ m-7}!]^{2\sigma}}{ \rho^{2( m-6)}}\abs{\vec a}_{\tilde \rho,\sigma}^2.
\end{eqnarray*}
Moreover, it follows from the fact   $\sigma\leq 2$ and \eqref{factor} that 
\begin{eqnarray*}
	 \tilde\rho^{-1} m^{\sigma-1} \frac{ \rho^{2( m-6)}}{ \tilde \rho^{2( m-6)}}  \leq  	 \tilde\rho^{-1} m \frac{ \rho^{ m-6}}{ \tilde \rho^{ m-6}}\leq \frac{C}{\tilde\rho-\rho}.
\end{eqnarray*}
Then combining these estimates  completes the proof.
\end{proof}

\begin{lemma}\label{lemjm}
	Let $\sigma\in[3/2,2].$  Then for any  $ m\geq 7,$  and  any pair $(\rho,\tilde\rho)$ with $0<\rho<\tilde\rho<\rho_0\leq 1,$ we have
	\begin{equation*}
		 \inner{ \comi z^\ell \mathcal J_m,  ~\comi z^\ell  f_m}_{L^2}
		  \leq   \frac{C [\inner{ m-7}!]^{2\sigma}}{ \rho^{2( m-6)}} \big(\abs{\vec a}_{ \rho,\sigma}^2+\abs{\vec a}_{ \rho,\sigma}^3\big)+ \frac{C [\inner{ m-7}!]^{2\sigma}}{ \rho^{2(m-6)}}\frac{\abs{\vec a}_{\tilde \rho,\sigma}^2}{\tilde\rho-\rho}.
	\end{equation*}
\end{lemma}

 \begin{proof} We divide the term into
 \begin{eqnarray*}
 \inner{ \comi z^\ell\mathcal J_m, ~ \comi z^\ell  f_m}_{L^2}=A_1+A_2+A_3+A_4
\end{eqnarray*}
with
\begin{eqnarray*}
A_1&=&\sum_{j=1}^{m}{m\choose j}\inner{\tau_1\chi  \comi z^\ell\big[\big(\partial_x^j  u \big)\partial_x\partial_x^{m-j}u+ \big(\partial_x^j v\big)\partial_y\partial_x^{m-j}u\big],  ~\comi z^\ell  f_m}_{L^2} \\
&& -
\sum_{j=1}^{m}{m\choose j}\inner{ \tau_1 \comi z^\ell\big[\big(\partial_x^j u\big)\partial_x\partial_x^{m-j}\psi+\big(\partial_x^j v\big)\partial_y\partial_x^{m-j}\psi\big] ,  ~\comi z^\ell  f_m}_{L^2},\\
A_2&=&\sum_{j=1}^{m-1}{m\choose j}  \inner{\tau_1\chi \comi z^\ell \big(\partial_x^j w\big) \partial_x^{m-j} \psi,  ~\comi z^\ell  f_m}_{L^2}\\
&&- \sum_{j=1}^{m-1}{m\choose j}  \inner{ \tau_1  \comi z^\ell \big(\partial_x^j w\big)\partial_z\partial_x^{m-j} \psi,  ~\comi z^\ell  f_m}_{L^2},\\
A_3&	= &
 2 \inner{ \comi z^\ell \tau_1 (\partial_z\chi)\partial_x^m\psi, ~ \comi z^\ell  f_m}_{L^2}\\
&&   -\Big( \comi z^\ell \tau_1 \Big[\frac{\partial_z g -\eta\partial_y\psi-\chi g }{\psi} -\partial_x \psi+\chi \partial_x u+\chi\partial_y v+2\chi\partial_z\chi\Big]\partial_x^m u,   ~\comi z^\ell  f_m\Big)_{L^2}
\end{eqnarray*}
and
\begin{eqnarray*}
A_4= 
 \inner{ \comi z^\ell \Big[\inner{w\tau_1'-\tau_1''} \big(\partial_x^{m} \psi-\frac{\partial_z \psi }{ \psi}\partial_x^{m}  u\big)-2\tau_1'\partial_z \big(\partial_x^{m} \psi-\frac{\partial_z \psi }{ \psi}\partial_x^{m}  u\big)\Big], ~ \comi z^\ell  f_m}_{L^2}.
\end{eqnarray*}
We now derive the upper bounds for  $A_j, 1\leq j\leq 4,$  term by term.    

\medskip
\noindent\underline{\it  Upper bound for $A_3$ and $A_4$}.  
  In view of the conditions \eqref{condi}-\eqref{+condi1} and the fact that  $\abs{\chi}\leq C\comi z^{-1}$ on supp\ $\tau_1$,  by observing \eqref{ellalpha},  we have 
    \begin{eqnarray*}
  	\norm{\tau_1\partial_z\chi}_{L^\infty}+\norm{\comi z\tau_1  \big[\frac{ -\eta\partial_y\psi }{\psi} -\partial_x \psi+\chi\partial_x u+\chi\partial_y v+2\chi\partial_z\chi\big] }_{L^\infty} \leq C.
  \end{eqnarray*}
 And  using \eqref{ellalpha} and  the representation \eqref{defR1} of $g $    gives
 \begin{eqnarray*}
 	 \norm{\comi z \tau_1  \frac{\partial_z g  -\chi g }{ \psi}    }_{L^\infty} \leq C.
 \end{eqnarray*} 
  Thus,
\begin{equation*}
	A_3\leq C  \inner{ \norm{\comi z^{\ell-1} \partial_x^m u}_{L^2}+\norm{\comi z^\ell  \partial_x^m \psi}_{L^2}}\norm{\comi z^\ell  f_m}_{L^2}\leq \frac{C [\inner{ m-7}!]^{2\sigma}}{ \rho^{2( m-6)}} \abs{\vec a}_{ \rho,\sigma}^2,
\end{equation*}
the last inequality following from \eqref{etan}. Similarly,  we have byobserving that $\tau_1'=\tau_1'\tau_2$ due to \eqref{suppch2} and that $\comi z$ is bounded on supp\ $\tau_1'$ or supp\ $\tau_1'',$ 
\begin{eqnarray*}
	A_4&\leq& \norm{ f_m}_{L^2} \Big( \norm{ \partial_x^m u}_{L^2}+\norm{   \partial_x^m \psi}_{L^2}+\norm{\tau_2 \partial_x^m \xi}_{L^2}\Big)   \leq   \frac{C [\inner{ m-7}!]^{2\sigma}}{ \rho^{2( m-6)}} \abs{\vec a}_{ \rho,\sigma}^2,
\end{eqnarray*}
where in the last inequality we have used \eqref{etan}. 
Thus, we conclude 
\begin{eqnarray*}
	A_3+A_4\leq   \frac{C [\inner{ m-7}!]^{2\sigma}}{ \rho^{2( m-6)}} \abs{\vec a}_{ \rho,\sigma}^2.
\end{eqnarray*}

\medskip
\noindent\underline{\it  Upper bound for $A_2$}.  We will show that
 $A_2$ satisfies
\begin{equation}
\label{a2}
	A_2\leq \frac{C [\inner{ m-7}!]^{2\sigma}}{\rho^{2( m-6)}} \abs{\vec a}_{ \rho,\sigma}^3+\frac{C [\inner{ m-7}!]^{2\sigma}}{\rho^{2( m-6)}} \frac{\abs{\vec a}_{\tilde \rho,\sigma}^2}{\tilde\rho-\rho}.
\end{equation} 
For this,  firstly, we have 
\begin{eqnarray*}
	A_2&\leq& \norm{\comi z^\ell  f_m}_{L^2}   \sum_{1\leq j\leq m-1}{m\choose j}  \norm{ \comi z^\ell \big(\partial_x^j w\big)\partial_z\partial_x^{m-j} \psi}_{L^2}\\
	&&+\norm{\comi z^\ell  f_m}_{L^2}   \sum_{1\leq j\leq m-1}{m\choose j}  \norm{ \comi z^\ell \big(\partial_x^j w\big)\partial_x^{m-j} \psi}_{L^2}\\
	 &\stackrel{\rm def}{ =}& A_{2,1}+A_{2,2}. 
\end{eqnarray*} 
We will  only estimate
     $A_{2,1}$  because $A_{2,2}$ can be handled similarly.    By  \eqref{+condi1}, \eqref{emix}  and \eqref{chi2est}, we have 
\begin{equation}
\label{been}
\begin{aligned}
	& \Big[\sum_{1\leq j\leq 2}+\sum_{m-2\leq j\leq m-1}\Big] {m\choose j}  \norm{ \comi z^\ell \big(\partial_x^j w\big)\partial_z\partial_x^{m-j} \psi}_{L^2}\\
	\leq &~ C m \frac{  [\inner{ m-7}!]^{\sigma}}{\tilde \rho^{ m-6}}\abs{\vec a}_{\tilde \rho,\sigma}\leq  \frac{ C[\inner{ m-7}!]^{\sigma}}{\rho^{ m-6}}\frac{\abs{\vec a}_{\tilde \rho,\sigma}}{\tilde\rho-\rho},
	\end{aligned}
\end{equation}
where in the last inequality we have used  \eqref{factor}.  Next,
we estimate the remaining terms in the summation. Note that
\begin{eqnarray*}
  \sum_{3\leq j\leq m-3}{m\choose j}  \norm{ \comi z^\ell \big(\partial_x^j w\big)\partial_z\partial_x^{m-j} \psi}_{L^2} \leq T_{1}+T_2   \end{eqnarray*}
with 
\begin{eqnarray*}
	T_1&=&\sum_{j=3}^{[(m-3)/2]}{m\choose j}  \norm{ \comi z^\ell \big(\partial_x^j w\big)\partial_z\partial_x^{m-j} \psi}_{L^2},\\
	T_2&=& \sum_{j=[(m-3)/2]+1}^{m-3}{m\choose j}  \norm{ \comi z^\ell \big(\partial_x^j w\big)\partial_z\partial_x^{m-j} \psi}_{L^2},
\end{eqnarray*}
where as standard, $[p] $ denotes the largest integer less than or equal to $p.$
By using Sobolev inequality \eqref{soblev} and \eqref{etan}-\eqref{chi2est}, we
have
 \begin{eqnarray*}
	T_1&\leq & \sum_{j=3}^{[(m-3)/2]}{m\choose j}  \norm{ \big(\partial_x^j w\big)}_{L^\infty}\norm{\comi z^\ell \partial_z\partial_x^{m-j} \psi}_{L^2}\\
	&\leq &C\frac{m!}{3!(m- 3)!}  \abs{\vec a}_{\rho,\sigma}\frac{[(m-9)!]^\sigma}{\rho^{m-8}} \abs{\vec a}_{\rho,\sigma}\\
	&&+  C \sum_{j=4}^{[(m-3)/2]}\frac{m!} {j!(m-j)!} \frac{[(j-4)!]^\sigma}{\rho^{j-3}} \abs{\vec a}_{\rho,\sigma}\frac{[(m-j-6)!]^\sigma}{\rho^{m-j-5}} \abs{\vec a}_{\rho,\sigma}\\
	&\leq &  \frac{C  \abs{\vec a}_{\rho,\sigma}^2}{\rho^{m-8}}m^3[(m-9)!]^{\sigma}+ \frac{C  \abs{\vec a}_{\rho,\sigma}^2}{\rho^{m-8}} \sum_{j=4}^{[(m-3)/2]}\frac{m! [(j-4)!]^{\sigma-1} [(m-j-6)!]^{\sigma-1}} {j^4(m-j)^6}   \\
	&\leq &  \frac{C  \abs{\vec a}_{\rho,\sigma}^2}{\rho^{m-6}} [(m-7)!]^{\sigma} \big(m^3m^{-2\sigma}\big)+\frac{C  \abs{\vec a}_{\rho,\sigma}^2}{\rho^{m-6}} \sum_{j=4}^{[(m-3)/2]}\frac{(m-7)! m^7} {j^4 m^6}   [(m-10)!]^{\sigma-1}\\
	&\leq & \frac{C  [(m-7)!]^{\sigma} \abs{\vec a}_{\rho,\sigma}^2 }{\rho^{m-6}}  \bigg( m^3m^{-2\sigma}+\sum_{j=4}^{[(m-3)/2]}\frac{  m } {j^4  m^{3(\sigma-1)}}\bigg)\\
	&\leq & \frac{C  [(m-7)!]^{\sigma}  }{\rho^{m-6}}  \abs{\vec a}_{\rho,\sigma}^2,
\end{eqnarray*}
where in the last inequality we have used
 the fact that $\sigma\geq 3/2.$   Similarly,
\begin{eqnarray*}
	T_2&\leq & \sum_{j=[(m-3)/2]+1}^{m-3}{m\choose j}  \norm{\big(\partial_x^j w\big)}_{L_{x,y}^2(L_z^\infty)} \norm{\comi z^\ell \partial_z\partial_x^{m-j} \psi}_{L_{x,y}^\infty(L_z^2)}\\
	&\leq &  \sum_{j=[(m-3)/2]+1}^{m-4}\frac{m!} {j!(m-j)!} \frac{[(j-6)!]^\sigma}{\rho^{j-5}} \abs{\vec a}_{\rho,\sigma}\frac{[(m-j-4)!]^\sigma}{\rho^{m-j-3}} \abs{\vec a}_{\rho,\sigma}\\
	&&+\frac{C  \abs{\vec a}_{\rho,\sigma}^2}{\rho^{m-8}}m^3[(m-9)!]^{\sigma}\\
	&\leq & \frac{C  \abs{\vec a}_{\rho,\sigma}^2}{\rho^{m-8}} \sum_{j=[(m-3)/2]+1}^{m-4}\frac{m![(j-6)!]^{\sigma-1} [(m-j-4)!]^{\sigma-1}} {j^6(m-j)^4}   +\frac{C  \abs{\vec a}_{\rho,\sigma}^2}{\rho^{m-8}}m^3[(m-9)!]^{\sigma}\\
	&\leq & \frac{C  \abs{\vec a}_{\rho,\sigma}^2}{\rho^{m-6}} \sum_{j=[(m-3)/2]+1}^{m-4}\frac{(m-7)! m^7} {m^6 (m-j)^4}   [(m-10)!]^{\sigma-1}+\frac{C  \abs{\vec a}_{\rho,\sigma}^2}{\rho^{m-6}}m^3[(m-9)!]^{\sigma}\\
	&\leq & \frac{C  [(m-7)!]^{\sigma} \abs{\vec a}_{\rho,\sigma}^2 }{\rho^{m-6}}\bigg(m^3m^{-2\sigma}+\sum_{j=[(m-3)/2]+1}^{m-4}\frac{  m} { (m-j)^4 m^{3(\sigma-1)}} \bigg)\\
	&\leq & \frac{C  [(m-7)!]^{\sigma}  }{\rho^{m-6}}  \abs{\vec a}_{\rho,\sigma}^2.
\end{eqnarray*}
Combining these inequalities yields
\begin{eqnarray*}
	 \sum_{3\leq j\leq m-3}{m\choose j}  \norm{ \comi z^\ell \big(\partial_x^j w\big)\partial_z\partial_x^{m-j} \psi}_{L^2} \leq \frac{C  [(m-7)!]^{\sigma}  }{\rho^{m-6}}  \abs{\vec a}_{\rho,\sigma}^2,
\end{eqnarray*}
which along with  \eqref{been} gives
\begin{eqnarray*}
	\sum_{1\leq j\leq m-1}{m\choose j}  \norm{ \comi z^\ell \big(\partial_x^j w\big)\partial_z\partial_x^{m-j} \psi}_{L^2}\leq  \frac{C [\inner{ m-7}!]^{ \sigma}}{\rho^{  m-6 }} \abs{\vec a}_{ \rho,\sigma}^2+\frac{C [\inner{ m-7}!]^{ \sigma}}{\rho^{  m-6 }} \frac{\abs{\vec a}_{\tilde \rho,\sigma}}{\tilde\rho-\rho},
\end{eqnarray*}
and thus,
\begin{eqnarray*}
	 A_{2,1}
	&\leq &\frac{C [\inner{ m-7}!]^{2\sigma}}{\rho^{2( m-6)}} \abs{\vec a}_{ \rho,\sigma}^3+\frac{C [\inner{ m-7}!]^{2\sigma}}{\rho^{2( m-6)}} \frac{\abs{\vec a}_{\tilde \rho,\sigma}^2}{\tilde\rho-\rho}.
\end{eqnarray*}
The upper bound for $A_{2,2}$ is similar  and thus  \eqref{a2} follows.  

\medskip
\noindent\underline{\it  Upper bound for $A_1$}. Just following the argument used  above for $A_2$ with slight modification,  we can obtain
\begin{eqnarray*}
	A_1\leq \frac{C [\inner{ m-7}!]^{2\sigma}}{\rho^{2( m-6)}} \abs{\vec a}_{ \rho,\sigma}^3+\frac{C [\inner{ m-7}!]^{2\sigma}}{\rho^{2( m-6)}} \frac{\abs{\vec a}_{\tilde \rho,\sigma}^2}{\tilde\rho-\rho}.
\end{eqnarray*}
This  completes the proof of Lemma \ref{lemjm}.    
 \end{proof}

\begin{proof}[Proof of Proposition \ref{propfm}] We first estimate $f_m.$ 
 	 Multiplying both sides of  equation \eqref{equforf} by $  \comi z^{2\ell} f_m$ and  then integrating over $\Omega$,  by integration by parts and observing $\partial_zf_m|_{z=0}=0,$ we have, for any $\eps>0,$
 	\begin{equation*}
 	\begin{aligned}
 	& 	\frac{1}{2}\frac{d}{dt}   \norm{\comi z^\ell f_m}_{L^2}^2+ \norm{\comi z^\ell  \partial_z  f_m }_{L^2}^2\\
 	 =&   \inner{ \comi z^\ell( \partial_x^mg +\mathcal J_m),  \,\comi z^\ell  f_m}_{L^2}+ \inner{ w (\partial_z\comi z^\ell)f_m,  \,\comi z^\ell  f_m}_{L^2} +\frac{1}{2}  \inner{  \big(\partial_{z}^2\comi z^{2\ell}\big)  f_m,   \,  f_m}_{L^2}  
 	\\
 	 \leq&\inner{ \comi z^\ell \partial_x^mg , ~ \comi z^\ell  f_m}_{L^2}+  \inner{ \comi z^\ell\mathcal J_m, ~ \comi z^\ell  f_m}_{L^2}+   C  \norm{   \comi z^\ell  f_m}_{L^2}^2 \\
 	 \leq&   \frac{C [\inner{ m-7}!]^{2\sigma}}{ \rho^{2( m-6)}} \inner{\abs{\vec a}_{ \rho,\sigma}^2+\abs{\vec a}_{ \rho,\sigma}^3}+ \frac{C [\inner{ m-7}!]^{2\sigma}}{ \rho^{2( m-6)}}\frac{\abs{\vec a}_{\tilde \rho,\sigma}^2}{\tilde\rho-\rho},
 	\end{aligned}
 	\end{equation*}
where 	the last line follows from Lemmas  \ref{lemg1}-\ref{lemjm} as well as \eqref{etan}.  Then   integrating over $[0,t],$  we obtain  the upper bound as desired for the term $\norm{\comi z^\ell f_m}_{L^2}.$  The estimate on   $\norm{q_m}_{L^2}$ is similar.  In fact, we multiply equation \eqref{eqpsi}    by  $\tau_2(\partial_z\xi)/\xi$, and apply $\tau_2\partial_x^m$ to the first equation in \eqref{seco}, and finally  subtracting one  by  another. This gives the equation solved by $q_m$:
\begin{multline*}
	\Big(\partial_t    +u\partial_x +v\partial_y +w\partial_z -\partial_z^2  \Big) q_m=\tau_2\sum_{j=1}^3\partial_x^m\theta_j-b\tau_2\partial_x^mg\\
	- \tau_2 \sum_{j=1}^{m}{m\choose j} \big[\big(\partial_x^j u\big)\ \partial_x^{m-j+1}\xi+\big(\partial_x^j v\big)\partial_y\partial_x^{m-j}\xi\big]  -\tau_2\sum_{j=1}^{m-1}{m\choose j} \big(\partial_x^j w\big)\partial_z\partial_x^{m-j}\xi\\
	+b\tau_2 \sum_{j=1}^{m}{m\choose j} \big[\big(\partial_x^j u\big) \partial_x^{m-j+1}\psi+\big(\partial_x^j v\big)\partial_y\partial_x^{m-j}\psi\big] + b\tau_2\sum_{j=1}^{m-1}{m\choose j} \big(\partial_x^j w\big) \partial_x^{m-j}\xi 	\\+2\inner{\partial_zb  }\partial_z (\tau_2\partial_x^m\psi) 
	  -  \tau_2 ' b w\partial_x^m \psi+  b \tau_2''  \partial_x^m \psi+ 2b  \tau_2' \partial_x^m \xi  
	 +  \tau_2 'v\partial_x^m \xi-  \tau_2'' \partial_x^m \xi- 2\tau_2' \partial_x^m \partial_z \xi +\mathcal P_{m} ,
\end{multline*}	
where  $b =\frac{\partial_z\xi}{\xi}$ and 
\begin{eqnarray*}
	\mathcal P_{m}&=& \frac{2\Big(\psi\partial_x \xi-\inner{\partial_x u} \partial_z\xi\Big)}{\xi}\tau_2\partial_x^m\psi-\frac{  \Big(\psi\partial_x \psi-\inner{\partial_x u}  \xi\Big)\partial_z\xi}{\xi^2}\tau_2\partial_x^m\psi\\
	&&-  \frac{ 2    \inner{\partial_z^2\xi}  \partial_z\xi  }{\xi^2} \tau_2\partial_x^m\psi+\frac{  2\inner{\partial_z\xi}^2 \partial_z\xi}{\xi^3}\tau_2\partial_x^m\psi.
	\end{eqnarray*}
 Then by repeating the argument for treating $f_m$,  we can obtain the desired upper bound for $\norm{q_m}_{L^2};$ 
 see also \cite[Subsection 5.2]{LY} for the detailed computation. 
 The estimate on $\tilde f_m$ and $\tilde q_m$  can be obtained similarly.   Then we have completed the proof of Proposition  \ref{propfm}. 
 \end{proof}

 \section{Estimates on $\Gamma_m, H_m, G_m$ and $\tilde\Gamma_m,\tilde H_m, \tilde G_m$}
 
This section is for deriving the upper bounds of the  $L^2$-norms of  $\Gamma_m, H_m$ and $G_m,$  defined in  \eqref{Gaal}.  And $\tilde\Gamma_m,\tilde H_m, \tilde G_m$  can be handled similarly.  The main tool here is   another kind of cancellation without monotonicity or concave condition.

  The proposition below is the main result in this section.

 \begin{proposition}\label{+propfm+++} Let $3/2\leq\sigma\leq 2$ and  $0<\rho_0\leq 1.$ 
Suppose  $(u,v)\in L^\infty\inner{[0, T];~X_{\rho_0,\sigma}}$  is the solution to the Prandtl system \eqref{prandtl} satisfying the conditions \eqref{condi}-\eqref{+condi1}.  
 Then
 	  for any  $ m\geq 7,$   any $t\in[0,T]$ and    any pair $\inner{\rho,\tilde\rho}$ with $0<\rho<\tilde\rho< \rho_0$,  we have
\begin{multline*}
    \frac{\rho^{2( m-6)}}{ [( m-7)!]^{2\sigma}}\Big(\norm{\comi z^{\kappa+\delta}\Gamma_m(t)}_{L^2}^2+\norm{H_m(t)}_{L^2}^2+\norm{G_m(t)}_{L^2}^2 \Big) \\
    \leq   C\abs{\vec a_0}_{\rho,\sigma}^2+
 	C \inner{  \int_0^t   \inner{ \abs{\vec a(s)}_{ \rho,\sigma}^2+\abs{\vec a(s)}_{ \rho,\sigma}^4 }   \,ds+    \int_0^t  \frac{  \abs{\vec a(s)}_{ \tilde\rho,\sigma}^2}{\tilde\rho-\rho}\,ds},
\end{multline*}
and similarly for the upper bound of  $ \tilde\Gamma_m, \tilde H_m $ and $\tilde G_m.$
 \end{proposition}

 We first hanlde  $\Gamma_m, $ and to do so we  apply another kind of  cancellation by virtue of the equations solved by and $u$ and $v$. Precisely, applying $\partial_x^m$ to the equation \eqref{prandtl} for $v$  gives 
 	\begin{eqnarray*}
 &&	\partial_t\partial_x^m  v+\big(u \partial_x  +v\partial_y +w\partial_z\big) \partial_x^m v-\partial _{z}^2\partial_x^m  v+\inner{\partial_x^m  w}\eta\\
 &=&-\sum_{j=1}^{m}{m\choose j}\big [\big(\partial_x^j u\big)\partial_x\partial_x^{m-j}v+ \big(\partial_x^j v\big)\partial_y\partial_x^{m-j}v\big]-\sum_{j=1}^{m-1}{m\choose j} \big(\partial_x^j w\big) \partial_x^{m-j}\eta.
 \end{eqnarray*}
 We multiply the above equation by $\psi$ and multiply the equation \eqref{uequ}     by $\eta$, and then subtract one by another to have
 \begin{equation}\label{gammequ}
  \partial_t \Gamma_m +\big(u \partial_x  +v\partial_y +w\partial_z\big)\Gamma_m  -\partial _{z}^2\Gamma_m =\mathcal L_m,
 	\end{equation}
 	where  
		\begin{eqnarray*}
 	\mathcal L_m
 &=&\sum_{j=1}^{m}{m\choose j}\big[\big(\partial_x^j u\big)\partial_x\partial_x^{m-j}u+\big(\partial_x^j v\big)\partial_y\partial_x^{m-j}u\big] \eta+\sum_{j=1}^{m-1}{m\choose j}\big(\partial_x^j w\big)\big( \partial_x^{m-j}\psi\big)\eta\\
	&&-\sum_{j=1}^{m}{m\choose j}\big[\big(\partial_x^j u\big)\partial_x\partial_x^{m-j}v+\big(\partial_x^j v\big)\partial_y\partial_x^{m-j}v\big] \psi -\sum_{j=1}^{m-1}{m\choose j}\big(\partial_x^j w\big)\big( \partial_x^{m-j}\eta\big) \psi \\
	&&+g \partial_x^m v-2\inner{ \partial_z\psi} \partial_x^m \eta-\Big(h  \partial_x^m u-2\inner{\partial_z\eta}\partial_x^m \psi\Big)
	 \end{eqnarray*}
 with $g, h  $ defined by \eqref{defR1}. Note the terms in the last line above  come from the commutators between the functions $ \psi ,\eta$ and the differential operators,  where we have used the equations  \eqref{equforpsi}  for $\psi$ and $\eta.$

 \begin{lemma}\label{lemlm}
	Let $\sigma\in[3/2,2].$  Then for any  $ m\geq 7,$  and  any pair $(\rho,\tilde\rho)$ with $0<\rho<\tilde\rho<\rho_0\leq 1,$ we have
	\begin{eqnarray*}
		 \inner{ \comi z^{\kappa+\delta} \mathcal L_m,  ~\comi z^{\kappa+\delta}  \Gamma_m}_{L^2} \leq   \frac{C [\inner{ m-7}!]^{2\sigma}}{ \rho^{2( m-6)}} \big(\abs{\vec a}_{ \rho,\sigma}^2+\abs{\vec a}_{ \rho,\sigma}^3\big)+ \frac{C [\inner{ m-7}!]^{2\sigma}}{ \rho^{2( m-6)}}\frac{\abs{\vec a}_{\tilde \rho,\sigma}^2}{\tilde\rho-\rho}.
	\end{eqnarray*}
\end{lemma}

\begin{proof}
	The proof is similar to those for Lemmas \ref{lemg1}-\ref{lemjm}.  For instance,  following the argument in Lemma \ref{lemg1}, it yields 
	\begin{eqnarray*}
		 -2\inner{ \comi z^{\kappa+\delta}\inner{\partial_z\psi} \partial_x^m \eta,  ~\comi z^{\kappa+\delta}  \Gamma_m}_{L^2}\leq \frac{C [\inner{ m-7}!]^{2\sigma}}{ \rho^{2( m-6)}}\frac{\abs{\vec a}_{\tilde \rho,\sigma}^2}{\tilde\rho-\rho}.
	\end{eqnarray*}
	Meanwhile the other terms involving in $\mathcal L_m$ can be handled 
 as  in Lemma \ref{lemjm}. We omit the detail here for brevity. 
\end{proof}

\begin{proof}[Proof of Proposition \ref{+propfm+++}] We 
 	 multiply both sides of  equation \eqref{gammequ} by $  \comi z^{2(\kappa+\delta)} \Gamma_m$ and  then integrating over $\Omega$,  by integration by parts and observing $\Gamma_m|_{z=0}=0,$  we have,
 	\begin{equation*}
 	\begin{aligned}
 	& 	\frac{1}{2}\frac{d}{dt}   \norm{\comi z^{\kappa+\delta} \Gamma_m}_{L^2}^2+ \norm{\comi z^{\kappa+\delta}  \partial_z  \Gamma_m }_{L^2}^2\\
 	 =&   \inner{ \comi z^{\kappa+\delta}  \mathcal L_m,  \,\comi z^\ell  \Gamma_m}_{L^2}+ \inner{ w (\partial_z\comi z^{\kappa+\delta})\Gamma_m,  \,\comi z^{\kappa+\delta}  \Gamma_m}_{L^2} +\frac{1}{2}  \inner{  \big(\partial_{z}^2\comi z^{2\inner{\kappa+\delta}}\big)  \Gamma_m,   \,  \Gamma_m}_{L^2}  
 	\\
 	 \leq& \inner{ \comi z^{\kappa+\delta}  \mathcal L_m,  \,\comi z^\ell  \Gamma_m}_{L^2}+   C  \norm{   \comi z^{\kappa+\delta}  \Gamma_m}_{L^2}^2 \\
 	 \leq&   \frac{C [\inner{ m-7}!]^{2\sigma}}{ \rho^{2( m-6)}} \inner{\abs{\vec a}_{ \rho,\sigma}^2+\abs{\vec a}_{ \rho,\sigma}^3}+ \frac{C [\inner{ m-7}!]^{2\sigma}}{ \rho^{2( m-6)}}\frac{\abs{\vec a}_{\tilde \rho,\sigma}^2}{\tilde\rho-\rho},
 	\end{aligned}
 	\end{equation*}
where 	the last line follows from Lemmas  \ref{lemlm} as well as \eqref{etan}.  Then   integrating over $[0,t],$  we obtain  the upper bound as desired for the term  $\norm{\comi z^{\kappa+\delta} \Gamma_m}_{L^2}.$      

The other terms $\norm{H_m}_{L^2}$ and $\norm{G_m}_{L^2}$ can be treated in the same way. In fact, we apply similar kind of  cancellation  by virtue of the equations solved by   $v$ and $\psi$ to get  the desired upper bound for $\norm{H_m}_{L^2}$, and meanwhile by virtue of the equations for   $u$ and $\psi$ to obtain the estimate on  $\norm{G_m}_{L^2}.$   Since the argument is quite similar as that for handling $\norm{\comi z^{\kappa+\delta} \Gamma_m}_{L^2}$ we omit it for brevity. 
 The estimate on $\tilde \Gamma_m, \tilde H_m$ and $\tilde G_m$  can be obtained in the same way.   Then we have completed the proof of Proposition \ref{+propfm+++}. 
 \end{proof}

 \section{Estimates  on $g_{\alpha}, h_\alpha$ and  $\vec\theta_\alpha, \vec\mu_\alpha$}
  \label{sec5}
 In this section, we estimate  $g_{\alpha}, h_\alpha$ and $\vec\theta_\alpha, \vec\mu_\alpha,$ which are defined by  \eqref{galpha} and \eqref{themu},  and  the main result can be stated as follows.
  
 \begin{proposition}\label{propgalpha} Let $3/2\leq\sigma\leq 2$ and  $0<\rho_0\leq 1.$ 
 Suppose  $(u,v)\in L^\infty\inner{[0, T];~X_{\rho_0,\sigma}}$  is the solution to the Prandtl system \eqref{prandtl} satisfying the conditions \eqref{condi}-\eqref{+condi1}.  
 Then
 	  for any  $ \alpha\in\mathbb Z_+^2$ with $\abs\alpha\geq 7,$  any $t\in[0,T]$ and    any pair $\inner{\rho,\tilde\rho}$ with $0<\rho<\tilde\rho<\rho_0$,  we have
\begin{multline*}
   \frac{\rho^{2( \abs\alpha-5)}}{ [( \abs\alpha-6)!]^{2\sigma}}\Big(\abs\alpha^2\norm{\comi z^{\kappa+\delta} g_\alpha(t)}_{L^2}^2+\abs\alpha^2\norm{\comi z^{\kappa+\delta}h_\alpha(t)}_{L^2}^2  \Big) \\
   + \frac{\rho^{2( \abs\alpha-5)}}{ [( \abs\alpha-6)!]^{2\sigma}}\Big( \abs\alpha^2\norm{ \vec\theta_\alpha(t)}_{L^2}^2+\abs\alpha^2\norm{ \vec\mu_\alpha(t)}_{L^2}^2 \Big) \\
   \leq  C\abs{\vec a_0}_{\rho,\sigma}^2+
 	C \inner{  \int_0^t   \inner{ \abs{\vec a(s)}_{ \rho,\sigma}^2+\abs{\vec a(s)}_{ \rho,\sigma}^4 }   \,ds+    \int_0^t  \frac{  \abs{\vec a(s)}_{ \tilde\rho,\sigma}^2}{\tilde\rho-\rho}\,ds}. 
\end{multline*}
 \end{proposition}
 
 We will only estimate $g_\alpha$ because the treatment of  
  $h_\alpha, \vec\theta_\alpha$ and $\vec\mu_\alpha$ is similar.
By \eqref{realp}, it suffices to estimate  $g_m$ and $\tilde g_m$ with 
 \begin{equation}
 \label{gmeq}
 	g_m\stackrel{\rm def}{ =}\partial_x^m g ,\quad \tilde g_m\stackrel{\rm def}{ =}\partial_y^m g .
 \end{equation}
 For the same reason, we handle only $g_m$. 
 And we begin with the equation solved by $g_m.$
 Let $m\geq 1$ and let   $g_m$ be given    by   \eqref{gmeq}. Then 
	\begin{equation}
		\label{equforg}
		\partial_t g_m +\big(u \partial_x  +v\partial_y +w\partial_z\big) g_m  -\partial _{z}^2g_m =\mathcal K_m,
	\end{equation}
	where   
	\begin{equation*}
	\begin{aligned}
	\mathcal K_m &= -\sum_{1\leq j \leq m}{{m}\choose j}[\big(\partial_x^j  u \big)\partial_xg_{m-j}+ \big(\partial_x^j v\big)\partial_yg_{m-j}+ \big(\partial_x^j w\big)\partial_zg_{m-j}] \\
 &\quad+2\partial_x^{m}[(\partial_y \psi)\partial_z\eta] -2\partial_x^{m}[(\partial_y\eta)\partial_z \psi].
	\end{aligned}
	\end{equation*}
To see this,
multiplying the first equation in \eqref{equforpsi}   by   $\partial_y v$ and the second equation by $\partial_y u,$  and then subtracting one by another, we get the equation   solved by $g$, i.e., 
 \begin{equation}\label{eqforg1}
 \partial_t g  +\big(u \partial_x  +v\partial_y +w\partial_z\big) g   -\partial _{z}^2g  =2(\partial_y \psi)\partial_z\eta-2(\partial_y\eta)\partial_z \psi.
 \end{equation}
Taking $\partial_x^{m}$ on both sides and then using  Leibniz formula, we obtain  the equation \eqref{equforg}.

The following three lemmas are for the estimates on the terms in $\mathcal K_m.$ 

\begin{lemma}\label{lemgm}
Let $\sigma\in[3/2,2].$  Then for any   $m\geq 7,$     any pair $(\rho,\tilde\rho)$ with $0<\rho<\tilde\rho<\rho_0\leq 1,$   and  any $\eps>0,$ we have  
	\begin{equation}\label{estpsieta}
	\begin{aligned}
		&2m^2\Big(\comi z^{\kappa+\delta}\partial_x^{m}[(\partial_y \psi)\partial_z\eta], ~\comi z^{\kappa+\delta} g_m\Big)_{L^2}\\
		\leq&~   \eps m^2\norm{\comi z^{\kappa+\delta} \partial_z g_m}_{L^2}^2+  \frac{C_\eps  [\inner{m-6}!]^{2\sigma}}{\rho^{2(m-5)}} \inner{\abs{\vec a}_{ \rho,\sigma}^3+   \abs{\vec a}_{ \rho,\sigma}^4+ \frac{\abs{\vec a}_{\tilde \rho,\sigma}^2}{\tilde\rho-\rho}},
		\end{aligned}
	\end{equation}
	where $C_\eps$ is a constant depending on $\eps.$
\end{lemma} 

\begin{proof} We use  Leibniz formula   to get 
\begin{eqnarray*}
	&&m^2\Big(\comi z^{\kappa+\delta}\partial_x^{m}[(\partial_y \psi)\partial_z\eta], ~\comi z^{\kappa+\delta} g_m\Big)_{L^2}\\
	&= &m^2\bigg(\sum_{0\leq j \leq [m/2]}+\sum_{[m/2]+1\leq j \leq m }\bigg){{m}\choose j} \Big(\comi z^{\kappa+\delta} (\partial_y \partial_x^j\psi)\partial_z\partial_x^{m-j}\eta , ~\comi z^{\kappa+\delta} g_m\Big)_{L^2}.
	\end{eqnarray*}
	Furthermore, using integration by parts for the first term on the right side gives
\begin{eqnarray*}
	 m^2\Big(\comi z^{\kappa+\delta}\partial_x^{m}[(\partial_y \psi)\partial_z\eta], ~\comi z^{\kappa+\delta} g_m\Big)_{L^2} \leq  B_1+B_2+B_3+B_4
	\end{eqnarray*}
	with
	\begin{eqnarray*}
	B_1&= &m^2 \sum_{0\leq j \leq [m/2]}{{m}\choose j} \norm{\comi z^{\kappa+\delta} (\partial_y \partial_x^j\psi)\partial_x^{m-j}\eta}_{L^2} \norm{\comi z^{\kappa+\delta} \partial_z g_m}_{L^2},\\
	B_2&=&m^2 \sum_{0\leq j \leq [m/2]}{{m}\choose j} \norm{\comi z^{\kappa+\delta} (\partial_y \partial_z \partial_x^j\psi)\partial_x^{m-j}\eta}_{L^2} \norm{\comi z^{\kappa+\delta}   g_m}_{L^2},\\
	B_3&=&m^2 \sum_{0\leq j \leq [m/2]}{{m}\choose j} \norm{\comi z^{\kappa+\delta} (\partial_y  \partial_x^j\psi)\partial_x^{m-j}\eta}_{L^2} \norm{\comi z^{\kappa+\delta}   g_m}_{L^2},\\
	B_4&=&m^2 \sum_{[m/2]+1\leq j \leq m }{{m}\choose j} \norm{\comi z^{\kappa+\delta} (\partial_y \partial_x^j\psi)\partial_z\partial_x^{m-j}\eta}_{L^2}\norm{\comi z^{\kappa+\delta} g_m}_{L^2}.
\end{eqnarray*}
By  \eqref{etan}-\eqref{chi2est} and the Sobolev inequality \eqref{soblev+}, we have
\begin{eqnarray*}
	&&m \sum_{0\leq j \leq [m/2]}{{m}\choose j} \norm{\comi z^{\kappa+\delta} (\partial_y \partial_z \partial_x^j\psi)\partial_x^{m-j}\eta}_{L^2} \\
	&\leq & C m \sum_{j=2}^{  [m/2]} \frac{m!} {j!(m-j)!} \frac{[(j-2)!]^\sigma}{\rho^{j-1}} \abs{\vec a}_{\rho,\sigma}\frac{[(m-j-6)!]^\sigma}{\rho^{m-j-5} (m-j)} \abs{\vec a}_{\rho,\sigma} +\frac{C[(m-6)!]^{\sigma} }{\rho^{m-5}} \abs{\vec a}_{\rho,\sigma}^2\\
	&\leq & \frac{C \abs{\vec a}_{\rho,\sigma}^2  }{\rho^{m-5}} \sum_{j=2}^{  [m/2]} \frac{ m![(j-2)!]^{\sigma-1} [(m-j-6)!]^{\sigma-1}m} {j^2(m-j)^7}   +\frac{C[(m-6)!]^{\sigma} }{\rho^{m-5}} \abs{\vec a}_{\rho,\sigma}^2\\
	&\leq & \frac{C  [(m-6)!]^{\sigma}  }{\rho^{m-5}}  \abs{\vec a}_{\rho,\sigma}^2.
\end{eqnarray*}
Similarly,     
\begin{eqnarray*}
	  m \sum_{0\leq j \leq [m/2]}{{m}\choose j}  \norm{\comi z^{\kappa+\delta} (\partial_y \partial_x^j\psi)\partial_x^{m-j}\eta}_{L^2}  \leq    \frac{C [\inner{m-6}!]^{ \sigma}}{\rho^{ m-5}} \abs{\vec a}_{ \rho,\sigma}^2.
\end{eqnarray*}
Thus,  we combine the above estimates to have
for any $\eps>0,$ 
\begin{eqnarray*}
	B_1 +B_2+B_3\leq  \eps m^2\norm{\comi z^{\kappa+\delta} \partial_z g_m}_{L^2}^2+  \frac{C_\eps  [\inner{m-6}!]^{2\sigma}}{\rho^{2(m-5)}} \Big(\abs{\vec a}_{ \rho,\sigma}^3+\abs{\vec a}_{ \rho,\sigma}^4\Big).
	\end{eqnarray*}
Applying a similar argument  as for \eqref{a2},   we have
\begin{eqnarray*}
	&&m \bigg(\sum_{[m/2]+1\leq j \leq m-2 }+\sum_{m-1\leq j \leq m}\bigg){{m}\choose j} \norm{\comi z^{\kappa+\delta} (\partial_y \partial_x^j\psi)\partial_z\partial_x^{m-j}\eta}_{L^2}\\
	&\leq&   \frac{C [\inner{m-6}!]^{\sigma}}{\rho^{m-5}} \abs{\vec a}_{ \rho,\sigma}^2+   \frac{C  [\inner{m-6}!]^{\sigma}}{\rho^{m-5}}\frac{ \abs{\vec a}_{ \tilde\rho,\sigma}}{\tilde\rho-\rho}.
\end{eqnarray*}
Thus, \begin{eqnarray*}
 B_4 \leq \frac{C [\inner{m-6}!]^{2\sigma}}{\rho^{2(m-5)}} \abs{\vec a}_{ \rho,\sigma}^3+\frac{C [\inner{m-6}!]^{2\sigma}}{\rho^{2(m-5)}} \frac{ \abs{\vec a}_{ \tilde\rho,\sigma}^2}{\tilde\rho-\rho}.
\end{eqnarray*}
Combining the estimates on $B_1$-$B_4,$ we obtain the desired inequality
and complete the proof. 
\end{proof}

\begin{lemma}
	Let $\sigma\in[3/2,2].$  Then for any   $m\geq 7,$     any pair $(\rho,\tilde\rho)$ with $0<\rho<\tilde\rho<\rho_0\leq 1,$   and  any $\eps>0,$ we have  
		\begin{multline*}
		 -2m^2\sum_{1\leq j \leq m}{{m}\choose j}\Big(\comi z^{\kappa+\delta} [\big(\partial_x^j  u \big)\partial_xg_{m-j}+ \big(\partial_x^j v\big)\partial_yg_{m-j}+ \big(\partial_x^j w\big)\partial_zg_{m-j}] , ~\comi z^{\kappa+\delta} g_m\Big)_{L^2}\\
		 \leq    \eps m^2\norm{\comi z^{\kappa+\delta} \partial_z g_m}_{L^2}^2+  \frac{C_\eps  [\inner{m-6}!]^{2\sigma}}{\rho^{2(m-5)}} \inner{\abs{\vec a}_{ \rho,\sigma}^3+   \abs{\vec a}_{ \rho,\sigma}^4 +\frac{\abs{\vec a}_{\tilde \rho,\sigma}^2}{\tilde\rho-\rho}},
	\end{multline*}
	where $C_\eps$ is a constant depending on $\eps.$
\end{lemma}

\begin{proof}
Applying the argument used in the proof of  Lemma \ref{lemjm},  we have
 \begin{eqnarray*}
		&&-2m^2\sum_{1\leq j \leq m}{{m}\choose j}\Big(\comi z^{\kappa+\delta} [\big(\partial_x^j  u \big)\partial_xg_{m-j}+ \big(\partial_x^j v\big)\partial_yg_{m-j}] , ~\comi z^{\kappa+\delta} g_m\Big)_{L^2}\\
		&\leq&    \frac{C   [\inner{m-6}!]^{2\sigma}}{\rho^{2(m-5)}} \inner{\abs{\vec a}_{ \rho,\sigma}^3+ \frac{\abs{\vec a}_{\tilde \rho,\sigma}^2}{\tilde\rho-\rho}}.
	\end{eqnarray*}
Moreover, we write 
	\begin{eqnarray*}
		&&-2m^2\sum_{1\leq j \leq m}{{m}\choose j}\Big(\comi z^{\kappa+\delta}  \big(\partial_x^j w\big)\partial_zg_{m-j}, ~\comi z^{\kappa+\delta} g_m\Big)_{L^2}\\
		&=&-2m^2\bigg[\sum_{1\leq j \leq m-2}+\sum_{m-1\leq j \leq m}\bigg]{{m}\choose j}\Big(\comi z^{\kappa+\delta}  \big(\partial_x^j w\big)\partial_zg_{m-j}, ~\comi z^{\kappa+\delta} g_m\Big)_{L^2}.
	\end{eqnarray*}
	Direct calculation shows that, by   \eqref{condi}-\eqref{+condi1} and the representation \eqref{defR1} of $g,$ 
	\begin{eqnarray*}
		 -2m^2 \sum_{m-1\leq j \leq m}{{m}\choose j}\Big(\comi z^{\kappa+\delta}  \big(\partial_x^j w\big)\partial_zg_{m-j}, ~\comi z^{\kappa+\delta} g_m\Big)_{L^2} 
		 \leq     \frac{C   [\inner{m-6}!]^{2\sigma}}{\rho^{2(m-5)}}  \frac{\abs{\vec a}_{\tilde \rho,\sigma}^2}{\tilde\rho-\rho}. 
	\end{eqnarray*}
	Furthermore, following the approach used to estimate $B_1$-$B_3$ in Lemma \ref{lemgm} as well as the argument in Lemma \ref{lemjm}, we
	can obtain  
	\begin{eqnarray*}
	&&-2m^2\sum_{1\leq j \leq m-2}{{m}\choose j}\Big(\comi z^{\kappa+\delta}  \big(\partial_x^j w\big)\partial_zg_{m-j}, ~\comi z^{\kappa+\delta} g_m\Big)_{L^2}	\\
	&\leq& \eps m^2\norm{\comi z^{\kappa+\delta} \partial_z g_m}_{L^2}^2+  \frac{C_\eps  [\inner{m-6}!]^{2\sigma}}{\rho^{2(m-5)}} \inner{\abs{\vec a}_{ \rho,\sigma}^3+   \abs{\vec a}_{ \rho,\sigma}^4}.
	\end{eqnarray*}
	Thus,  combining these estimates yields
	\begin{eqnarray*}
		&&-2m^2\sum_{1\leq j \leq m}{{m}\choose j}\Big(\comi z^{\kappa+\delta}  \big(\partial_x^j w\big)\partial_zg_{m-j}, ~\comi z^{\kappa+\delta} g_m\Big)_{L^2}\\
		&\leq&  \eps m^2\norm{\comi z^{\kappa+\delta} \partial_z g_m}_{L^2}^2+  \frac{C_\eps  [\inner{m-6}!]^{2\sigma}}{\rho^{2(m-5)}} \inner{\abs{\vec a}_{ \rho,\sigma}^3+   \abs{\vec a}_{ \rho,\sigma}^4 +\frac{\abs{\vec a}_{\tilde \rho,\sigma}^2}{\tilde\rho-\rho}}.
	\end{eqnarray*}
Then the desired upper bound estimate follows, and it
completes the proof.  
\end{proof}

\begin{lemma}\label{lemepsis}
Let $\sigma\in[3/2,2].$ Then for any   $m\geq 7,$      any pair $(\rho,\tilde\rho)$ with $0<\rho<\tilde\rho<\rho_0\leq 1,$  and  any $\eps>0,$ we have  
	\begin{equation}\label{espsis}
	\begin{aligned}
		& -2m^2\Big(\comi z^{\kappa+\delta}\partial_x^{m}[ (\partial_y\eta)\partial_z \psi], ~\comi z^{\kappa+\delta} g_m\Big)_{L^2}\\
		\leq& ~   \eps m^2\norm{\comi z^{\kappa+\delta} \partial_z g_m}_{L^2}^2+  \frac{C_\eps  [\inner{m-6}!]^{2\sigma}}{\rho^{2(m-5)}} \inner{\abs{\vec a}_{ \rho,\sigma}^2+   \abs{\vec a}_{ \rho,\sigma}^4+ \frac{\abs{\vec a}_{\tilde \rho,\sigma}^2}{\tilde\rho-\rho}},
		\end{aligned}
	\end{equation}
	where $C_\eps$ is a constant depending on $\eps.$
\end{lemma} 

\begin{proof} By using \eqref{+condi1} and \eqref{etan}-\eqref{mixoneta}, 
	direct calculation yields
	\begin{eqnarray*}
	&&-2 m^2 \bigg(\sum_{0\leq j \leq1}+\sum_{m-2\leq j \leq m-1}\bigg){{m}\choose j} \Big(\comi z^{\kappa+\delta} (\partial_y \partial_x^j\eta)\partial_z\partial_x^{m-j}\psi , ~\comi z^{\kappa+\delta} g_m\Big)_{L^2}\\
	&  \leq & \frac{C[\inner{m-6}!]^{2\sigma}}{\rho^{2(m-5)}} \frac{\abs{\vec a}_{\tilde \rho,\sigma}^2}{\tilde\rho-\rho}.
	\end{eqnarray*}
Moreover,  following the argument used in Lemma \ref{lemgm},  we have
	\begin{eqnarray*}
	&&-2 m^2 \sum_{2\leq j \leq m-3}{{m}\choose j} \Big(\comi z^{\kappa+\delta} (\partial_y \partial_x^j\eta)\partial_z\partial_x^{m-j}\psi , ~\comi z^{\kappa+\delta} g_m\Big)_{L^2}  \leq \frac{C[\inner{m-6}!]^{2\sigma}}{\rho^{2(m-5)}}  \abs{\vec a}_{  \rho,\sigma}^3.
	\end{eqnarray*}
	Then it remains to show that, for any $\eps>0,$	
	\begin{equation}\label{esnix}
	\begin{aligned}
		&-2 m^2  \Big(\comi z^{\kappa+\delta} (\partial_y\partial_x^m  \eta)\partial_z\psi, ~\comi z^{\kappa+\delta} g_m\Big)_{L^2}\\
		\leq &\eps m^2\norm{\comi z^{\kappa+\delta} \partial_z g_m}_{L^2}^2+  \frac{C_\eps  [\inner{m-6}!]^{2\sigma}}{\rho^{2(m-5)}} \inner{\abs{\vec a}_{ \rho,\sigma}^2+   \abs{\vec a}_{ \rho,\sigma}^4+ \frac{\abs{\vec a}_{\tilde \rho,\sigma}^2}{\tilde\rho-\rho}}.
		\end{aligned}
	\end{equation}
For this, by the representation of $g $,  applying $\partial_z\partial_x^{m}$ to the first equation in \eqref{defR1} gives
\begin{eqnarray*}
	\partial_z  g_m &=& \inner{\partial_y\partial_x^{m}\eta} \psi+\sum_{0\leq j \leq m-1} {{m}\choose j}\inner{\partial_y\partial_x^j\eta}\partial_x^{m-j} \psi\\
	&&+ \partial_x^{m}[\inner{\partial_y  v} \partial_z \psi]-\partial_x^{m} [\inner{\partial_y \psi}  \eta] - \partial_x^{m} [\inner{\partial_y u} \partial_z\eta].
	\end{eqnarray*}
	Thus 	
	\begin{eqnarray*}
		-2 m^2  \Big(\comi z^{\kappa+\delta} (\partial_y\partial_x^m  \eta)\partial_z \psi , ~\comi z^{\kappa+\delta} g_m\Big)_{L^2}=\sum_{1\leq j\leq 5}S_j,
	\end{eqnarray*}
	where $S_j$ are given by,  using the notation   $\chi= \partial_z\psi/\psi,$ 
	\begin{eqnarray*}
		S_1&=& -2 m^2  \Big(\comi z^{\kappa+\delta} \chi  \partial_z g_m , ~\comi z^{\kappa+\delta} g_m\Big)_{L^2},\\
		S_2&=&  2 m^2 \sum_{0\leq j \leq m-1} {{m}\choose j}  \Big(\comi z^{\kappa+\delta}\chi\inner{\partial_y\partial_x^j\eta}\partial_x^{m-j} \psi, ~\comi z^{\kappa+\delta} g_m\Big)_{L^2},\\
		S_3&=&  2 m^2    \Big(\comi z^{\kappa+\delta}\chi \partial_x^{m}[\inner{\partial_y  v} \partial_z  \psi], ~\comi z^{\kappa+\delta} g_m\Big)_{L^2},\\
		S_4&=& - 2 m^2    \Big(\comi z^{\kappa+\delta} \chi   \partial_x^{m} [\inner{\partial_y \psi}  \eta], ~\comi z^{\kappa+\delta} g_m\Big)_{L^2},\\
		S_5&=& - 2 m^2    \Big(\comi z^{\kappa+\delta}\chi    \partial_x^{m} [\inner{\partial_y u} \partial_z\eta], ~\comi z^{\kappa+\delta} g_m\Big)_{L^2}.
	\end{eqnarray*}
	The estimation on  $S_3$-$S_5$ is similar to that for \eqref{estpsieta}, in fact,  we have 
	\begin{eqnarray*}
		S_3+S_4 +S_5 \leq    \eps m^2\norm{\comi z^{\kappa+\delta} \partial_z g_m}_{L^2}^2+  \frac{C_\eps  [\inner{m-6}!]^{2\sigma}}{\rho^{2(m-5)}} \inner{\abs{\vec a}_{ \rho,\sigma}^3+   \abs{\vec a}_{ \rho,\sigma}^4+ \frac{\abs{\vec a}_{\tilde \rho,\sigma}^2}{\tilde\rho-\rho}}.
	\end{eqnarray*}
	And following the argument in Lemma \ref{lemjm}, it gives
	\begin{eqnarray*}
		S_2\leq \frac{C  [\inner{m-6}!]^{2\sigma}}{\rho^{2(m-5)}} \inner{  \abs{\vec a}_{ \rho,\sigma}^3+ \frac{\abs{\vec a}_{\tilde \rho,\sigma}^2}{\tilde\rho-\rho}}.
	\end{eqnarray*}
	Finally,
	\begin{eqnarray*}
		S_1=m^2\Big( \big[\partial_z(\comi z^{2(\kappa+\delta)} \chi)\big]    g_m , ~ g_m\Big)_{L^2}\leq \frac{C  [\inner{m-6}!]^{2\sigma}}{\rho^{2(m-5)}}   \abs{\vec a}_{ \rho,\sigma}^2.
	\end{eqnarray*}
Combining the upper bounds for $S_1$-$S_5$ gives \eqref{esnix}, and then \eqref{espsis} follows.  The  proof is completed. 
\end{proof}

We are now ready to give
 
 \begin{proof}[Proof of Proposition \ref{propgalpha}] Observe $g_m|_{z=0}=0.$ 
We multiply  both sides of the equation \eqref{equforg}   by $m^2 \comi z^{2\inner{\kappa+\delta}} g_{m},$   then integrate over $\Omega,$ to have
 	\begin{eqnarray*}
 	&&	\frac{1}{2}\frac{d}{dt} m^2 \norm{\comi z^{\kappa+\delta} g_{m}}_{L^2}^2+ m^2 \norm{\comi z^{\kappa+\delta}  \partial_z  g_{m} }_{L^2}^2\\
 	&=&  m^2\inner{ \comi z^{\kappa+\delta}\mathcal K_{m},  ~\comi z^{\kappa+\delta}  g_{m}}_{L^2}+ m^2\inner{ w (\partial_z\comi z^{\kappa+\delta})g_{m}, ~ \comi z^{\kappa+\delta}  g_{m}}_{L^2}
 	\\&& +\frac{1}{2} m^2\inner{  \big(\partial_{z}^2\comi z^{2\inner{\kappa+\delta}}\big)  g_{m},  ~   g_{m}}_{L^2} 
 	\\
 	&\leq&  m^2\inner{ \comi z^{\kappa+\delta}\mathcal K_{m}, ~ \comi z^{\kappa+\delta}  g_{m}}_{L^2}+   C m^2\norm{   \comi z^{\kappa+\delta}  g_{m}}_{L^2}^2.
 	\end{eqnarray*}	
 	Furthermore, by Lemmas \ref{lemgm}-\ref{lemepsis}  and \eqref{chi2est}  we see the terms in the last line are bounded from above by   
 	\begin{eqnarray*}
 		  \eps m^2\norm{\comi z^{\kappa+\delta} \partial_z g_m}_{L^2}^2+  \frac{C_\eps  [\inner{m-6}!]^{2\sigma}}{\rho^{2(m-5)}} \inner{\abs{\vec a}_{ \rho,\sigma}^2+   \abs{\vec a}_{ \rho,\sigma}^4+ \frac{\abs{\vec a}_{\tilde \rho,\sigma}^2}{\tilde\rho-\rho}}
 		 	\end{eqnarray*} 
 	for any $\eps>0.$ Then choosing $\eps$ small enough gives 
 	\begin{eqnarray*}
 			 \frac{d}{dt} m^2 \norm{\comi z^{\kappa+\delta} g_{m}}_{L^2}^2+ m^2 \norm{\comi z^{\kappa+\delta}  \partial_z  g_{m} }_{L^2}^2\leq   \frac{C  [\inner{m-6}!]^{2\sigma}}{\rho^{2(m-5)}} \inner{\abs{\vec a}_{ \rho,\sigma}^2+   \abs{\vec a}_{ \rho,\sigma}^4+ \frac{\abs{\vec a}_{\tilde \rho,\sigma}^2}{\tilde\rho-\rho}},
 	\end{eqnarray*}
 	and
 	thus, integrating over $[0,t]$ yields
 	\begin{eqnarray*}
 			&&  \frac{\rho^{2(m-5)}}{ [\inner{m-6}!]^{2\sigma}}  m^2 \norm{\comi z^{\kappa+\delta} g_{m}(t)}_{L^2}^2 \\
 			&\leq & C\abs{\vec a_0}_{\rho,\sigma}^2+
 	C \inner{  \int_0^t   \inner{ \abs{\vec a(s)}_{ \rho,\sigma}^2+\abs{\vec a(s)}_{ \rho,\sigma}^4 }   \,ds+    \int_0^t  \frac{  \abs{\vec a(s)}_{ \tilde\rho,\sigma}^2}{\tilde\rho-\rho}\,ds}. 
 	\end{eqnarray*}
 	A similar argument applies to the estimation on the upper bound for  $\tilde g_m.$ As a result, it follows from \eqref{realp} that  
 	\begin{eqnarray*}
 		&&  \frac{\rho^{2( \abs\alpha-5)}}{ [( \abs\alpha-6)!]^{2\sigma}} \abs\alpha^2\norm{\comi z^{\kappa+\delta} g_\alpha(t)}_{L^2}^2 \\
 		  &\leq&  C\abs{\vec a_0}_{\rho,\sigma}^2+
 	C \inner{  \int_0^t   \inner{ \abs{\vec a(s)}_{ \rho,\sigma}^2+\abs{\vec a(s)}_{ \rho,\sigma}^4 }   \,ds+    \int_0^t  \frac{  \abs{\vec a(s)}_{ \tilde\rho,\sigma}^2}{\tilde\rho-\rho}\,ds}.
 	\end{eqnarray*}
 	We obtain the desired upper bound for $g_\alpha.$  The other terms $h_\alpha, \vec\theta_\alpha$ and $\vec\mu_\alpha$ can be handled similarly.  In fact, when applying a similar kind of cancellation  as  for $g_\alpha,$   we see  $h,\theta_j$ and $\mu_j$ satisfy  equations quite similar to \eqref{eqforg1},  where  $\partial_xw$ or $\partial_yw$ is not involved in the source term  due to the cancellation.  This enables us to 
 repeat the argument for treating $g_\alpha,$ to  obtain  the desired  estimate  for $h_\alpha ,\vec\theta_\alpha$ and $\vec\mu_\alpha.$  
  We omit the details for brevity. 
 	 Then we have completed the proof of Proposition \ref{propgalpha}. 
 	 	\end{proof}

 \section{Upper bound of $\tau_2 \partial_x^m\psi$ and $\tau_2 \partial_x^m\xi$ }

This section is devoted to handling  $\tau_2\partial_x^m\xi,$ which relies on the estimate on    $\tau_2\partial_x^m\psi.$  The main result can be stated as follows.

\begin{proposition}
	\label{prpnear}
Let $3/2\leq\sigma\leq 2$ and  $0<\rho_0\leq 1.$ 
Suppose  $(u,v)\in L^\infty\inner{[0, T];~X_{\rho_0,\sigma}}$  is the solution to the Prandtl system \eqref{prandtl} satisfying the conditions \eqref{condi}-\eqref{+condi1}.  
 Then
 	  for any  $ m\geq 7,$   any $t\in[0,T]$ and    any pair $\inner{\rho,\tilde\rho}$ with $0<\rho<\tilde\rho< \rho_0$,  we have
\begin{multline*}
    \frac{\rho^{2( m-6)}}{ [( m-7)!]^{2\sigma}}\Big(\norm{\tau_2\partial_x^m \psi(t)}_{L^2}^2+\norm{\tau_2\partial_x^m \xi(t)}_{L^2}^2 \Big) \\
    \leq   C\abs{\vec a_0}_{\rho,\sigma}^2+
 	C \inner{  \int_0^t   \inner{ \abs{\vec a(s)}_{ \rho,\sigma}^2+\abs{\vec a(s)}_{ \rho,\sigma}^4 }   \,ds+    \int_0^t  \frac{  \abs{\vec a(s)}_{ \tilde\rho,\sigma}^2}{\tilde\rho-\rho}\,ds},
\end{multline*}
recalling $\xi=\partial_z\psi$.
 \end{proposition}

We   will  follow the same cancellation method used in \cite{GM} to prove the above proposition.  
Note $\abs\xi>0$ on supp  $\tau_2$. We may assume $-\xi>0$   on supp  $\tau_2$ without loss of generality.  In view of \eqref{eqpsi} we see $(-\xi)^{-1/2}\tau_2\partial_x^m\psi$ solves the equation
\begin{eqnarray}\label{eqxi}
	 \inner{\partial_t  + u \partial_x  +v\partial_y +w\partial_z  -\partial _{z}^2} (-\xi)^{-1/2}\tau_2\partial_x^m  \psi-\tau_2\inner{\partial_x^m  w}(-\xi)^{1/2} =\Xi
\end{eqnarray}
with
\begin{multline*}
	\Xi =(-\xi)^{-1/2}w\tau_2'\partial_x^m\psi-(-\xi)^{-1/2}\tau_2''\partial_x^m\psi-2(-\xi)^{-1/2}\tau_2'\partial_x^m\xi\\
	+\Big[\inner{\partial_t  + u \partial_x  +v\partial_y +w\partial_z  -\partial _{z}^2} (-\xi)^{-1/2}\Big]\tau_2\partial_x^m  \psi-2\big[\partial_z(-\xi)^{-1/2}\big]\partial_z\inner{\tau_2\partial_x^m  \psi}\\
	+(-\xi)^{-{1\over2}}\tau_2\bigg[\partial_x^m g  -\sum_{j=1}^{m}{m\choose j} \big[\big(\partial_x^j u\big) \partial_x^{m-j+1}\psi+\big(\partial_x^j v\big)\partial_y\partial_x^{m-j}\psi\big]  -\sum_{j=1}^{m-1}{m\choose j} \big(\partial_x^j w\big) \partial_x^{m-j}\xi\bigg].
\end{multline*}
Following the argument for proving Lemma \ref{lemg1} and  Lemma \ref{lemjm}, we can compute that 
\begin{eqnarray*}
\abs{	 \inner{ \Xi,  ~(-\xi)^{-1/2}\tau_2\partial_x^m  \psi}_{L^2} }\leq  \frac{C [\inner{ m-7}!]^{2\sigma}}{ \rho^{2( m-6)}} \big(\abs{\vec a}_{ \rho,\sigma}^2+\abs{\vec a}_{ \rho,\sigma}^3\big)+ \frac{C [\inner{ m-7}!]^{2\sigma}}{ \rho^{2(m-6)}}\frac{\abs{\vec a}_{\tilde \rho,\sigma}^2}{\tilde\rho-\rho}.
\end{eqnarray*}
As a result,  multiplying both side of the equation \eqref{eqxi} by the factor $(-\xi)^{-1/2}\tau_2\partial_x^m\psi$ and then taking integration over $\Omega,$ we obtain
\begin{multline*} \frac{1}{2}\frac{d}{dt}\norm{(-\xi)^{-1/2}\tau_2\partial_x^m\psi}_{L^2}^2\\
 \leq    \frac{C [\inner{ m-7}!]^{2\sigma}}{ \rho^{2( m-6)}} \big(\abs{\vec a}_{ \rho,\sigma}^2+\abs{\vec a}_{ \rho,\sigma}^3\big)+ \frac{C [\inner{ m-7}!]^{2\sigma}}{ \rho^{2(m-6)}}\frac{\abs{\vec a}_{\tilde \rho,\sigma}^2}{\tilde\rho-\rho}+ \abs{\inner{\tau_2\partial_x^mw, ~\tau_2 \partial_x^m \psi}_{L^2}}.
\end{multline*}
As for the last term on the right side  we use integration by parts to get
\begin{multline*}
	\abs{\inner{\tau_2\partial_x^mw, ~\tau_2 \partial_x^m \psi}_{L^2}}\\
	\leq 2 \abs{\inner{\tau_2'\partial_x^mw, ~\tau_2 \partial_x^m u}_{L^2}}+\abs{\inner{\tau_2\partial_x^{m+1}u, ~\tau_2 \partial_x^m u}_{L^2}}+\abs{\inner{\tau_2\partial_x^{m}\partial_yv, ~\tau_2 \partial_x^m u}_{L^2}}\\
	=2 \abs{\inner{\tau_2'\partial_x^mw, ~\tau_2 \partial_x^m u}_{L^2}} +\abs{\inner{\tau_2\partial_x^{m}\partial_yv, ~\tau_2 \partial_x^m u}_{L^2}}.
\end{multline*}
Then
\begin{multline}\label{xim}
\frac{1}{2}\frac{d}{dt}\norm{(-\xi)^{-1/2}\tau_2\partial_x^m\psi}_{L^2}^2
 \leq    \frac{C [\inner{ m-7}!]^{2\sigma}}{ \rho^{2( m-6)}} \big(\abs{\vec a}_{ \rho,\sigma}^2+\abs{\vec a}_{ \rho,\sigma}^3\big)+ \frac{C [\inner{ m-7}!]^{2\sigma}}{ \rho^{2(m-6)}}\frac{\abs{\vec a}_{\tilde \rho,\sigma}^2}{\tilde\rho-\rho}\\
 + 2 \abs{\inner{\tau_2'\partial_x^mw, ~\tau_2 \partial_x^m u}_{L^2}} +\abs{\inner{\tau_2\partial_x^{m}\partial_yv, ~\tau_2 \partial_x^m u}_{L^2}}.
\end{multline}

The following two lemmas are devoted to estimating the last two terms on the right side of \eqref{xim}.

\begin{lemma}\label{lemwu1}
We have
	 \begin{multline*}
 	 \abs{\inner{\tau_2'\partial_x^mw, ~\tau_2 \partial_x^m u}_{L^2}} \\
 	 \leq  \frac{m^{2\sigma}}{\rho^2}\norm{\partial_zf_{m-1}}_{L^2}^2+ \frac{C  [\inner{ m-7}!]^{2\sigma}}{ \rho^{2( m-6)}}\Big(  \abs{\vec a}_{ \rho,\sigma}^2+ \abs{\vec a}_{ \rho,\sigma}^4\Big)+\frac{C [\inner{ m-7}!]^{2\sigma}}{ \rho^{2(m-6)}}\frac{\abs{\vec a}_{\tilde \rho,\sigma}^2}{\tilde\rho-\rho}.
 	  \end{multline*}
\end{lemma}

\begin{lemma}\label{lemwu2}
	We have
	\begin{multline*}
		\abs{\inner{\tau_2\partial_x^{m}\partial_yv, ~\tau_2 \partial_x^m u}_{L^2}}\\
		\leq  \frac{m^{2\sigma}}{\rho^2}\norm{\partial_zf_{m-1}}_{L^2}^2+ \frac{C  [\inner{ m-7}!]^{2\sigma}}{ \rho^{2( m-6)}}\Big(  \abs{\vec a}_{ \rho,\sigma}^2+ \abs{\vec a}_{ \rho,\sigma}^4\Big)+\frac{C [\inner{ m-7}!]^{2\sigma}}{ \rho^{2(m-6)}}\frac{\abs{\vec a}_{\tilde \rho,\sigma}^2}{\tilde\rho-\rho}.
	\end{multline*}
\end{lemma}

The proof of Lemmas \ref{lemwu1} and \ref{lemwu2} are postponed to the following   Subsections \ref{subsec61} and  \ref{subsec62}.  Now we continue to the proof of Proposition \ref{prpnear}.  Integrating both sides of \eqref{xim} over $[0,t]$ and using the estimates in Lemmas \ref{lemwu1} and \ref{lemwu2},  it follows that 
\begin{multline*}
	\norm{(-\xi)^{-1/2}\tau_2\partial_x^m\psi(t)}_{L^2}^2\leq \frac{2m^{2\sigma}}{\rho^2}\int_0^t \norm{\partial_zf_{m-1}(s)}_{L^2}^2ds\\
	+   \frac{C  [\inner{ m-7}!]^{2\sigma}}{ \rho^{2( m-6)}} \Big( \abs{\vec a_0}_{\rho,\sigma}^2+
   \int_0^t   \inner{ \abs{\vec a(s)}_{ \rho,\sigma}^2+\abs{\vec a(s)}_{ \rho,\sigma}^4 }   \,ds+    \int_0^t  \frac{  \abs{\vec a(s)}_{ \tilde\rho,\sigma}^2}{\tilde\rho-\rho}\,ds\Big)\\
   \leq  \frac{C  [\inner{ m-7}!]^{2\sigma}}{ \rho^{2( m-6)}} \Big( \abs{\vec a_0}_{\rho,\sigma}^2+
   \int_0^t   \inner{ \abs{\vec a(s)}_{ \rho,\sigma}^2+\abs{\vec a(s)}_{ \rho,\sigma}^4 }   \,ds+    \int_0^t  \frac{  \abs{\vec a(s)}_{ \tilde\rho,\sigma}^2}{\tilde\rho-\rho}\,ds\Big),
\end{multline*}
where in the last inequality we have  used  Proposition \ref{propfm}.  Observe $(-\xi)^{-1/2}$ has a strictly positive lower bound on supp $\tau_2$. The we obtain the desired upper for $ \tau_2\partial_x^m\psi:$ 
\begin{multline*}
	\norm{ \tau_2\partial_x^m\psi(t)}_{L^2}^2 
   \leq  \frac{C  [\inner{ m-7}!]^{2\sigma}}{ \rho^{2( m-6)}} \Big( \abs{\vec a_0}_{\rho,\sigma}^2+
   \int_0^t   \inner{ \abs{\vec a(s)}_{ \rho,\sigma}^2+\abs{\vec a(s)}_{ \rho,\sigma}^4 }   \,ds+    \int_0^t  \frac{  \abs{\vec a(s)}_{ \tilde\rho,\sigma}^2}{\tilde\rho-\rho}\,ds\Big).
\end{multline*}  
 It remains to handle  $\tau_2\partial_x^m\xi$ and we use \eqref{t2} and Proposition \ref{propfm} to compute
\begin{multline*}
	\norm{\tau_2\partial_x^m\xi(t)}_{L^2}^2 \leq \norm{q_m(t)}_{L^2}^2+C\norm{\tau_2\partial_x^m\psi(t)}_{L^2}^2\\
\leq	 \frac{C  [\inner{ m-7}!]^{2\sigma}}{ \rho^{2( m-6)}} \Big( \abs{\vec a_0}_{\rho,\sigma}^2+
   \int_0^t   \inner{ \abs{\vec a(s)}_{ \rho,\sigma}^2+\abs{\vec a(s)}_{ \rho,\sigma}^4 }   \,ds+    \int_0^t  \frac{  \abs{\vec a(s)}_{ \tilde\rho,\sigma}^2}{\tilde\rho-\rho}\,ds\Big).
\end{multline*}
Combining the above two inequalities we prove Proposition \ref{prpnear}.

So in order to complete the proof of Proposition \ref{prpnear}, it remains to prove Lemmas \ref{lemwu1} and \ref{lemwu2} giving in  the following two subsections.

\subsection{Proof of Lemma \ref{lemwu1}}\label{subsec61}   

To do so we first recall the crucial representations  of $\partial_x^m u$  in terms of  $\varphi_m$  (see \cite[Lemma 3]{GM}),  with $\varphi_m$ defined by 
\begin{equation}\label{varp}
\varphi_m =\Big(\vartheta\psi+ 1-\vartheta  \Big)\inner{\partial_x^m \psi-\frac{  \xi }{  \psi }\partial_x^m u}=\Big(\vartheta\psi+ 1-\vartheta  \Big)\psi\partial_z\Big(\frac{\partial_x^mu}{\psi}\Big),  
\end{equation}
where $m\geq 1$ and $\vartheta(z)\in C_0^\infty(\mathbb R)$ is a given function such that $\vartheta\equiv 1$ in $[0, 2].$      Then  
$\partial_x^m u$ can be represented as (see \cite[Lemma 3]{GM})
\begin{equation}\label{albe}
\begin{aligned}
	\partial_x^m u(t,x,y,z)= \alpha_m(t,x,y,z) +\psi(t,x,y,z) \beta_m(t,x,y)\mathbf{1}_{\{z>1\}} , 		\end{aligned}
\end{equation}
where $\mathbf{1}_{A}$ stands for the characteristic  function on a set $A,$ and $\alpha_m,\beta_m$ are defined as follows:
\begin{eqnarray*}
	\alpha_m(t,x,y,z)=\left\{
	\begin{aligned}
		& \psi(t,x,y,z) \int_0^z  \frac{\varphi_m}{\Big(\vartheta\psi+ 1-\vartheta  \Big) \psi  }\,dz,   \quad {\rm if}\ z<1,\\
		& \psi(t,x,y,z) \int_{2 }^z  \frac{\varphi_m}{\Big(\vartheta\psi+ 1-\vartheta  \Big) \psi }dz,  \quad {\rm if}\ z>1,
	\end{aligned}
	\right.
\end{eqnarray*}
recalling  $1$ is the only non-degenerate critical point  of $u,$   and
\begin{eqnarray*}
	\beta_m(t,x,y)=  \frac{\partial_x^m u(t,x,y, 2)}{\psi(t,x,y,2)}.
\end{eqnarray*}
Observe
\begin{eqnarray*}
\forall~\abs{z-1}\geq \epsilon\ \textrm{with} \ 0\leq z\leq 2, \quad	\frac{1}{\big|\big(\vartheta\psi+ 1-\vartheta  \big)\psi\big|}\leq C,
\end{eqnarray*}
and thus
\begin{eqnarray}\label{supal}
	\sup_{0\leq z\leq 1-\epsilon}\norm{\alpha_m(\cdot,z)}_{L^2\inner{\mathbb T^2}}+\sup_{1+\epsilon \leq z\leq 2}\norm{\alpha_m(\cdot,z)}_{L^2\inner{\mathbb T^2}}\leq C \norm{\varphi_m}_{L^2\inner{\mathbb T^2\times[0, 2]}}.
\end{eqnarray}
This implies
\begin{equation}\label{alm}
	\norm{\alpha_m}_{L^2\inner{\mathbb T^2\times  {\rm supp}\,\tau_2'}}\leq C \norm{\varphi_m}_{L^2\inner{\mathbb T^2\times[0, 2]}}.
\end{equation}
The following estimates on   $\beta_m$ are  obtained by \cite[Lemma 6]{GM}:  
\begin{eqnarray}\label{ene}
\left\{
\begin{aligned}
	&  \norm{\beta_m}_{L^2\inner{\mathbb T^2}}\leq C \norm{\partial_x^m\psi}_{L^2\inner{\mathbb T^2\times\mathbb R_+}}\\
	&\norm{\partial_y\beta_m}_{L^2\inner{\mathbb T^2}}\leq C \norm{\partial_x^m\partial_y\psi}_{L^2\inner{\mathbb T^2\times\mathbb R_+}} + C \norm{\partial_x^m \psi}_{L^2\inner{\mathbb T^2\times\mathbb R_+}}.
	\end{aligned}
	\right.
\end{eqnarray}

\begin{lemma}\label{lemvar}
	Let $\varphi_m, G_m$ and $\theta_3$ be defined respectively by \eqref{varp}, \eqref{Gaal} and \eqref{sor}. Then 
	\begin{multline*}
		\norm{\varphi_m}_{L^2\inner{\mathbb T^2\times[0, 2]}}=\norm{G_m}_{L^2}\\
		\leq  \norm{\partial_{x}^{m-1}\theta_3}_{L^2}+C\norm{\partial_z f_{m-1}} +C m^{1-\sigma}\frac{ [\inner{ m-7}!]^{\sigma}}{ \tilde\rho^{ m-7}} \abs{\vec a}_{ \tilde \rho,\sigma}+C m^{2-2\sigma}\frac{ [\inner{ m-7}!]^{\sigma}}{ \rho^{ m-7}} \abs{\vec a}_{\rho,\sigma}^2.
	\end{multline*}
\end{lemma}

\begin{proof}
[Sketch of the proof]	 Observe $\varphi_m=-G_m$ on $\mathbb T^2\times[0,2]$ and 
\begin{eqnarray*}
	G_m=\partial_x^{m-1}\theta_3-\sum_{j=1}^{m-1}{{m-1}\choose j} \Big[\inner{\partial_x^j\xi}\partial_x^{m-j}u-\inner{\partial_x^j\psi}\partial_x^{m-j}\psi\Big].
\end{eqnarray*}
Furthermore, direct computation shows
\begin{multline*}
 	\sum_{j=1}^{m-1}{{m-1}\choose j} \norm{\inner{\partial_x^j\psi}\partial_x^{m-j}\psi}_{L^2} \leq m \norm{\inner{\partial_x\psi}\partial_x^{m-1}\psi}_{L^2}+ \sum_{j=2}^{m-1}{{m-1}\choose j}  \norm{\inner{\partial_x^j\psi}\partial_x^{m-j}\psi}_{L^2}  	\\
 	\leq C m^{1-\sigma}\frac{ [\inner{ m-7}!]^{\sigma}}{ \tilde\rho^{ m-7}} \abs{\vec a}_{ \tilde \rho,\sigma}+C m^{2-2\sigma}\frac{ [\inner{ m-7}!]^{\sigma}}{ \rho^{ m-7}} \abs{\vec a}_{\rho,\sigma}^2.
\end{multline*}
On the other hand, 
\begin{multline*}
 	\sum_{j=1}^{m-1}{{m-1}\choose j} \norm{\inner{\partial_x^j\xi}\partial_x^{m-j}u}_{L^2} \\
 	\leq \norm{\inner{\partial_x^{m-1}\xi}\partial_x u}_{L^2}+m\norm{\inner{\partial_x^{m-2}\xi}\partial_x^{2}u}_{L^2}+ \sum_{j=1}^{m-3}{{m-1}\choose j}  \norm{\inner{\partial_x^j\xi}\partial_x^{m-j}u}_{L^2} 	\\
 	\leq C \norm{  \partial_z f_{m-1}}_{L^2} +  C m^{1-\sigma}\frac{ [\inner{ m-7}!]^{\sigma}}{ \tilde\rho^{ m-7}} \abs{\vec a}_{ \tilde \rho,\sigma}+C m^{2-2\sigma}\frac{ [\inner{ m-7}!]^{\sigma}}{ \rho^{ m-7}} \abs{\vec a}_{\rho,\sigma}^2,
\end{multline*}
the last inequality following from \eqref{emix} and \eqref{fm1} as well as \eqref{etan}.
 Combining the above inequalities gives the desired estimate and we have completed the proof of Lemma \ref{lemvar}.
\end{proof}

\begin{lemma}
\label{lemgm++} Let $z_0\in[0, 1]$ be a given number and let $p_0\in L^\infty(\Omega).$    Then we have
\begin{eqnarray*}
	\big\|\int_{z_0}^z p_0 \alpha_m dz\big\|_{L^2\inner{\mathbb T^2\times[1+\frac{3}{2}\epsilon, 2]}}\leq C \norm{\varphi_m}_{L^2\inner{\mathbb T^2\times[0, 2]}},
\end{eqnarray*}
 where the constant $C$ depends   on  $\norm{p_0}_{L^\infty}.$
\end{lemma}  

\begin{proof}
	We will follow the proof of  \cite[Lemma 6]{GM} with slight modification.   Since $p_0\in L^\infty(\Omega)$ then for any $z\in[1+\frac{3}{2}\epsilon, 2], $
		\begin{multline*}
		\Big|\int_{z_0}^z p_0  \alpha_m   d\tilde z\Big| \leq C\int_{0}^z  \abs{\alpha_m}   d\tilde z \leq C \int_{0}^1  \abs\psi    \Big(\int_0^{\tilde z} \Big|  \frac{\varphi_m}{(\vartheta\psi +  1-\vartheta ) \psi }\Big|\,d  z_*  \Big)  d\tilde z\\
		+C\int_{1}^z  \abs\psi  \Big(\int_{\tilde z}^{2} \Big|  \frac{\varphi_m}{(\vartheta\psi +  1-\vartheta ) \psi }\Big|\,d  z_*  \Big)  d\tilde z.
	\end{multline*}
	Observe $1$ is the only critical point of $u$ so that $\psi$ doesn't change sign in the interval $[0,1],$ saying $\psi\geq 0.$ Then  
	\begin{multline*}
		\int_{0}^1  \abs\psi    \Big(\int_0^{\tilde z} \Big|  \frac{\varphi_m}{(\vartheta\psi +  1-\vartheta ) \psi }\Big|\,d  z_*  \Big)  d\tilde z=\int_{0}^1   \psi   \Big(\int_0^{\tilde z} \Big|  \frac{\varphi_m}{(\vartheta\psi +  1-\vartheta ) \psi }\Big|\,d  z_*  \Big)  d\tilde z\\
	=-	\int_{0}^1    u    \Big|  \frac{\varphi_m}{(\vartheta\psi +  1-\vartheta ) \psi }\Big|   d\tilde z+u(\cdot,1) \int_0^{1} \Big|  \frac{\varphi_m}{(\vartheta\psi +  1-\vartheta ) \psi }\Big|\,d  z_* \\
	\leq \int_{0}^1      \Big|    \frac{u-u(\cdot, 1)}{\inner{\vartheta \psi + 1-\vartheta  } \psi }\Big|  \cdot\abs{\varphi_m} d\tilde z.
	\end{multline*}
	Similarly 
	\begin{multline*}
		\int_{1}^z  \abs\psi  \Big(\int_{\tilde z}^2 \Big|  \frac{\varphi_m}{(\vartheta\psi +  1-\vartheta ) \psi }\Big|\,d  z_*  \Big)  d\tilde z\\
	\leq  \int_{1}^z      \Big|    \frac{u-u(\cdot, 1)}{\inner{\vartheta \psi + 1-\vartheta  } \psi }\Big|  \cdot\abs{\varphi_m} d\tilde z+\abs{u(\cdot, z)-u(\cdot,1)} \int_z^{2}      \Big|    \frac{1}{\inner{\vartheta \psi + 1-\vartheta  } \psi }\Big|  \cdot\abs{\varphi_m} d\tilde z.
	\end{multline*}
	Thus combining the above inequalities gives,
	 for any $z\in[1+\frac{3}{2}\epsilon, 2], $
		\begin{multline*}
		\Big|\int_{b_0}^z p_0  \alpha_m   d\tilde z\Big|  
		\leq  C  \int_{0}^z      \Big|    \frac{u-u(\cdot, 1)}{\inner{\vartheta \psi + 1-\vartheta  } \psi }\Big|  \cdot\abs{\varphi_m} d\tilde z\\
		+C\abs{u(\cdot, z)-u(\cdot,1)} \int_z^{2}      \Big|    \frac{1}{\inner{\vartheta \psi + 1-\vartheta  } \psi }\Big|  \cdot\abs{\varphi_m} d\tilde z 
		\leq C \norm{\varphi_m}_{L^2\inner{[0, 2]}},
	\end{multline*}
	where the last inequality holds because 
	\begin{eqnarray*}
		   \Big|    \frac{u-u(\cdot, 1)}{\inner{\vartheta \psi + 1-\vartheta  } \psi }\Big| \leq C ~\textrm{for any} \ z\in[0, 2], 
	\end{eqnarray*}
	and
	\begin{eqnarray*}
		   \Big|    \frac{1}{\inner{\vartheta \psi + 1-\vartheta  } \psi }\Big| \leq C ~\textrm{for any} \ z\in[1+\frac{3}{2}\epsilon, 2].
	\end{eqnarray*}
As a result, the assertion follows. We complete the proof of Lemma  \ref{lemgm++}. 
\end{proof}

The rest of this subsection is devoted to 
\begin{proof}[Proof of Lemma \ref{lemwu1}] Note that $ \abs{\inner{\tau_2'\partial_x^mw, ~\tau_2 \partial_x^m u}_{L^2}}$  is bounded from above by the sum of the following two terms:
\begin{eqnarray*}
	 \abs{\inner{\tau_2'\int_{0}^z \partial_x^{m+1}u(\cdot, \tilde z)d\tilde z, ~\tau_2 \partial_x^m u}_{L^2}} ,\quad  \abs{\inner{\tau_2'\int_{0}^z \partial_x^{m}\partial_y  v(\cdot, \tilde z)d\tilde z, ~\tau_2 \partial_x^m u}_{L^2}}. 
\end{eqnarray*}
 We only need to handle the second term,  since the first one has been estimated in \cite{GM} (see the treatment of $E_j$ in \cite{GM}). For this, we claim
  \begin{equation}
 	\label{sf+}
 	\begin{aligned}
 &	 \Big|\Big(\tau_2'\int_{0}^z \partial_x^{m}\partial_y  v(\cdot, \tilde z)d\tilde z, ~\tau_2 \partial_x^m u\Big)_{L^2}\Big| \\
 	&\qquad   \leq   \frac{m^{2\sigma}}{\rho^2}\norm{\partial_zf_{m-1}}_{L^2}^2+ \frac{C  [\inner{ m-7}!]^{2\sigma}}{ \rho^{2( m-6)}}  \Big(\abs{\vec a}_{ \rho,\sigma}^2+\abs{\vec a}_{ \rho,\sigma}^4\Big)+ \frac{C [\inner{ m-7}!]^{2\sigma}}{ \rho^{2(m-6)}}\frac{\abs{\vec a}_{\tilde \rho,\sigma}^2}{\tilde\rho-\rho}.
 	 \end{aligned}
 	 	 \end{equation}
 	Observe supp $\tau_2'\subset \Omega_1\cup \Omega_2$ with  	 
 	   $\Omega_1=\mathbb T^2\times \big[1+\frac{3}{2}\epsilon, 1+2\epsilon\big] $ and $\Omega_2=\mathbb T^2\times \big[1-2\epsilon, 1-\frac{3}{2}\epsilon\big],$ and thus the desired estimate 
 	   \eqref{sf+} follows if we can prove that 
 	 	\begin{equation}
 	\label{+sf++}
 	\begin{aligned}
 &	\sum_{j=1}^2 \Big|\Big(\tau_2'\int_{0}^z \partial_x^{m}\partial_y  v(\cdot, \tilde z)d\tilde z, ~\tau_2 \partial_x^m u\Big)_{L^2(\Omega_j)}\Big| \\
 	&\qquad   \leq   \frac{m^{2\sigma}}{\rho^2}\norm{\partial_zf_{m-1}}_{L^2}^2+ \frac{C  [\inner{ m-7}!]^{2\sigma}}{ \rho^{2( m-6)}}  \Big(\abs{\vec a}_{ \rho,\sigma}^2+\abs{\vec a}_{ \rho,\sigma}^4\Big)+ \frac{C [\inner{ m-7}!]^{2\sigma}}{ \rho^{2(m-6)}}\frac{\abs{\vec a}_{\tilde \rho,\sigma}^2}{\tilde\rho-\rho}.
 	 \end{aligned}
 	 	 \end{equation} 
 	 In the following argument we  will focus on deriving the upper bound for the integration over $\Omega_1$,  and the estimate for $\Omega_2$ can be handled in the same way with simpler argument.      To do so  we write,  for $1+\frac{3}{2}\epsilon\leq z\leq 1+2\epsilon,$  
\begin{eqnarray*}
	\int_{0}^z \partial_x^{m}   v(\cdot, \tilde z)d\tilde z&=&\int_{0}^{1-2\epsilon} \partial_x^{m}   v(\cdot, \tilde z)d\tilde z+\int_{1-2\epsilon}^z \partial_x^{m}   v(\cdot, \tilde z)d\tilde z\\
	&=&\int_{0}^{1-2\epsilon} \Big[\frac{\Gamma_m}{\psi}+\frac{\eta}{\psi}\partial_x^m u\Big] d\tilde z+\int_{1-2\epsilon}^z \Big[\frac{H_m}{\xi}+\frac{\eta}{\xi}\partial_x^m \psi\Big]d\tilde z,
\end{eqnarray*}
where the last line holds because of \eqref{condi} and \eqref{Gaal}.   Moreover, for the last term above, 
\begin{eqnarray*}
	\int_{1-2\epsilon}^z  \frac{\eta}{\xi}\partial_x^m \psi  d\tilde z =-\int_{1-2\epsilon}^z \Big(\partial_z \frac{\eta}{\xi}\Big)\partial_x^m u d\tilde z+\frac{\eta}{\xi}\partial_x^mu-\Big(\frac{\eta}{\xi}\partial_x^mu\Big)\Big|_{z=1-2\epsilon}.
\end{eqnarray*}
Then
\begin{eqnarray*}
	 	\Big|\Big(\tau_2'\int_{0}^z \partial_x^{m}\partial_y  v(\cdot, \tilde z)d\tilde z, ~\tau_2 \partial_x^m u \Big)_{L^2(\Omega_1)}\Big|&= &	 \Big|\Big(\tau_2'\int_{0}^z \partial_x^{m} v(\cdot, \tilde z)d\tilde z, ~\tau_2 \partial_x^m \partial_y  u\Big)_{L^2(\Omega_1)}\Big|\\
	 	 &\leq  & \sum_{j=1}^3 R_j
\end{eqnarray*}
with
\begin{eqnarray*}
	R_1&=&\Big|\Big(\tau_2'\int_{0}^{1-2\epsilon}  \frac{\Gamma_m}{\psi} d\tilde z, ~\tau_2 \partial_x^m \partial_y  u\Big)_{L^2(\Omega_1)}\Big|+\Big|\Big(\tau_2'\int_{1-2\epsilon}^z  \frac{H_m}{\xi} d\tilde z, ~\tau_2 \partial_x^m \partial_y  u\Big)_{L^2(\Omega_1)}\Big|,\\
	R_2&=&\Big|\Big(\tau_2'\int_{0}^{1-2\epsilon}  \frac{\eta}{\psi} \partial_x^m u  d\tilde z, ~\tau_2 \partial_x^m \partial_y  u\Big)_{L^2(\Omega_1)}\Big|\\
	&&+\Big|\Big(\tau_2'\int_{1-2\epsilon}^z \Big(\partial_z \frac{\eta}{\xi}\Big)\partial_x^m u  d\tilde z, ~\tau_2 \partial_x^m \partial_y  u\Big)_{L^2(\Omega_1)}\Big|,\\
	R_3&=&\Big|\Big(\tau_2'\frac{\eta}{\xi}\partial_x^m u , ~\tau_2 \partial_x^m \partial_y  u\Big)_{L^2(\Omega_1)}\Big|+\Big|\Big(\Big(\frac{\eta}{\xi}\partial_x^mu\Big)\Big|_{z=1-2\epsilon}\tau_2'
, ~\tau_2 \partial_x^m \partial_y  u\Big)_{L^2(\Omega_1)}\Big|.
\end{eqnarray*}

\medskip
\noindent\underline{\it Estimate on $R_1$.}  
To estimate $R_1,$ we use the relationship
\begin{eqnarray*}
 		\Gamma_m=-\partial_x^{m-1}h
 		+\sum_{j=0}^{m-2}\frac{(m-1)!}{j!(m-1-j)!}\big[(\partial_{x}^{j+1} u)\partial_x^{m-1-j}\eta-(\partial_{x}^{j+1} v)\partial_x^{m-1-j}\psi\big]
 	\end{eqnarray*}
 due to the definition \eqref{defR1} and \eqref{Gaal} of $h$ and $\Gamma_m.$	 This yields 
 \begin{multline*}
 	\Big|\Big(\tau_2'\int_{0}^{1-2\epsilon}  \frac{\Gamma_m}{\psi} d\tilde z, ~\tau_2 \partial_x^m \partial_y  u\Big)_{L^2(\Omega_1)}\Big| \leq  C \norm{\partial_x^{m-1}h}_{L^2} \norm{\partial_x^m\partial_y u}_{L^2}\\
 	+C\sum_{j=0}^{m-2}{{m-1}\choose j}\Big|\Big(\tau_2'\int_{0}^{1-2\epsilon}  \frac{ (\partial_{x}^{j+1} v)\partial_x^{m-1-j}\psi}{\psi}   d\tilde z, ~\tau_2 \partial_x^m \partial_y  u\Big)_{L^2(\Omega_1)}\Big|\\
 	+C\sum_{j=0}^{m-2}{{m-1}\choose j}\Big|\Big(\tau_2'\int_{0}^{1-2\epsilon}  \frac{(\partial_{x}^{j+1} u)\partial_x^{m-1-j}\eta}{\psi}   d\tilde z, ~\tau_2 \partial_x^m \partial_y  u\Big)_{L^2(\Omega_1)}\Big|.
 	 \end{multline*}
 	 As for the first and second terms  on the right side,  we use  \eqref{etan}-\eqref{chi2est} and follow the argument for proving Lemmas \ref{lemg1} and \eqref{lemjm},  to conclude  with direct computation  that they are bounded from above by 
 	 \begin{eqnarray*}
 	\frac{C  [\inner{ m-7}!]^{2\sigma}}{ \rho^{2( m-6)}} \big(\abs{\vec a}_{ \rho,\sigma}^2+\abs{\vec a}_{ \rho,\sigma}^3\big)+ \frac{C [\inner{ m-7}!]^{2\sigma}}{ \rho^{2(m-6)}}\frac{\abs{\vec a}_{\tilde \rho,\sigma}^2}{\tilde\rho-\rho}.
 	 \end{eqnarray*}
 	By using additionally \eqref{chi2est},  we can see the last term is also controlled  by the above upper bound.  Then
\begin{multline*}
		\Big|\Big(\tau_2'\int_{0}^{1-2\epsilon}  \frac{\Gamma_m}{\psi} d\tilde z, ~\tau_2 \partial_x^m \partial_y  u\Big)_{L^2(\Omega_1)}\Big| \\
		\leq  \frac{C  [\inner{ m-7}!]^{2\sigma}}{ \rho^{2( m-6)}} \big(\abs{\vec a}_{ \rho,\sigma}^2+\abs{\vec a}_{ \rho,\sigma}^3\big)+ \frac{C [\inner{ m-7}!]^{2\sigma}}{ \rho^{2(m-6)}}\frac{\abs{\vec a}_{\tilde \rho,\sigma}^2}{\tilde\rho-\rho}.
		\end{multline*}
Similarly,   we have the upper bound of  
\begin{eqnarray*}
	\Big|\Big(\tau_2'\int_{1-2\epsilon}^z  \frac{H_m}{\xi} d\tilde z, ~\tau_2 \partial_x^m \partial_y  u\Big)_{L^2(\Omega_1)}\Big|.
\end{eqnarray*}
Thus
\begin{equation}
	\label{R1}
	R_1\leq  \frac{C  [\inner{ m-7}!]^{2\sigma}}{ \rho^{2( m-6)}} \big(\abs{\vec a}_{ \rho,\sigma}^2+\abs{\vec a}_{ \rho,\sigma}^3\big)+ \frac{C [\inner{ m-7}!]^{2\sigma}}{ \rho^{2(m-6)}}\frac{\abs{\vec a}_{\tilde \rho,\sigma}^2}{\tilde\rho-\rho}.
\end{equation}

\medskip
\noindent\underline{\it Estimate on $R_2$.} We only need to handle the second term in the representation of $R_2,$ since the upper bound for the first one can be derived in the same way.   Using the representation \eqref{albe},  we have
\begin{eqnarray*}
\begin{aligned}
	& \Big|\Big(\tau_2'\int_{1-2\epsilon}^z \Big(\partial_z \frac{\eta}{\xi}\Big)\partial_x^m u  d\tilde z, ~\tau_2 \partial_x^m \partial_y  u\Big)_{L^2(\Omega_1)}\Big|\\
	\leq &  \Big|\Big(\tau_2'\int_{1-2\epsilon}^z \Big(\partial_z \frac{\eta}{\xi}\Big)\partial_x^m   u  d\tilde z, ~\tau_2 \partial_y  \alpha_m \Big)_{L^2(\Omega_1)}\Big|\\
	&+  \Big|\Big(\tau_2'\int_{1-2\epsilon}^z \Big(\partial_z \frac{\eta}{\xi}\Big)\partial_x^m   u  d\tilde z, ~\tau_2\inner{\beta_m\partial_y\psi+\psi\partial_y\beta_m} \Big)_{L^2(\Omega_1)}\Big|\\
	\leq & C\norm{\beta_m}_{L^2\inner{\mathbb T^2}}\norm{\partial_x^mu}_{L^2}+ C 	\norm{\alpha_m}_{L^2\inner{\mathbb T^2\times  {\rm supp}\,\tau_2'}}\inner{\norm{\partial_x^m u}_{L^2}+\norm{\partial_x^m\partial_y u}_{L^2}} \\
	&+ \Big|\Big(\tau_2'\int_{1-2\epsilon}^z \Big(\partial_z \frac{\eta}{\xi}\Big)\partial_x^m   u  d\tilde z, ~\tau_2 \psi\partial_y\beta_m\Big)_{L^2(\Omega_1)}\Big|\\
	\leq & C\norm{\beta_m}_{L^2\inner{\mathbb T^2}}\norm{\partial_x^mu}_{L^2}+ C 	\norm{\varphi_m}_{L^2\inner{\mathbb T^2\times[0,2]}}\inner{\norm{\partial_x^m u}_{L^2}+\norm{\partial_x^m\partial_y u}_{L^2}} \\
	&+ \Big|\Big(\tau_2'\int_{1-2\epsilon}^z \Big(\partial_z \frac{\eta}{\xi}\Big)\partial_x^m   u  d\tilde z, ~\tau_2 \psi\partial_y\beta_m\Big)_{L^2(\Omega_1)}\Big|,
	\end{aligned}
\end{eqnarray*}
where in the last inequality we have used \eqref{alm}. 
As for the last term in the above inequality, we use \eqref{albe} again to get
\begin{eqnarray*}
	&&\Big|\Big(\tau_2'\int_{1-2\epsilon}^z \Big(\partial_z \frac{\eta}{\xi}\Big)\partial_x^m   u  d\tilde z, ~\tau_2 \psi\partial_y\beta_m\Big)_{L^2(\Omega_1)}\Big|\\
	&\leq &\Big|\Big(\tau_2'\int_{1-2\epsilon}^z \Big(\partial_z \frac{\eta}{\xi}\Big)\alpha_m d\tilde z, ~\tau_2 \psi\partial_y\beta_m\Big)_{L^2(\Omega_1)}\Big|\\
	&&+\Big|\Big(\tau_2'\beta_m \int_{1}^z \Big(\partial_z \frac{\eta}{\xi}\Big)\psi d\tilde z, ~\tau_2 \psi\partial_y\beta_m\Big)_{L^2(\Omega_1)}\Big|\\
	&\leq& C \norm{\partial_y\beta_m}_{L^2\inner{\mathbb T^2}}\big\|\int_{1-2\epsilon}^z  \Big(\partial_z \frac{\eta}{\xi}\Big)\alpha_m dz\big\|_{L^2\inner{\mathbb T^2\times[1+\frac{3}{2}\epsilon, 2]}}+ C \norm{\beta_m}_{L^2\inner{\mathbb T^2}}^2\\
	&\leq& C \norm{\partial_y\beta_m}_{L^2\inner{\mathbb T^2}}\norm{\varphi_m}_{L^2\inner{\mathbb T^2\times[0,2]}}+ C \norm{\beta_m}_{L^2\inner{\mathbb T^2}}^2,
\end{eqnarray*}
where in the second inequality we have used the fact that
\begin{multline*}
	\Big|\Big(\tau_2'\beta_m \int_{1}^z \Big(\partial_z \frac{\eta}{\xi}\Big)\psi d\tilde z, ~\tau_2 \psi\partial_y\beta_m\Big)_{L^2(\Omega_1)}\Big|
	=\Big|{1\over 2} \Big( \partial_y\Big[\tau_2' \tau_2 \psi \int_{1}^z \Big(\partial_z \frac{\eta}{\xi}\Big)\psi d\tilde z\Big], ~\beta_m^2\Big)_{L^2(\Omega_1)}\Big|\end{multline*}
and the last inequality follows from Lemma \ref{lemgm++}.   Thus combining the above
inequalities and observing the estimate \eqref{ene}, we obtain  
\begin{multline}\label{lasvar}
	\Big|\Big(\tau_2'\int_{1-2\epsilon}^z \Big(\partial_z \frac{\eta}{\xi}\Big)\partial_x^m u  d\tilde z, ~\tau_2 \partial_x^m \partial_y  u\Big)_{L^2(\Omega_1)}\Big|
	\leq  C \norm{\beta_m}_{L^2\inner{\mathbb T^2}}^2+C\norm{\beta_m}_{L^2\inner{\mathbb T^2}}\norm{\partial_x^mu}_{L^2}\\
	+C \norm{\varphi_m}_{L^2\inner{\mathbb T^2\times[0,2]}}\Big(\norm{\partial_y\beta_m}_{L^2\inner{\mathbb T^2}}+\norm{\partial_x^m u}_{L^2}+\norm{\partial_x^m\partial_y u}_{L^2}\Big)\\
	\leq \frac{C  [\inner{ m-7}!]^{2\sigma}}{ \rho^{2( m-6)}}  \abs{\vec a}_{ \rho,\sigma}^2 +\frac{C  [\inner{ m-6}!]^{\sigma}}{ \tilde\rho^{( m-5)}} \abs{\vec a}_{\tilde \rho,\sigma}  \norm{G_m}_{L^2},
\end{multline}
where in  the last inequality we have used the fact that   $\varphi_m=-G_m$ on $\mathbb T^2\times[0,2]$ with $G_m$ defined in \eqref{Gaal}.   As for the last term in the above inequality, 
we use Lemma \ref{lemvar} as well as \eqref{chi2est}  to obtain
 \begin{multline*}
 	\norm{G_m}_{L^2}
 	\leq m^{-1}\frac{  [\inner{ m-7}!]^{\sigma}}{ \tilde\rho^{( m-6)}}  \abs{\vec a}_{\tilde \rho,\sigma}+C \norm{\partial_zf_{m-1}}\\
 	+C  m^{1-\sigma}\frac{ [\inner{ m-7}!]^{\sigma}}{ \tilde\rho^{ m-7}} \abs{\vec a}_{ \tilde \rho,\sigma}+C m^{2-2\sigma}\frac{ [\inner{ m-7}!]^{\sigma}}{ \rho^{ m-7}} \abs{\vec a}_{\rho,\sigma}^2, 
 \end{multline*}
 and thus, recalling $3/2\leq\sigma\leq 2$ and $0<\rho<\tilde\rho\leq 1,$
 \begin{multline}\label{veca}
 	\frac{[\inner{ m-6}!]^{\sigma}}{ \tilde\rho^{( m-5)}} \abs{\vec a}_{\tilde \rho,\sigma}  \norm{G_m}_{L^2} 
 	\leq  C\frac{m+m^{\sigma-1}}{\tilde \rho} \frac{   [\inner{ m-7}!]^{2\sigma}}{ \tilde\rho^{2( m-6)}}\abs{\vec a}_{\tilde \rho,\sigma} ^2\\ 
 	+C\frac{m^{2-\sigma}\rho}{\tilde\rho}\frac{  [\inner{ m-7}!]^{2\sigma}}{ \tilde\rho^{ m-6}\rho^{m-6}} \abs{\vec a}_{\tilde \rho,\sigma} \abs{\vec a}_{ \rho,\sigma} ^2+C\frac{\rho^2}{m^{2\sigma}}\frac{   [\inner{ m-6}!]^{2\sigma}}{ \tilde\rho^{2( m-5)}} \abs{\vec a}_{\tilde \rho,\sigma}^2 +\frac{m^{2\sigma}}{\rho^2}\norm{\partial_zf_{m-1}}_{L^2}^2\\
 	\leq   C\frac{m}{\tilde \rho} \frac{   [\inner{ m-7}!]^{2\sigma}}{ \tilde\rho^{2( m-6)}}\abs{\vec a}_{\tilde \rho,\sigma} ^2+ \frac{C  [\inner{ m-7}!]^{2\sigma}}{ \rho^{2( m-6)}}  \abs{\vec a}_{ \rho,\sigma}^4 +\frac{m^{2\sigma}}{\rho^2}\norm{\partial_zf_{m-1}}_{L^2}^2\\
 	\leq \frac{m^{2\sigma}}{\rho^2}\norm{\partial_zf_{m-1}}_{L^2}^2+ \frac{C  [\inner{ m-7}!]^{2\sigma}}{ \rho^{2( m-6)}}  \abs{\vec a}_{ \rho,\sigma}^4+ \frac{C [\inner{ m-7}!]^{2\sigma}}{ \rho^{2(m-6)}}\frac{\abs{\vec a}_{\tilde \rho,\sigma}^2}{\tilde\rho-\rho}.
 \end{multline}
 As a result,  combining the above inequalities \eqref{lasvar} and \eqref{veca} yields 
 \begin{multline*}
 	\Big|\Big(\tau_2'\int_{1-2\epsilon}^z \Big(\partial_z \frac{\eta}{\xi}\Big)\partial_x^m u  d\tilde z, ~\tau_2 \partial_x^m \partial_y  u\Big)_{L^2(\Omega_1)}\Big|\\
\leq  \frac{m^{2\sigma}}{\rho^2}\norm{\partial_zf_{m-1}}_{L^2}^2+ \frac{C  [\inner{ m-7}!]^{2\sigma}}{ \rho^{2( m-6)}}  \Big(\abs{\vec a}_{ \rho,\sigma}^2+\abs{\vec a}_{ \rho,\sigma}^4\Big)+ \frac{C [\inner{ m-7}!]^{2\sigma}}{ \rho^{2(m-6)}}\frac{\abs{\vec a}_{\tilde \rho,\sigma}^2}{\tilde\rho-\rho}.
 \end{multline*}
 Similar estimate holds  for the other term in the representation of $R_2.$ 
Thus  
 \begin{equation}
 	\label{r2}
 	R_2\leq  \frac{m^{2\sigma}}{\rho^2}\norm{\partial_zf_{m-1}}_{L^2}^2+ \frac{C  [\inner{ m-7}!]^{2\sigma}}{ \rho^{2( m-6)}}  \Big(\abs{\vec a}_{ \rho,\sigma}^2+\abs{\vec a}_{ \rho,\sigma}^4\Big)+ \frac{C [\inner{ m-7}!]^{2\sigma}}{ \rho^{2(m-6)}}\frac{\abs{\vec a}_{\tilde \rho,\sigma}^2}{\tilde\rho-\rho}.
 \end{equation}
 
 \medskip
\noindent\underline{\it Estimate on $R_3$.}  It is clear that, using integration by parts, 
\begin{eqnarray*}
	\Big|\Big(\tau_2'\frac{\eta}{\xi}\partial_x^m u , ~\tau_2 \partial_x^m \partial_y  u\Big)_{L^2(\Omega_1)}\Big|\leq C\norm{\partial_x^m    u}_{L^2}^2\leq \frac{C  [\inner{ m-7}!]^{2\sigma}}{ \rho^{2( m-6)}}  \abs{\vec a}_{ \rho,\sigma}^2.
\end{eqnarray*}
Using again the representation \eqref{albe} of $\partial_x^m u,$ we have
\begin{multline*}
 	\Big|\Big(\Big(\frac{\eta}{\xi}\partial_x^mu\Big)\Big|_{z=1-2\epsilon}\tau_2'
, ~\tau_2 \partial_x^m \partial_y  u\Big)_{L^2(\Omega_1)}\Big|
\leq  \Big|\Big(\Big(\frac{\eta}{\xi}\alpha_m\Big)\Big|_{z=1-2\epsilon}\tau_2'
, ~\tau_2 \partial_x^m \partial_y  u\Big)_{L^2(\Omega_1)}\Big| \\
\leq C \norm{\alpha_m(\cdot, 1-2\eps)}_{L^2(\mathbb T^2)}\norm{\partial_x^m \partial_y  u}_{L^2}\leq  C\norm{\varphi_m}_{L^2\inner{\mathbb T^2\times[0, 2]}}\frac{   [\inner{ m-6}!]^{\sigma}}{ \tilde\rho^{( m-5)}} \abs{\vec a}_{\tilde \rho,\sigma}, 
\end{multline*}
where the  last inequality follows from \eqref{supal} as well as \eqref{etan}.   This, along with the fact that  $\norm{\varphi_m}_{L^2\inner{\mathbb T^2\times[0, 2]}}=\norm{G_m}_{L^2}$ and the estimate \eqref{veca},  yields    
\begin{multline*}
	\Big|\Big(\Big(\frac{\eta}{\xi}\partial_x^mu\Big)\Big|_{z=1-2\epsilon}\tau_2'
, ~\tau_2 \partial_x^m \partial_y  u\Big)_{L^2(\Omega_1)}\Big|\\
\leq  \frac{m^{2\sigma}}{\rho^2}\norm{\partial_zf_{m-1}}_{L^2}^2+ \frac{C  [\inner{ m-7}!]^{2\sigma}}{ \rho^{2( m-6)}}  \abs{\vec a}_{ \rho,\sigma}^4+ \frac{C [\inner{ m-7}!]^{2\sigma}}{ \rho^{2(m-6)}}\frac{\abs{\vec a}_{\tilde \rho,\sigma}^2}{\tilde\rho-\rho}.
\end{multline*}
and thus
\begin{eqnarray*}
R_3\leq 	\frac{m^{2\sigma}}{\rho^2}\norm{\partial_zf_{m-1}}_{L^2}^2+ \frac{C  [\inner{ m-7}!]^{2\sigma}}{ \rho^{2( m-6)}}\Big(  \abs{\vec a}_{ \rho,\sigma}^2+\abs{\vec a}_{ \rho,\sigma}^4\Big)+ \frac{C [\inner{ m-7}!]^{2\sigma}}{ \rho^{2(m-6)}}\frac{\abs{\vec a}_{\tilde \rho,\sigma}^2}{\tilde\rho-\rho}.
\end{eqnarray*}
Now we combine the above inequality and  the estimates \eqref{R1} and \eqref{r2} on $R_1$ and $R_2$ to obtain the desired  
 upper bound for the integration over $\Omega_1$ on the left side of \eqref{+sf++}. The estimation
    for the integration over $\Omega_2$ is similar.  Then  
the  estimate \eqref{sf+} follows.  We then complete the proof of Lemma \ref{lemwu1}. 
\end{proof} 

\subsection{Proof of Lemma \ref{lemwu2}}\label{subsec62}
Observe $\abs\xi>0$ on supp $\tau_2$. Thus we use the representation \eqref{sor}  and \eqref{Gaal} of $\theta_1$  and $G_m,$ to write
\begin{eqnarray*}
	\abs{\inner{\tau_2\partial_x^{m}\partial_yv, ~\tau_2 \partial_x^m u}_{L^2}}&\leq&	\frac{1}{2}\abs{\inner{\xi^{-1}\tau_2\partial_x^{m}\theta_1 , ~\tau_2 \partial_x^m u}_{L^2}}+\abs{\inner{\xi^{-1}\tau_2\eta\partial_x^{m} \partial_y\psi , ~\tau_2 \partial_x^m u}_{L^2}}\\
	&&+ \sum_{j=1}^{m} {m\choose j}\abs{\inner{\xi^{-1}\tau_2 (\partial_x^{j}\xi)\partial_x^{m-j}\partial_y v, ~\tau_2 \partial_x^m u}_{L^2}}\\
	&&\quad+ \sum_{j=1}^{m} {m\choose j}\abs{\inner{\xi^{-1}\tau_2 (\partial_x^{j}\eta)\partial_x^{m-j}\partial_y \psi, ~\tau_2 \partial_x^m u}_{L^2}}.
\end{eqnarray*}
Then  Lemma \ref{lemwu2}  holds  if we can show that
\begin{equation}\label{vu1}
	\abs{\inner{\xi^{-1}\tau_2\partial_x^{m}\theta_1 , ~\tau_2 \partial_x^m u}_{L^2}}\leq \frac{C [\inner{ m-7}!]^{2\sigma}}{ \rho^{2( m-6)}}\frac{\abs{\vec a}_{\tilde \rho,\sigma}^2}{\tilde\rho-\rho},
\end{equation}
\begin{equation}
	\label{vu2}
	\begin{aligned}	
&	\abs{\inner{\xi^{-1}\tau_2\eta\partial_x^{m} \partial_y\psi , ~\tau_2 \partial_x^m u}_{L^2}}\\
	&\quad\leq  \frac{m^{2\sigma}}{\rho^2}\norm{\partial_zf_{m-1}}_{L^2}^2+ \frac{C  [\inner{ m-7}!]^{2\sigma}}{ \rho^{2( m-6)}}\Big(  \abs{\vec a}_{ \rho,\sigma}^2+ \abs{\vec a}_{ \rho,\sigma}^4\Big)+\frac{C [\inner{ m-7}!]^{2\sigma}}{ \rho^{2(m-6)}}\frac{\abs{\vec a}_{\tilde \rho,\sigma}^2}{\tilde\rho-\rho}
	\end{aligned}	
	\end{equation}
	and
	\begin{multline}\label{vu3}
		 \sum_{j=1}^{m} {m\choose j}\Big[\abs{\inner{\xi^{-1}\tau_2 (\partial_x^{j}\xi)\partial_x^{m-j}\partial_y v, ~\tau_2 \partial_x^m u}_{L^2}}+ \abs{\inner{\xi^{-1}\tau_2 (\partial_x^{j}\eta)\partial_x^{m-j}\partial_y \psi, ~\tau_2 \partial_x^m u}_{L^2}}\Big]\\
		\leq  \frac{C  [\inner{ m-7}!]^{2\sigma}}{ \rho^{2( m-6)}}   \abs{\vec a}_{ \rho,\sigma}^2 + \frac{C  [\inner{ m-7}!]^{2\sigma}}{  \rho^{2( m-6)}} \abs{\vec a}_{ \rho,\sigma}^2.
	\end{multline}

\begin{proof}
	[Proof of \eqref{vu1}]
	This is just a direct consequence of  \eqref{etan} and \eqref{chi2est}, using the same argument in the proof of Lemma \ref{lemg1}.
\end{proof}

\begin{proof}
	[Proof of \eqref{vu2}] By virtue of the representation  \eqref{Gaal} of  $G_m,$ we have, using \eqref{etan},
	\begin{multline*}
			\abs{\inner{\xi^{-1}\tau_2\eta\partial_x^{m} \partial_y\psi , ~\tau_2 \partial_x^m u}_{L^2}}\\ \leq 	\abs{\inner{\xi^{-1}\tau_2\eta\partial_x^{m} \partial_y\psi , ~ \xi^{-1}\tau_2 G_m}_{L^2}}+	\abs{\inner{\xi^{-1}\tau_2\eta\partial_x^{m} \partial_y\psi , ~ \xi^{-1} \tau_2\psi \partial_x^m \psi}_{L^2}}\\
			\leq C\norm{\partial_x^{m} \partial_y\psi}_{L^2} \norm{G_m}_{L^2}+C\norm{\partial_x^{m}  \psi}_{L^2} ^2\\
			\leq \frac{C  [\inner{ m-6}!]^{\sigma}}{ \tilde\rho^{( m-5)}} \abs{\vec a}_{\tilde \rho,\sigma}  \norm{G_m}_{L^2}+\frac{C  [\inner{ m-7}!]^{2\sigma}}{  \rho^{2( m-6)}} \abs{\vec a}_{ \rho,\sigma}^2,	\end{multline*}
which togethr with  \eqref{veca} give \eqref{vu2}. 
\end{proof}

\begin{proof}
	[Proof of \eqref{vu3}]  Following the argument in the proof of Lemma \ref{lemjm} and using
	 \eqref{etan}, we have 
	\begin{multline*} 
		 \sum_{j=1}^{m} {m\choose j}\abs{\inner{\xi^{-1}\tau_2 (\partial_x^{j}\xi)\partial_x^{m-j}\partial_y v, ~\tau_2 \partial_x^m u}_{L^2}}\leq  \sum_{j=1}^{m} {m\choose j}\norm{ (\tau_2\partial_x^{j}\xi)\partial_x^{m-j}\partial_y v}_{L^2}\norm{\partial_x^m u}_{L^2}\\
		 \leq  \frac{C  [\inner{ m-7}!]^{2\sigma}}{ \rho^{2( m-6)}}   \abs{\vec a}_{ \rho,\sigma}^2 +\frac{C [\inner{ m-7}!]^{2\sigma}}{ \rho^{2(m-6)}}\frac{\abs{\vec a}_{\tilde \rho,\sigma}^2}{\tilde\rho-\rho}.
	\end{multline*}
Similarly, using \eqref{chi2est} and  \eqref{etan} gives
 	\begin{multline*}
		 \sum_{j=1}^{m} {m\choose j}\abs{\inner{\xi^{-1}\tau_2 (\partial_x^{j}\eta)\partial_x^{m-j}\partial_y \psi, ~\tau_2 \partial_x^m u}_{L^2}}\\
		 \leq  \frac{C  [\inner{ m-7}!]^{2\sigma}}{ \rho^{2( m-6)}}   \abs{\vec a}_{ \rho,\sigma}^2 +\frac{C [\inner{ m-7}!]^{2\sigma}}{ \rho^{2(m-6)}}\frac{\abs{\vec a}_{\tilde \rho,\sigma}^2}{\tilde\rho-\rho}.
	\end{multline*}
	Combining these inequalities gives \eqref{vu3}.
	\end{proof}

 \section{Estimate on $\norm{\vec a}_{\rho,\sigma}$}\label{sec6}
 This section is about the estimate on    $\norm{\vec a}_{\rho,\sigma}$ with $\vec a=(u,v).$  We will handle the tangential derivatives and mixed derivatives   in Subsection \ref{subsec1} and Subsection \ref{subsec52}
 respectively.

 \subsection{Upper bound for the tangential derivatives}\label{subsec1} 
 For the derivatives in $x,y$ variables, we have the following
 
 \begin{proposition}\label{propuv}
 Let $3/2\leq\sigma\leq 2$ and $0<\rho_0\leq 1.$
 	Suppose  $(u,v)\in L^\infty\inner{[0, T];~X_{\rho_0,\sigma}}$  is the solution to the Prandtl system \eqref{prandtl} satisfying the conditions \eqref{condi}-\eqref{+condi1}.  
 Then 
 	  for any  $\alpha\in\mathbb Z_+^2$ with $ \abs\alpha\geq 7,$  any $t\in[0,T]$ and  
  any pair $\inner{\rho,\tilde\rho}$ with $0<\rho<\tilde\rho<\rho_0$,  we have
\begin{equation*}\label{taest}
\begin{aligned}
&   \frac{\rho^{2(  \abs\alpha-6)}}{ [( \abs\alpha-7)!]^{2\sigma}}\Big(\norm{\comi z^{\ell-1}\partial^\alpha u(t)}_{L^2}^2+\norm{\comi z^{\ell}\partial^\alpha \psi(t)}_{L^2}^2+\norm{\comi z^{\kappa}\partial^\alpha v(t)}_{L^2}^2 \Big) \\
  &\qquad+\frac{\rho^{2(  \abs\alpha-5)}}{ [( \abs\alpha-6)!]^{2\sigma}} 
  \abs\alpha^2\norm{\comi z^{\kappa+2}\partial^\alpha \eta(t)}_{L^2}^2\\
   \leq & ~C\abs{\vec a_0}_{\rho,\sigma}^2+
 	C \inner{  \int_0^t   \inner{ \abs{\vec a(s)}_{ \rho,\sigma}^2+\abs{\vec a(s)}_{ \rho,\sigma}^4 }   \,ds+    \int_0^t  \frac{  \abs{\vec a(s)}_{ \tilde\rho,\sigma}^2}{\tilde\rho-\rho}\,ds}. 
\end{aligned}
\end{equation*}
 \end{proposition}
 
 The proof of this proposition is based on the following two lemmas.

 \begin{lemma}[Estimates on  $\partial^\alpha u,\partial^\alpha \psi$ and  $\partial^\alpha v$] 
 \label{lemmixuvpsi} Under the same assumption as in Proposition \ref{propuv},   
 	  for any  $\alpha\in\mathbb Z_+^2$ with $ \abs\alpha\geq 7,$   any $t\in[0,T]$ and   any pair $\inner{\rho,\tilde\rho}$ with $0<\rho<\tilde\rho<\rho_0\leq 1,$ we have
   \begin{equation*}
\begin{aligned}
&   \frac{\rho^{2(  \abs\alpha-6)}}{ [( \abs\alpha-7)!]^{2\sigma}}\Big(\norm{\comi z^{\ell-1}\partial^\alpha u(t)}_{L^2}^2+\norm{\comi z^{\ell}\partial^\alpha \psi(t)}_{L^2}^2+\norm{\comi z^{\kappa}\partial^\alpha v(t)}_{L^2}^2 \Big)\\
   \leq & ~C\abs{\vec a_0}_{\rho,\sigma}^2+
 	C \inner{  \int_0^t   \inner{ \abs{\vec a(s)}_{ \rho,\sigma}^2+\abs{\vec a(s)}_{ \rho,\sigma}^4 }   \,ds+    \int_0^t  \frac{  \abs{\vec a(s)}_{ \tilde\rho,\sigma}^2}{\tilde\rho-\rho}\,ds}. 
\end{aligned}
\end{equation*}
 
\end{lemma}

\begin{proof}
	The estimates on  $\partial^\alpha u$ and $\partial^\alpha \psi$  can be obtained
	 in the same way as  \cite[Proposition 6.1]{LY}. In fact,  it follows from the Hardy type inequality that  
	\begin{multline*}
		 \norm{\tau_1\comi{z}^{\ell-1} \partial_x^m u}_{L^2}+	 \norm{\tau_1\comi{z}^{\ell} \partial_x^m \psi}_{L^2} \\
		 \leq C \norm{ \comi{z}^{\ell} f_m}_{L^2}+ C\big\|  \tau_1'\partial_x^m\psi-\tau_1'\frac{\partial_z\xi}{\psi}\partial_x^mu\big\|_{L^2}+C\norm{ \tau_2 \partial_x^m\psi}_{L^2}; 
	\end{multline*}
	see \cite[Lemma 6.2]{LY} for the detailed proof  of the above inequality.  
		Moreover,  just repeating the argument for estimating $\norm{ \comi{z}^{\ell} f_m}_{L^2}$ (see Proposition \ref{propfm}), we  have
	\begin{multline*}
    \frac{\rho^{2( m-6)}}{ [( m-7)!]^{2\sigma}}\big\|  \tau_1'\partial_x^m\psi-\tau_1'\frac{\partial_z\xi}{\psi}\partial_x^mu\big\|_{L^2}\\
    \leq   C\abs{\vec a_0}_{\rho,\sigma}^2+
 	C \inner{  \int_0^t   \inner{ \abs{\vec a(s)}_{ \rho,\sigma}^2+\abs{\vec a(s)}_{ \rho,\sigma}^4 }   \,ds+    \int_0^t  \frac{  \abs{\vec a(s)}_{ \tilde\rho,\sigma}^2}{\tilde\rho-\rho}\,ds}. 
\end{multline*}
This together with   Propositions \ref{propfm} and \ref{prpnear}  give
\begin{multline*}
	 \frac{\rho^{2( m-6)}}{ [( m-7)!]^{2\sigma}}\Big( \norm{\tau_1\comi{z}^{\ell-1} \partial_x^m u}_{L^2}+	 \norm{\tau_1\comi{z}^{\ell} \partial_x^m \psi}_{L^2}\Big)\\
	   \leq   C\abs{\vec a_0}_{\rho,\sigma}^2+
 	C \inner{  \int_0^t   \inner{ \abs{\vec a(s)}_{ \rho,\sigma}^2+\abs{\vec a(s)}_{ \rho,\sigma}^4 }   \,ds+    \int_0^t  \frac{  \abs{\vec a(s)}_{ \tilde\rho,\sigma}^2}{\tilde\rho-\rho}\,ds},
\end{multline*}
and thus,  using Proposition \ref{prpnear} and the fact that $1\leq \tau_1+\tau_2,$  we have 
\begin{multline}\label{l1}
	 \frac{\rho^{2( m-6)}}{ [( m-7)!]^{2\sigma}}\Big( \norm{\tau_1\comi{z}^{\ell-1} \partial_x^m u(t)}_{L^2}+	 \norm{\comi{z}^{\ell} \partial_x^m \psi(t)}_{L^2}\Big)\\
	   \leq   C\abs{\vec a_0}_{\rho,\sigma}^2+
 	C \inner{  \int_0^t   \inner{ \abs{\vec a(s)}_{ \rho,\sigma}^2+\abs{\vec a(s)}_{ \rho,\sigma}^4 }   \,ds+    \int_0^t  \frac{  \abs{\vec a(s)}_{ \tilde\rho,\sigma}^2}{\tilde\rho-\rho}\,ds}.
\end{multline}
	On the other hand, it follows from Poincar\'e inequality as well as the relationship in  \eqref{suppch2} that
	\begin{eqnarray*}
	\begin{aligned}
	&	\norm{\tau_2 \comi z^{\ell-1}\partial_x^m u(t)}_{L^2}\leq C\norm{\tau_2  \partial_x^m u(t)}_{L^2}\leq C(\norm{\tau_2'  \partial_x^m u(t)}_{L^2}+\norm{\tau_2 \partial_x^m \psi(t)}_{L^2} )\\
	&\quad	\leq C(\norm{\tau_1 \partial_x^m u(t)}_{L^2}+\norm{ \partial_x^m \psi(t)}_{L^2})\\
		&\qquad \leq    \frac{C [( m-7)!]^{2\sigma}}{\rho^{2( m-6)}} \Big( \abs{\vec a_0}_{\rho,\sigma}^2+
 	  \int_0^t   \inner{ \abs{\vec a(s)}_{ \rho,\sigma}^2+\abs{\vec a(s)}_{ \rho,\sigma}^4 }   \,ds+    \int_0^t  \frac{  \abs{\vec a(s)}_{ \tilde\rho,\sigma}^2}{\tilde\rho-\rho}\,ds\Big),
 	  \end{aligned},
	\end{eqnarray*}
	where in the last inequality we have used \eqref{l1}.
We then  combine the above two inequalities and observe the fact that $1\leq \tau_1+\tau_2,$  to obtain 
 \begin{equation}\label{ualpha}
 \begin{aligned}
 &	 \frac{\rho^{2( m-6)}}{ [( m-7)!]^{2\sigma}}\big(  \norm{ \comi z^{\ell-1} \partial_x^m u(t)}_{L^2}^2  +\norm{\comi z^{\ell}\partial_x^m\psi(t)}_{L^2}^2\big)\\
 \leq& ~ C\abs{\vec a_0}_{\rho,\sigma}^2+
 	C \inner{  \int_0^t   \inner{ \abs{\vec a(s)}_{ \rho,\sigma}^2+\abs{\vec a(s)}_{ \rho,\sigma}^4 }   \,ds+    \int_0^t  \frac{  \abs{\vec a(s)}_{ \tilde\rho,\sigma}^2}{\tilde\rho-\rho}\,ds}.  
 	\end{aligned}
 \end{equation}
 Similar estimates hold for the  weighted $L^2$-norms of  $\partial_y^m u$ and $\partial_y^m \psi.$  Then we  use \eqref{realp} with $\abs\alpha=m$ for the desired estimates on $\partial^\alpha u$ and $\partial^\alpha \psi.$   
 
 It remains to handle $\partial^\alpha v.$  And again it suffices to consider $\partial_x^m v$ and $\partial_y^m v$ with $m=\abs\alpha,$ due to \eqref{realp}.  In view of \eqref{Gaal}, we have 
 \begin{eqnarray*}
 	\partial_x^m v=
 	\left\{
 	\begin{aligned}
 	& \big(\Gamma_m+\eta \partial_x^m u\big)/ \psi,\quad  \textrm{if} \ z\in \textrm{supp}\,\tau_1,\\
 	& \big(H_m+\eta \partial_x^m \psi\big)/ \xi,\quad \textrm{if} \ z\in \textrm{supp}\,\tau_2,  
 	\end{aligned}
 	\right.	
 	 \end{eqnarray*}  
 and thus, recalling $ \psi\sim \comi z^{-\delta}$ on supp $\tau_1$,
 \begin{multline*}
 \norm{\comi z^{\kappa} \partial_x^m v}_{L^2}	\leq   \norm{\tau_1\comi z^{\kappa} \partial_x^m v}_{L^2} + C \norm{\tau_2 \partial_x^m v}_{L^2}\\
 \leq C \norm{\comi z^{\kappa+\delta} \Gamma_m}_{L^2}	+C\norm{\comi z^{\kappa+\delta} \eta\partial_x^m u   }_{L^2}+ C \norm{ H_m}_{L^2}	+C\norm{ \partial_x^m \psi}_{L^2}\\
 	\leq  C\norm{\comi z^{\kappa+\delta} \Gamma_m}_{L^2}	+C\norm{\comi z^{\ell-1} \partial_x^m u   }_{L^2}+C \norm{ H_m}_{L^2}	+C\norm{ \partial_x^m \psi}_{L^2},
 \end{multline*} 
 where in the last inequality we have used  \eqref{ellalpha} and \eqref{+condi1}. 
 This together with Proposition \ref{propfm} and  \eqref{ualpha},  yield the   upper bound for $\partial_x^m v,$ i.e.,
  \begin{eqnarray*}
 && \frac{\rho^{2( m-6)}}{ [( m-7)!]^{2\sigma}}  \norm{ \comi z^{\kappa} \partial_x^m v(t)}_{L^2}^2  \\
 &
  \leq&   C\abs{\vec a_0}_{\rho,\sigma}^2+
 	C \inner{  \int_0^t   \inner{ \abs{\vec a(s)}_{ \rho,\sigma}^2+\abs{\vec a(s)}_{ \rho,\sigma}^4 }   \,ds+    \int_0^t  \frac{  \abs{\vec a(s)}_{ \tilde\rho,\sigma}^2}{\tilde\rho-\rho}\,ds}.  
 \end{eqnarray*}
 The same argument applies to  $\partial_y^m v.$   Thus,  the desired bound for $\partial^\alpha v$  follows and  the proof is completed.  
\end{proof}

 \begin{lemma}
 	[Estimate on $\partial^\alpha \eta$] \label{lemetaes}
 	  Under the same assumption as in Proposition \ref{propuv},
 	  for any  $\alpha\in\mathbb Z_+^2$ with $ \abs\alpha\geq 7,$   any $t\in[0,T]$ and    any pair $\inner{\rho,\tilde\rho}$ with $0<\rho<\tilde\rho<\rho_0\leq 1,$ we have
 	\begin{eqnarray*}
 		&&   \frac{\rho^{2(  \abs\alpha-5)}}{ [( \abs\alpha-6)!]^{2\sigma}}\abs\alpha^2 \norm{\comi z^{\kappa+2}\partial^\alpha \eta(t)}_{L^2}^2\\
   &\leq & ~C\abs{\vec a_0}_{\rho,\sigma}^2+
 	C \inner{  \int_0^t   \inner{ \abs{\vec a(s)}_{ \rho,\sigma}^2+\abs{\vec a(s)}_{ \rho,\sigma}^4 }   \,ds+    \int_0^t  \frac{  \abs{\vec a(s)}_{ \tilde\rho,\sigma}^2}{\tilde\rho-\rho}\,ds}. 
 	\end{eqnarray*}
 \end{lemma}
 
 \begin{proof}
 As before,	it suffices to consider only $\partial_x^m \eta$ and $\partial_y^m \eta.$  We apply $\partial_x^m$ to the equation for $\eta$ in \eqref{equforpsi} to obtain
  \begin{eqnarray*}
 	&&\partial_t  \partial_x^m \eta+\big(u \partial_x  +v\partial_y +w\partial_z\big) \partial_x^m \eta-\partial _{z}^2\partial_x^m \eta \\
& =& \partial_x^m h-\sum_{1\leq j\leq m}{m\choose j} [\big(\partial_x^j u\big)\partial_x\partial_x^{m-j}\eta+\big(\partial_x^j v\big)\partial_y\partial_x^{m-j}\eta+\big(\partial_x^j w\big)\partial_z\partial_x^{m-j}\eta].
 \end{eqnarray*}
Multiplying both sides above by $m^2 \comi z^{2(\kappa+2)} \partial_x^m \eta,$  and then integrating over $\Omega,$ we have by noting $\partial_z\eta|_{z=0}=0,$
\begin{eqnarray*}
 	\frac{1}{2} \frac{d}{dt}  m^2\norm{\comi z^{\kappa+2} \partial_x^m \eta}_{L^2}^2+ m^2 \norm{\comi z^{\kappa+2}  \partial_z\partial_x^m \eta }_{L^2}^2=P_1+P_2+P_3
 	\end{eqnarray*}
 	with 
 	\begin{eqnarray*}
	P_1&=&  m^2\Big( \comi z^{\kappa+2} \partial_x^m h,  ~\comi z^{\kappa+2}  \partial_x^m \eta\Big)_{L^2}+ m^2\Big( w (\partial_z\comi z^{\kappa+2})\partial_x^m \eta, ~ \comi z^{\kappa+2} \partial_x^m \eta\Big)_{L^2}\\
	&&+\frac{1}{2} m^2\inner{  \big(\partial_{z}^2\comi z^{2(\kappa+2)}\big)  \partial_x^m \eta,  ~  \partial_x^m \eta}_{L^2}, \\
 P_2	&=& - m^2\sum_{1\leq j\leq m}{m\choose j}  \Big( \comi z^{\kappa+2} [\big(\partial_x^j u\big) \partial_x\partial_x^{m-j}\eta+\big(\partial_x^j v\big)\partial_y\partial_x^{m-j}\eta ],  ~\comi z^{\kappa+2}  \partial_x^m \eta\Big)_{L^2}, \\
 	P_3&=& - m^2\sum_{1\leq j\leq m}{m\choose j}  \Big( \comi z^{\kappa+2}  \big(\partial_x^j w\big)\partial_z\partial_x^{m-j}\eta,  ~\comi z^{\kappa+2}  \partial_x^m \eta\Big)_{L^2}.
 	\end{eqnarray*}
Furthermore, using \eqref{chi2est}-\eqref{mixoneta} we obtain by noting  $\kappa+2\leq \kappa+\delta$ due to \eqref{ellalpha}  and $\partial_x^m h=h_\gamma$ with $\gamma=(m,0),$
  \begin{eqnarray*}
 P_1\leq 	\frac{C[\inner{m-6}!]^{2\sigma}}{\rho^{2(m-5)}}  \abs{\vec a}_{\rho,\sigma}^2.
 \end{eqnarray*} 
We now estimate $P_3.$ By  \eqref{chi2est}, \eqref{mixoneta} and \eqref{+condi1}, we   have
 \begin{eqnarray*}
 &&- m^2\bigg(\sum_{1\leq j\leq 2}+\sum_{m-1\leq j\leq m}\bigg){m\choose j}  \Big( \comi z^{\kappa+2}  \big(\partial_x^j w\big)\partial_z\partial_x^{m-j}\eta,  ~\comi z^{\kappa+2}  \partial_x^m \eta\Big)_{L^2}\\
 &  \leq &  \frac{C[\inner{m-6}!]^{2\sigma}}{\rho^{2(m-5)}} \frac{\abs{\vec a}_{\tilde \rho,\sigma}^2}{\tilde\rho-\rho}.
 \end{eqnarray*}
 Moreover, for the terms in the middle of  the summation, it follows from the argument used  in Lemma \ref{lemjm} that   
 \begin{eqnarray*}
 	 - m^2\sum_{3\leq j\leq m-2}{m\choose j}  \Big( \comi z^{\kappa+2}  \big(\partial_x^j w\big)\partial_z\partial_x^{m-j}\eta,  ~\comi z^{\kappa+2}  \partial_x^m \eta\Big)_{L^2}\leq \frac{C[\inner{m-6}!]^{2\sigma}}{\rho^{2(m-5)}}  \abs{\vec a}_{  \rho,\sigma}^3.
 	  	  \end{eqnarray*}
 	  	  Thus, we obtain
 	  	  \begin{eqnarray*}
 	  	  	P_3\leq \frac{C[\inner{m-6}!]^{2\sigma}}{\rho^{2(m-5)}}\Big(\abs{\vec a}_{  \rho,\sigma}^3+ \frac{\abs{\vec a}_{\tilde \rho,\sigma}^2}{\tilde\rho-\rho}\Big).
 	  	  \end{eqnarray*}
 	  	  Similarly,   we have
 	  	  \begin{eqnarray*}
 	  	  	P_2\leq\frac{C[\inner{m-6}!]^{2\sigma}}{\rho^{2(m-5)}}\Big(\abs{\vec a}_{  \rho,\sigma}^3+ \frac{\abs{\vec a}_{\tilde \rho,\sigma}^2}{\tilde\rho-\rho}\Big).
 	  	  \end{eqnarray*}
 	  	  Combining these estimates gives
 	  	  \begin{eqnarray*}
 	  	  	\frac{1}{2} \frac{d}{dt} m^2 \norm{\comi z^{\kappa+2} \partial_x^m \eta}_{L^2}^2+ m^2 \norm{\comi z^{\kappa+2}  \partial_z\partial_x^m \eta }_{L^2}^2\leq   \frac{C[\inner{m-6}!]^{2\sigma}}{\rho^{2(m-5)}}\Big(\abs{\vec a}_{  \rho,\sigma}^2+\abs{\vec a}_{  \rho,\sigma}^3+ \frac{\abs{\vec a}_{\tilde \rho,\sigma}^2}{\tilde\rho-\rho}\Big).
 	  	  \end{eqnarray*}
 	  	  Integrating over  $[0,t]$ yields
 	  	  \begin{eqnarray*}
 	  && \frac	{\rho^{2(m-5)}}{ [\inner{m-6}!]^{2\sigma}}m^2 \norm{\comi z^{\kappa+2} \partial_x^m \eta(t)}_{L^2}^2\\
 	  &\leq & C\abs{\vec a_0}_{\rho,\sigma}^2+
 	C \inner{  \int_0^t   \inner{ \abs{\vec a(s)}_{ \rho,\sigma}^2+\abs{\vec a(s)}_{ \rho,\sigma}^4 }   \,ds+    \int_0^t  \frac{  \abs{\vec a(s)}_{ \tilde\rho,\sigma}^2}{\tilde\rho-\rho}\,ds}. 
 	  	  \end{eqnarray*}
 	  	   The upper bound of $\norm{\comi z^{\kappa+2} \partial_y^m \eta(t)}_{L^2}$ can be obtained similarly.
 	  	   Thus, the  estimate on $\partial^\alpha \eta$ follows
 	  	   and this completes the proof.
 \end{proof}
 
As an immediate consequence of Lemmas \ref{lemmixuvpsi}-\ref{lemetaes},   we 
have Proposition \ref{propuv}.   

 \subsection{Upper bound for the mixed derivatives }
 \label{subsec52}
 Now we estimate the mixed derivatives that  is stated in
 
  \begin{proposition}
 \label{norma}
 Let $3/2\leq\sigma\leq 2$ and $0<\rho_0\leq 1.$ Suppose   $(u,v)\in L^\infty\inner{[0, T];~X_{\rho_0,\sigma}}$  is the solution to the Prandtl system \eqref{prandtl} satisfying the conditions \eqref{condi}-\eqref{+condi1}.   Then we have, for any pair $(\alpha,j)\in\mathbb Z_+^2\times\mathbb Z_+$ with $1\leq j\leq 4$ and $ \abs\alpha+j\geq 7,$  any $t\in[0,T],$ and    any pair $(\rho,\tilde\rho)$  with  $0<\rho<\tilde\rho<\rho_0,$
\begin{eqnarray*}
 && \frac{\rho^{2(\abs\alpha+j-6)}}{[\inner{\abs\alpha+j-7}!]^{2\sigma}}\norm{\comi z^{\ell+1} \partial^\alpha \partial_z^j\psi(t)}_{L^2} ^2  + \frac{\rho^{2(\abs\alpha+j-5)}}{[\inner{\abs\alpha+j-6}!]^{2\sigma}} \abs\alpha^2 \norm{\comi z^{\kappa+2} \partial^\alpha \partial_z^j\eta(t)}_{L^2} ^2\\
 &\leq&     C \abs{\vec a_0}_{\rho,\sigma}^2+
 	  C \inner{  \int_0^t   \inner{ \abs{\vec a(s)}_{ \rho,\sigma}^2+\abs{\vec a(s)}_{ \rho,\sigma}^4 }   \,ds+    \int_0^t  \frac{  \abs{\vec a(s)}_{ \tilde\rho,\sigma}^2}{\tilde\rho-\rho}\,ds}. 
\end{eqnarray*}
\end{proposition}

\begin{proof}[Sketch of the proof] 
	Based on the above discussion, we will only sketch the proof for brevity.
	First of all, 
we need to consider  here the boundary conditions since the normal derivatives  are involved  when we use integration by parts. The situation is exactly the same as in 2D, where we need to use the equation \eqref{equforpsi} and the boundary conditions  
\begin{eqnarray*}
	\partial_z\psi|_{z=0}=  \partial_z\eta|_{z=0}=0,
\end{eqnarray*} 
  to  reduce the order of normal derivatives   in the boundary terms. Precisely,    
we use the equation for $\psi$ in \eqref{equforpsi} to obtain 
\begin{eqnarray*}
	\partial_z^3\psi  |_{z=0} =  \partial_z \big(\partial_t    \psi+ \inner{u \partial_x +v\partial_y +w\partial_z} \psi -g  \big)    |_{z=0}   =    \psi  (\partial_x\psi-\partial_y \eta)  |_{z=0}+2\eta \partial_y \psi |_{z=0},
\end{eqnarray*}
and  moreover,  direct computation yields
\begin{eqnarray*}
  &&\partial_z^5\psi  |_{z=0}= -\big(\partial_{z}^2 \psi \big)\inner{\partial_x\psi+\partial_y\eta}|_{z=0}+4 \psi \partial_x\partial_z^2\psi|_{z=0}+4\eta\partial_y\partial_z^2\psi|_{z=0}.\end{eqnarray*}
 Similar equalities hold for $\partial_z^3\eta|_{z=0}$ and $\partial_z^5\eta|_{z=0}.$ Then, we can 
follow the argument in \cite{LY} to deduce the energy estimates on $\partial^\alpha\partial_z^j \psi$ and $\partial^\alpha\partial_z^j\eta,$  with  only difference coming from the appearance of  two additional terms
\begin{equation}\label{twterm}
 \Big( \comi z^{\ell+1} \partial^\alpha\partial_z^j g,  ~\comi z^{\ell+1}  \partial^\alpha\partial_z^j  \psi\Big)_{L^2}~{\rm and}~~\, \abs\alpha^2\Big( \comi z^{\kappa+2} \partial^\alpha\partial_z^j h,  ~\comi z^{\kappa+2}  \partial^\alpha\partial_z^j  \eta\Big)_{L^2}.
\end{equation}
   However, there will be no additional difficulty  to  control the above
   two terms because  we can use \eqref{defR1} to write, by noting $1\leq j\leq 4,$
 $$\partial^\alpha\partial_z^{j}  g=\partial^\alpha\partial_z^{j-1}\Big[ (\partial_y \eta) \psi + (\partial_y v)\partial_z \psi -(\partial_y\psi) \eta-(\partial_xu) \partial_z\eta\Big].$$
Then following the argument in the proof of Lemma \ref{lemjm}  yields that  
\begin{eqnarray*}
	\norm{\comi z^{\ell+1}\partial^\alpha\partial_z^{j-1}  g}_{L^2}\leq  \frac{C[\inner{\abs\alpha+j-7}!]^{\sigma}}{\rho^{ \abs\alpha+j-6}} \abs{\vec a}_{  \rho,\sigma}^2+\frac{C[\inner{\abs\alpha+j-7}!]^{\sigma}}{\rho^{ \abs\alpha+j-6}} \frac{\abs{\vec a}_{  \tilde\rho,\sigma}}{\tilde\rho-\rho}.
\end{eqnarray*}
  Hence, we have  
\begin{eqnarray*}
 \Big( \comi z^{\ell+1} \partial^\alpha\partial_z^j g,  ~\comi z^{\ell+1}  \partial^\alpha\partial_z^j  \psi\Big)_{L^2} 
	 \leq   \frac{C[\inner{\abs\alpha+j-7}!]^{2\sigma}}{\rho^{2( \abs\alpha+j-6)}}\Big( \abs{\vec a}_{  \rho,\sigma}^3+\frac{\abs{\vec a}_{  \tilde\rho,\sigma}^2}{\tilde\rho-\rho}\Big). 
\end{eqnarray*}
Similarly, we can show that  the  second term in \eqref{twterm} is bounded from above by
\begin{eqnarray*}
 \frac{C[\inner{\abs\alpha+j-6}!]^{2\sigma}}{\rho^{ 2(\abs\alpha+j-5)}} \abs{\vec a}_{  \rho,\sigma}^3.
\end{eqnarray*}
And the other terms  can be estimated in the same way as   \cite{LY}, so
that we omit the detail for brevity,  cf. Subsection 6.2 in \cite{LY}.  Then the estimate given in Proposition \ref{norma} follows. 
\end{proof}

\section{Proof of a priori estimate}
\label{sec7}
In this section, we will prove the a priori estimate, in the two cases of  $\sigma\in[3/2,2]$  
 and $1<\sigma<3/2$ separately. For this, we
 will first prove  Theorem \ref{apriori} for the case when $\sigma\in[3/2,2]$.
 
 \bigskip
 \noindent\underline{\it The case when $\sigma\in[3/2,2]$}.  
By Propositions \ref{propfm},   \ref{propgalpha}, \ref{propuv} and \ref{norma},  we have the  upper bound on the  terms in $\abs{\vec a}_{\rho,\sigma}$  when the order of derivatives is greater than or equal to $7,$ that is, these terms are bounded by 
\begin{eqnarray*}
  C\abs{\vec a_0}_{\rho,\sigma}^2+C
 	  \int_0^t   \inner{ \abs{u(s)}_{ \rho,\sigma}^2+\abs{u(s)}_{ \rho,\sigma}^4 }   \,ds+C
 	  \int_0^t   \frac{\abs{u(s)}_{\tilde \rho,\sigma}^2}{\tilde\rho-\rho}   \,ds.
\end{eqnarray*}
For the derivatives with order less than $7$,  it can be checked
 straightforwardly
that the  same  upper bound holds.   Hence, we have the a priori  estimate   \eqref{esapr} and then complete the proof of   Theorem \ref{apriori}.

 \bigskip
 \noindent\underline{\it The case when $1<\sigma<3/2$}.     The proof
 for this case is similar to the one when $\sigma\in[3/2,2],$  just replacing the norm $\norm{\cdot}_{\rho,\sigma}$ by an equivalent norm.     Precisely,  if $1<\sigma<3/2$, then we can find an integer $N\geq 2$ such that 
\begin{equation}\label{nsigma}
	(N+1)/N\leq \sigma. 
\end{equation}
Define another Gevrey norm $\norm{\cdot}_{\rho,\sigma, N}$ by   replacing  respectively 
$
	\sup_{\abs\alpha\geq 7}
$ and  $
	\sup_{\abs\alpha\leq 6}
$  in \eqref{trinorm} 
by 
$
	\sup_{\abs\alpha\geq N+5}
$ and  $\sup_{\abs\alpha\leq N+4}.$  Do the same for  $\sup_{\abs\alpha+j\geq 7}$  and $\sup_{\abs\alpha+j\leq 6}.$ This new norm $\norm{\cdot}_{\rho,\sigma, N}$  is 
equivalent to$\norm{\cdot}_{\rho,\sigma}$ in the sense that 
\begin{eqnarray*}
	\norm{\cdot}_{\rho,\sigma} \leq 	\norm{\cdot}_{\rho,\sigma,N}\leq C_N\rho^{2-N} \norm{\cdot}_{\rho,\sigma} 
\end{eqnarray*}
for $\rho\leq 1,$  where $C_N$ is a constant depending only on $N.$  Moreover, following the     same  argument as above,  we  replace   $\abs{\cdot}_{\rho,\sigma}$ given in Definition \ref{gevspace} by
 $\abs{\cdot}_{\rho,\sigma,N},$ 
and replace \eqref{+condi1}  by 
\begin{equation}\label{++condi1}
\begin{aligned}
&\sum_{\abs\alpha\leq N+1}  \Big(\norm{\comi z^{\ell-1}\partial^\alpha u}_{L^\infty} +  \norm{\comi z^{\kappa}\partial^{\alpha} v}_{L^\infty} +  \norm{ \partial^{\alpha} w}_{L^\infty} \Big)\\
 &\qquad+\sum_{\abs\alpha\leq N+2}  \norm{\comi z^{\ell}\partial^{\alpha}  \psi}_{L_{x,y}^\infty(L_z^2)} +\sum_{\stackrel{ \abs\alpha+j\leq N+2}{j\geq 1}}   \norm{\comi z^{\ell+1}\partial^{\alpha}\partial_z^j \psi}_{L_{x,y}^\infty(L_z^2)}\\
 &\qquad\qquad\quad\qquad\qquad\quad\qquad+\sum_{  \abs\alpha+j\leq N+2} \norm{\comi z^{\kappa+2}\partial^{\alpha}\partial_z^j \eta}_{L_{x,y}^\infty(L_z^2)}\leq \tilde C.
\end{aligned}
\end{equation}

And then we have

\begin{theorem}\label{thmnewapr}
	Let $1<\sigma<3/2$ and  $0<\rho_0\leq 1.$  Suppose  $(u,v)\in L^\infty\inner{[0, T];~X_{\rho_0,\sigma}}$ is the solution to the Prandtl system \eqref{prandtl}    such that  the properties  in  \eqref{condi} and \eqref{++condi1} hold. 
 Then  there exists   a constant   $C_*>1,$   such that     
\begin{equation}\label{+esapr}
	\abs{\vec a (t)}_{\rho,\sigma,N}^2\leq C_* \abs{\vec a_0}_{\rho, \sigma,N}^2+C_* \int_{0}^t \inner{\abs{\vec a(s)}_{\rho,\sigma,N}^2+\abs{\vec a(s)}_{\rho,\sigma,N}^4} \,ds+C_* \int_{0}^t\frac{ \abs{\vec a(s)}_{\tilde\rho,\sigma,N}^2}{\tilde \rho-\rho}\,ds 
\end{equation}	 
 holds for any pair $(\rho,\tilde\rho)$ with   $0<\rho<\tilde \rho<\rho_0$ and any $t\in[0,T].$ 
\end{theorem}

\begin{proof}
	[Sketch of the proof] For brevity, we only give a sketch of the proof.
	In fact, it is similar to that  for  Theorem \ref{apriori}.  Here we will show how to modify the argument used there.
	
	We first recall how  the assumption $\sigma\in[ 3/2, 2]$ is used when  proving Theorem \ref{apriori}.  For $\sigma\geq 3/2$,  we use  the   following type of splitting in the summation
\begin{equation}\label{sum+}
	\Big[\sum_{0\leq j\leq 1}+\sum_{m-1\leq j\leq m }\Big] {m\choose j} \cdots+ \sum_{2\leq j\leq    m-2 }  {m\choose j} \cdots.
\end{equation} 
By direct computation, we see that the factors
$
	m$ and $ m^{2-\sigma}$
appear in the first two summations in \eqref{sum+} respectively. 
Then we can use \eqref{factor} and \eqref{+condi1} to conclude that the   summation
\begin{eqnarray*}
	\Big[\sum_{0\leq j\leq 1}+\sum_{m-1\leq j\leq m }\Big] {m\choose j} \cdots 
	\end{eqnarray*}
	is basically  bounded from above by 
	\begin{eqnarray*}
		\frac{ C[\inner{ m-7}!]^{\sigma}}{\rho^{ m-6}}\frac{\abs{\vec a}_{\tilde \rho,\sigma}}{\tilde\rho-\rho}\quad \textrm{or}  \quad \frac{ C[\inner{ m-6}!]^{\sigma}}{\rho^{ m-5}}\frac{\abs{\vec a}_{\tilde \rho,\sigma}}{\tilde\rho-\rho} .
	\end{eqnarray*}
	Meanwhile, for   the last summation in \eqref{sum+} we have  factors like
	\begin{eqnarray*}
		m^{3-2\sigma}, ~m^{4-3\sigma}, \cdots,
	\end{eqnarray*}
	which are less than 1 when $\sigma\geq  3/2.$ Thus,  the last summation in \eqref{sum+} is basically bounded from above  by 
 \begin{eqnarray*}
		\frac{ C[\inner{ m-7}!]^{\sigma}}{\rho^{ m-6}} \abs{\vec a}_{  \rho,\sigma}^2 \quad \textrm{or}  \quad \frac{ C[\inner{ m-6}!]^{\sigma}}{\rho^{ m-5}} \abs{\vec a}_{\tilde \rho,\sigma}^2.
	\end{eqnarray*}
	
Now we turn to the case when $1<\sigma<3/2,$ and instead of \eqref{sum+},  we can use a new splitting \begin{equation}\label{sum++}
	\Big[\sum_{0\leq j\leq N-1}+\sum_{m-N+1\leq j\leq m }\Big] {m\choose j} \cdots+ \sum_{N\leq j\leq  m-N }  {m\choose j} \cdots.
\end{equation}
Then we have  factors 
\begin{eqnarray*}
	m, m^{2-\sigma}, m^{3-2\sigma}, ~m^{4-3\sigma}, m^{N-(N-1)\sigma},
\end{eqnarray*}
appearing in the first two summations. Note that these factors are bounded by $m$, and then we can use again \eqref{factor} and \eqref{++condi1} to estimate the first two summations.   Meanwhile,  factors like
\begin{eqnarray*}
	m^{N+1-N\sigma}, ~m^{N+2-(N+1)\sigma}, \cdots,
\end{eqnarray*}
appear in the last summation in \eqref{sum++},  and these factors are less than 1 because of \eqref{nsigma}.   So the situation is  similar to
 the case  of   $\sigma\in[3/2,2].$    Then one can apply the argument used
  in Sections \ref{sec4}-\ref{sec6}  to the case when $1<\sigma<3/2$, with $\abs{\cdot}_{\rho,\sigma}$ and  \eqref{+condi1} therein replaced  by  $\abs{\cdot}_{\rho, \sigma, N}$ and \eqref{++condi1}.   Then the desired a priori estimate \eqref{+esapr} follows, and it completes the proof.  
	\end{proof}

 \section{Proof of the main results }
\label{sec8}

We will prove in this section the main results on the existence and uniqueness for Prandtl system.   We first prove Theorem \ref{maithm1} that corresponds to the constant outer flow
 when $(U,V)=(0,0),$ and then explain how to extend the argument to the  general outer flow.  Since the proof is  similar as in 2D case after we
 have the a priori estimate,  we will only  give  a sketch.

\begin{proof}
	[Sketch of the proof of Theorem \ref{maithm1}] 
	The proof relies on the a priori estimates given in Theorems \ref{apriori} and \ref{thmnewapr}. 
	
	In order to  obtain  the existence  of solutions to   the   Prandtl equations \eqref{prandtl},  there are  two main 
ingredients, one  of which is to investigate the existence of   approximate solutions to the regularized Prandtl system
\begin{equation}
\label{repradtl}
\left\{
\begin{aligned}
	&\partial_t u_\eps+\big(u_\eps \partial_x  +v_\eps\partial_y	 +w_\eps\partial_z\big)  u_\eps-\partial _{z}^2u_\eps-\eps\partial _{x}^2 u_\eps-\eps\partial _{y}^2 u_\eps=0,\\
 &\partial_t v_\eps+\big(u_\eps \partial_x  +v_\eps\partial_y	 +w_\eps\partial_z\big) v_\eps-\partial _{z}^2v_\eps-\eps\partial _{x}^2 v_\eps-\eps\partial _{y}^2 v_\eps=0,\\
 &(u_\eps,v_\eps)|_{z=0}=(0,0),\quad\lim_{z\rightarrow +\infty}   (u_\eps,v_\eps)=(0, 0),\\
 &(u_\eps,v_\eps)|_{t=0}=(u_0,v_0),
 \end{aligned}
 \right.
\end{equation}
with $w_\eps=-\int_0^z(\partial_x u_\eps+\partial_y v_\eps )d\tilde z.$       We remark that  the  regularized equations above share the same compatibility condition   \eqref{comcon} as the  original system  \eqref{prandtl}.  Another ingredient is to  derive a uniform estimate with respect to $\eps$ for the approximate solutions $(u_\eps, v_\eps).$  

  The existence for the parabolic system \eqref{repradtl} is standard.  Indeed, 
suppose that  $(u_0,v_0)\in X_{2\rho_0,\sigma}.$  Then we can  construct,   following  the 
  similar scheme as that in \cite[Section 7]{LY},  a solution  $(u_\eps,v_\eps)\in L^\infty\big([0,T_\eps^*];X_{3\rho_0/2, \sigma}\big)$ to \eqref{repradtl}   such that
\begin{equation}\label{unesfor uv}
	\sup_{t\in[0, T_\eps^*]}\norm{(u_\eps(t),v_\eps(t))}_{3\rho_0/2, \sigma}\leq C\norm{(u_0,v_0)}_{2\rho_0,\sigma},
\end{equation}  
where $T_\eps^*$ may depend on $\eps.$      

It remains to  derive a uniform estimate for the approximate solutions $(u_\eps, v_\eps),$  so that we can remove the $\eps$-dependence of the lifespan $T_\eps^*.$    
For this, we need to verify that
$(u_\eps,v_\eps)$  satisfies the condition \eqref{condi}, and the condition \eqref{+condi1} if $\sigma\in[3/2,2]$ and  \eqref{++condi1} if $1<\sigma<3/2$  respectively.

   In view of  \eqref{unesfor uv}  and by Sobolev inequalities and the definition of $\norm{\cdot}_{\rho,\sigma}$,  we know that  $u_\eps$ and $v_\eps$ satisfy the condition \eqref{+condi1} if $\sigma\in[3/2,2]$, and the condition \eqref{++condi1} in the case of $1<\sigma<3/2.$     In order to verify  \eqref{condi}, we suppose    
 $u_0$ satisfies Assumption \ref{maas} additionally and   show  that these properties therein preserve in  time.  This is clear when we consider a small perturbation around a shear flow.   For the general initial data,  it was shown in \cite{MW} by  using the classical  maximum principle  for  the parabolic equation,  that  such properties listed in Assumption \ref{maas}   also  preserve in time. Here, we can adopt the argument in \cite[Subsection 5.2]{MW}, and apply the maximum principle to the first equation in \eqref{repradtl}. 
 Precisely, applying $\partial_z$ to the first equation in \eqref{repradtl} and using the notations $\psi_\eps=\partial_zu_\eps$ and $\eta_\eps=\partial_zv_\eps,$, we have
 \begin{eqnarray*}
 	\partial_t \psi_\eps+\big(u_\eps \partial_x  +v_\eps\partial_y	 +w_\eps\partial_z\big)  \psi_\eps-\partial _{z}^2\psi_\eps-\eps\partial _{x}^2 \psi_\eps-\eps\partial _{y}^2 \psi_\eps=(\partial_y v_\eps)\psi_\eps- ( \partial_yu_\eps)\eta_\eps, 
 \end{eqnarray*}
that is, 
\begin{eqnarray*}
 	&&\Big[\partial_t  +\Big(u_\eps \partial_x  +v_\eps\partial_y	 +(w_\eps+2\delta(1+z)^{-1})\partial_z\Big)   -\partial _{z}^2 -\eps\partial _{x}^2 -\eps\partial _{y}^2  \Big]\big((1+ z)^{\delta}\psi_\eps\big)\\
 	&=&\Big(\partial_y v_\eps- ( \partial_yu_\eps)\eta_\eps/\psi_\eps + \delta w_\eps/(1+z)+\delta(\delta+1)/(1+z)^2\Big) (1+ z)^{\delta} \psi_\eps.
 \end{eqnarray*}
By using  \eqref{unesfor uv} and the maximal principle for parabolic equations (see \cite[Lemma E.2]{MW}), one can show that there exists
$c_*>0$ independent of $\eps$ such that  for any $(t,x,y,z)\in [0,T_\eps^*]\times\Omega,$
 \begin{eqnarray*}
     \psi_\eps(t,x,y,z) \geq  c_*  (1+ z)^{-\delta}.
 \end{eqnarray*}
This  gives   the  lower bound on $\partial_z u_\eps$  given in \eqref{condi}. Similarly, we can  derive also  the upper bounds on $\partial_z u_\eps$ and its normal derivatives given
  in \eqref{condi}.
  
 Consequently, following the argument in Sections \ref{sec4}-\ref{sec7},  the estimates \eqref{esapr} and  \eqref{+esapr} also hold,  with $\vec a$ there replaced by $(u_\eps, v_\eps).$   Finally,  by the uniform estimate
on $(u_\eps,v_\eps)$, we can repeat the argument in \cite[Section 8]{LY} to conclude the existence and uniqueness of solutions to the Prandtl system \eqref{prandtl}.  For brevity, we omit the detail here  and refer it to \cite[Section 5.2]{MW} and \cite[Section 8]{LY}  for  the comprehensive discussion. Thus, the proof of Theorem \ref{maithm1} is completed.    
\end{proof}

\begin{proof}
	[Sketch of the proof of Theorem \ref{maintheorem}]
Now we consider the general outer flow $U,V,p\in Y_{2\rho_0,\sigma}.$  The argument is similar to the case with the constant outer flow
discussed  above. The main difference comes from the appearance of extra source term and boundary term.  Since the extra source terms  only involve $U,V$ and $p$,  they are bounded by the Gevrey norms of $U,V$ and $p$. 

In addition, the boundary terms do not create additional difficulty in the
estimation. To see this, we consider for instance $\hat f_m$ which is defined in the same way as $f_m,$  that is,
\begin{eqnarray*}
\hat f_m=	\partial_x^{m} \psi-\frac{\partial_z \psi }{ \psi}\partial_x^{m}  (u-U)=\psi\partial_z\Big(\frac{\partial_x^m (u-U)}{\psi}\Big).
\end{eqnarray*}
Different from the case with  constant data $U$,   since  $\hat f_m\partial_z\hat f_m|_{z=0}\neq 0,$ then we have boundary terms like
\begin{eqnarray*}
	\int_{\mathbb T^2} \hat f_m(x,y,0) \partial_z\hat f_m(x,y,0)dxdy,
\end{eqnarray*}
when applying integration by parts.  Furthermore,  note that, denoting $\chi=\partial_z\psi/\psi,$  
\begin{eqnarray*}
	 \hat f_m(x,y,0) =\partial_x^{m} \psi(x,y,0)+\chi(x,y,0)\partial_x^mU(x,y),
\end{eqnarray*}
by using the fact that $\partial_z\psi|_{z=0}=\partial_xp,$ we have
\begin{eqnarray*}
	\partial_z \hat f_m(x,y,0) = \partial_x^{m+1} p(x,y)+(\partial_z\chi)(x,y,0)\partial_x^mU(x,y)-\chi(x,y,0)\partial_x^m\psi(x,y,0).
\end{eqnarray*}
Then
\begin{eqnarray*}
	\Big|\int_{\mathbb T^2} \hat f_m(x,y,0) \partial_z\hat f_m(x,y,0)dxdy\Big| \leq C\norm{\partial_x^m\psi (\cdot, 0)}_{L^2(\mathbb T^2)}^2+\textrm{ terms involving} ~U, V,p.
\end{eqnarray*}
As for the first term on the right side, we have
\begin{eqnarray*}
	 \norm{\partial_x^m\psi (\cdot, 0)}_{L^2(\mathbb T^2)}^2 &\leq&  \eps  \norm{ \partial_z\partial_x^m  \psi}_{L^2(\Omega)}^2+C_\eps  \norm{\partial_x^m\psi}_{L^2(\Omega)}^2\\
	& \leq&  \eps  \norm{ \partial_z \hat f_m}_{L^2(\Omega)}^2+C_\eps  \inner{\norm{\partial_x^m\psi}_{L^2(\Omega)}^2+\norm{\partial_x^m(u-U)}_{L^2(\Omega)}^2},
\end{eqnarray*}
where in the last inequality we have used the representation of $\hat f_m.$  We can choose $\eps$   small enough to   absorb  the  first  
term
$
	 \norm{ \partial_z f_m}_{L^2(\Omega)}^2.
$
Then the  quantity 
\begin{eqnarray*}
		\Big|\int_{\mathbb T^2} \hat f_m(x,y,0) \partial_z\hat f_m(x,y,0)dxdy\Big| 
\end{eqnarray*}
has a suitable upper bound,  just as $\norm{\partial_x^m\psi}_{L^2}^2$ and $\norm{\partial_x^m(u-U)}_{L^2}^2.$ We can apply a  similar argument as above to the other boundary terms when estimating  $g_\alpha, h_\alpha, \Gamma_m,$ etc.

     In summary, we can  extend the  result to  general
     outer flow $(U, V),$ with $(U,V)$ in  a  Gevrey space with the same index $\sigma$  as the initial data.    For  this,  we replace $(u,v)$   by $(u-U, v-V)$  in the estimation,  and perform  the  estimates on $(u-U, v-V)$  following the discussions  in Sections \ref{sec4}-\ref{sec7}.  Since it does not involve 
     any extra essential difficulty, we omit the detail for brevity.
  \end{proof}

 \bigskip
\noindent {\bf Acknowledgements.}
We would like to thank the reviewer's valuable suggestions for the revision of the paper.
   The research of the first author was supported by NSFC (11871054, 11771342) and Fok Ying Tung Education Foundation (151001). And the research of the second
author was supported by the General Research Fund of Hong Kong, CityU
 No.11320016 and the NSFC project 11731008.

\end{document}